\documentclass{article}
\usepackage{a4wide} 

\usepackage[utf8]{inputenc}
\usepackage{amsmath, amsthm}
\usepackage{amsfonts}
\usepackage{amssymb}
\usepackage{bussproofs}
\usepackage{tikz}
\usepackage{ stmaryrd }
\usepackage{graphicx}
\usepackage{forest}
\usepackage{xcolor}
\usepackage{mathrsfs}
\usepackage{enumitem}
\usepackage{booktabs}
\usepackage{comment}
\usepackage{cancel}
\usepackage{url}
\usepackage{mathrsfs}
\usepackage{lmodern}
\usepackage{thmtools,thm-restate}

	\newtheorem{definition}{Definition}
	
	\newtheorem{theorem}{Theorem}
	\newtheorem{lemma}{Lemma}
	\newtheorem{cor}{Corollary}
    \newtheorem{example}{Example}
	\newtheorem{prop}{Proposition}

\newcommand{\fillBox}{\hfill$\Box$}

\setlist[enumerate,1]{leftmargin=1cm}
\setlist[itemize,1]{leftmargin=1cm}

\title{Inferentialist Public Announcement Logic:\\ Base-extension Semantics}

\author{
\begin{tabular}{cc}
Timo Eckhardt &  David Pym \\
UCL \& Institute of Philosophy &  UCL \& Institute of Philosophy \\ 
University of London & University of London \\ 
t.eckhardt@ucl.ac.uk & david.pym@sas.ac.uk 
\end{tabular}
}

\date{} 

\begin{document}

\maketitle

\begin{abstract}
Proof-theoretic semantics, and base-extension semantics in particular, can be seen as a logical realization of inferentialism, in which the meaning of expressions is understood through their use. We present a base-extension semantics for public announcement logic, building on earlier 
work giving a base-extension semantics for the modal logic $S5$, which 
in turn builds on earlier such work for $K$, $KT$, $K4$, and $S4$. 
These analyses rely on a notion of `modal relation' on bases.  
The main difficulty in extending the existing B-eS for $S5$ to public announcement logic is to account announcements of the form $[\psi]\phi$, 
which, in this setting, update the modal relations on bases. 
We provide a detailed analysis of two classical examples, namely the three-player card game and the muddy children puzzle. These examples 
illustrate how the inferentialist perspective requires fully explicit 
information about the state of the participating agents. 
\end{abstract}

\medskip 


\section{Introduction} \label{sec:introduction}


Proof-theoretic semantics (P-tS) is a 
logical realization of inferentialism \cite{brandom2009}, within which the meaning of expressions is understood through their use. It is an approach to semantics based on \emph{proof} --- as opposed to truth --- where `proof' means `valid argument' --- as opposed to derivation in a fixed system.  This includes both the semantics \emph{of} proofs (i.e., validity conditions on `arguments'; e.g., \cite{Dummett1991,Prawitz1971,Prawitz2006,SchroederHeister2006,Schroeder2007modelvsproof} among much more) and semantics \emph{in terms of} proofs (i.e., the meaning of logical constants in terms of consequential relationships; e.g., \cite{Sandqvist2005,Sandqvist2009,Sandqvist2015,piecha2017definitional} among much more). We call the first \emph{proof-theoretic validity} (P-tV) and the second \emph{base-extension semantics} (B-eS). This nomenclature follows from traditions in the literature, but both branches concern \emph{validity} and make use of \emph{base-extensions} in their set-up. 

Here we are concerned with B-eS, which establishes a concept of validity for formulae 
that is grounded in atomic inference (see, for example, \cite{Sandqvist2015}). That is, in the simplest case, we take a \emph{base} $\mathscr{B}$ that is a (countable) set of atomic rules written in the form 
\[
    \frac{p_1 \ldots p_n}{p} \quad \mbox{or} \quad 
    p_1, \dots, p_n \Rightarrow p
\]
where $p$ and each $p_i$ is an atom (and $n$ may be $0$). We then define a notion of validity relative to 
$\mathscr{B}$ for formulae $\phi$, $\Vdash_{\mathscr{B}} \phi$, by an inductive definition in which the base case is of the form 
\[
   \begin{array}{rcl}
        \Vdash_{\mathscr{B}} p & \mbox{iff} & \mbox{$\vdash_{\mathscr{B}} p$} 
   \end{array}
\]
where $\vdash_{\mathscr{B}}$ denotes provability using the rules of $\mathscr{B}$. 
The meaning of complex formulae is then defined inductively in a familiar way; for example, 
\[
   \begin{array}{rcl}
        \Vdash_{\mathscr{B}} \phi \wedge \psi & \mbox{iff} & \mbox{$\Vdash_{\mathscr{B}} \phi$ and $\Vdash_{\mathscr{B}} \psi$} 
   \end{array}
\]
In such a setting, the validity of implication is given as something like 
\[
 \begin{array}{rcl}
        \Vdash_{\mathscr{B}} \phi \supset \psi & \mbox{iff} & 
        \mbox{$\phi \Vdash_{\mathscr{B}} \psi$} 
 \end{array}
\]
Note that this requires $\Vdash_\mathscr{B}$ to be defined at the level of consequence. In B-eS, this is usually handled by an `Inf' clause, of a form something like 
\[
\begin{array}{rrcl}
        \mbox{for $\Gamma \neq \emptyset$,}
        & \Gamma \Vdash_{\mathscr{B}} \phi  & \mbox{iff} & 
        \mbox{for all $\mathscr{C} \supseteq \mathscr{B}$,  $\Vdash_{\mathscr{C}} \Gamma$ implies $\Vdash_{\mathscr{C}} \phi$} 
 \end{array}
\]
Note that here the base-extension $\mathscr{C} \supseteq \mathscr{B}$ is reminiscent of the ordering of worlds in Kripke semantics. This is misleading: the essence of B-eS lies in the semantics of atomic propositions, which is given by the intensional notion of  provability in a base, whereas in Kripke semantics the meaning of atoms is given extensionally through interpretation in a model.\footnote{Prawitz \cite{Prawitz1971} (Appendix~A) explains the need for base-extension in giving the semantics of intuitionistic implication: We expect a
construction of an implicational formula $\phi \supset \psi$ to be a construction that given a construction of $\phi$ yields a construction of $\psi$. However, such a condition on a construction of $\phi \supset \psi$, as
formulated above, would be satisfied vacuously if there be no construction of $\phi$ relative to the system $S$ in question. It follows (cf. Kripke’s semantics of implication) that we must give the semantics of proofs of implications relative to all possible extensions $S' \supseteq S$ \cite{kripke1959,kripke1965semantical}.}     

See, for example,  \cite{Sandqvist2005,Sandqvist2009,Sandqvist2015,Piecha2019,Piecha2016} for more about the basics of B-eS.

Specifically, this paper is concerned with B-eS for public announcement logic (PAL) (see \cite{Baltag1998,Ditmarsch2007}). It builds directly on two previous papers. 
In the first \cite{Eckhardt2024}, a base-extension semantics for the family of logics $K$, $KT$, $K4$, and $S4$ was presented by adding a modal accessibility relation between bases that is used to define the validity of modal formulae similarly to Kripke semantics \cite{kripke1959,kripke1965semantical}. A set of conditions for modal relations 
that ensure soundness and completeness is given for the logics above. It is also shown that the conditions given in this paper allow for a Euclidean relation on which the corresponding axiom fails. In the second \cite{EckhardtPymS5}, this work was extended to multi-agent $S5$, by adapting the conditions and allowing for multiple modal relations. 


The language of PAL, in the absence of common knowledge, is that of classical propositional logic together with epistemic modalities 
$K_a$ and announcements  $[\psi]$. The modality 
$K_a$ is an $S5$ modality corresponding to an agent $a$ and a formula of the form $K_a \phi$ is read as `$a$ knows $\phi$'. Public announcements are truthful, in the sense that the announced formula has to be true at the time of announcing, and public, in the sense that every agent is aware of the announcement taking place. Announcements (and similar epistemic actions) are generally treated {\it dynamically} as an update operation that takes us from a model that represents the initial situation to one that a model of the situation after the announcement. An announcement formula $[\psi]\phi$ is read as `$\phi$ holds after the announcement of $\psi$'.  

In the present paper, we extend the base-extension analysis of $S5$ to PAL, the modal logic of knowledge, belief, and 
public communication --- and perhaps the most basic dynamic epistemic 
logic (DEL). This represents, we suggest, a step towards considering the 
value of P-tS in the setting of logics motivated by, and deployed in, 
the analysis of knowledge and public communication. We do so by adding dynamic components to the inferentialist account of $S5$.

The main difficulty in extending the existing B-eS for $S5$ is to account for PAL's announcements of the form $[\psi]\phi$. 

As already mentioned, in \cite{EckhardtPymS5}, a multi-agent $S5$ that defines the validity of epistemic formula (like $K_a \phi$) in reference to relations between bases in much the same way as modalities are handled in Kripke semantics has been defined. However, further restrictions are required on our relations (beyond the frame conditions) to ensure that these relations behave appropriately given that we also have the subset relation. Such relations are called {\it modal relations}. Announcements are usually treated by an update of the model. In Kripke semantics, for example, an announcement of $\phi$ simply causes all $\neg \phi$ worlds to be discarded from the model. In the Be-S case, the set of bases cannot be changed so simply and we have to restrict our updates to the relations between bases. This will require a more sophisticated update semantics given that we have to ensure that the updated relations are still modal relations. As we will see below, this will cause public announcements to no longer be general in the sense that the same announcement can produce different modal relations if it is updated at different bases (more on that later). Not all sets of modal relations can be updated in such a way, because it is not guaranteed that the relations updated by the announcement are still modal relations. We therefore define update sets on which this is possible and show that the B-eS for $S5$ restricted to such update sets is still sound and complete.


We use this B-eS for PAL to model two established and  illustrative examples from the DEL literature: a three-player card game and the `muddy children' puzzle. We contrast how these are these modelled in the inferentialist setting provided by B-eS with how these examples are treated in the established Kripke semantics for PAL. 

The B-eS modelling allows for a more intensional approach to these examples in which bases are not chosen just because of the formulae that are supported by them line up with the situation we want to model, but rather they are chosen such that their rules correspond directly, in a rather strong sense that we discuss in Section~\ref{sec:ExamplesRevisited}, to the example at hand. 

This correspondence manifests in two ways in the examples: first, it allows the rules of a game to be implemented  directly as base rules; second, it allows the information required for the execution of the model to be identified by looking at the minimal set-ups necessary for the agents in the examples to be able to carry out the expected reasoning. We use this observation as a point of departure for considering briefly, in Section~\ref{sec:conclusion}, some preliminary ideas on the connection between inferentionalism, intensionality, and information. The development of these ideas would be a substantial further project.   

In Section~\ref{sec:Background}, we introduce $S5$ modal logic and PAL together with their Kripke semantics. In Section~\ref{sec:Examples} we discuss some classic examples of public announcements: the three-player card game and the muddy children puzzle. 
In Section~\ref{sec:B-eS-S5}, we discuss the B-eS for $S5$. In Section~\ref{subsec:besS5EP}, we review the B-eS developed for $S5$ in \cite{EckhardtPymS5} and in Section~\ref{subsec:update-sets} we add further restrictions to the sets of relations between bases that are necessary for the development of PAL. We also prove soundness and completeness for these new sets of relations. This set-up forms the starting point for the B-eS of PAL, which we develop in Section~\ref{sec:B-eS-PAL}. In Section~\ref{sec:ExamplesRevisited}, we revisit the examples of the B-eS of PAL. These examples reveal how the inferentialist perspective, as 
formulated in terms of P-tS, expose the need for an 
explicit representation of all of the required assumptions. Finally, in Section~\ref{sec:conclusion}, 
we discuss possible future work on (other) dynamic 
epistemic logics and the informational content of B-eS 
for such logics. The proofs for Sections \ref{sec:B-eS-S5} and \ref{sec:B-eS-PAL} 
are mostly deferred to Appendices \ref{App1} and \ref{App2}, 
respectively. 

\section{Background} 
\label{sec:Background}


Here we summarize the key background in modal and epistemic logic that is required for our development. 

We begin, in Section~\ref{sec:Kripke-intro}, with a summary of multiagent $S5$, giving its established Kripke semantics together with a Hilbert-type proof system for which soundness and completeness obtain. Then, in Section~\ref{sec:PAL-intro}, we summarize PAL, also giving its established Kripke semantics together with a Hilbert-type proof system for which soundness and completeness obtain.    

\subsection{$S5$ and Its Kripke Semantics}
\label{sec:Kripke-intro}


We define a Kripke semantics for multi-agent $S5$ modal logic (an efficient, yet thorough, presentation of $S5$ in the context of epistemic logic may be found in \cite{Meyer_Hoek_1995}). We also give a corresponding axiomatic proof system. As we interpret the modal operator epistemically, we write $K_a$ to denote the $\square_a$-operator corresponding to an agent $a$. Accordingly, a formula $K_a \phi$ is read as `agent $a$ knows that $\phi$'.

A formula $K_a \phi$ holds at a world $w$ iff $\phi$ holds at all worlds $v$ such that $R_a wv$, where $R_a$ is a reflexive, transitive, and Euclidean relation between worlds. The natural way of interpreting this relation is as indistinguishability between agents for the agent $a$, that is, agent $a$ at world $w$ cannot tell whether world $w$ or $v$ is the actual world. Therefore, $K_a\phi$ holds at a world $w$ iff at all worlds agent $a$ knows are possible $\phi$ holds.

\medskip 

\begin{definition} 
\label{def:KripkeLanguage}
For a set of atomic formulae $P$, atomic formulae $p\in P$ and agents $a$, the language  for modal logic $S5$ is generated by the following grammar:
\begin{center}
    $\phi ::= p \mid \bot \mid \neg \phi \mid \phi\to\phi \mid K_a \phi$ 
\end{center} 
\vspace{-7mm}
\fillBox 

\end{definition} \medskip

In Kripke semantics, validity is defined with reference to specific worlds and models at which the formula is evaluated. We define models in the usual manner. 

\medskip 

\begin{definition}
    \label{def:model}
    A frame is a pair $F = \langle W, R_A\rangle$ in which $W$ is a set of possible worlds and $R_A$ a set of binary relations on $W$, one for each agent $a\in A$. A model $M= \langle F,V\rangle$ is a pair of a frame $F$ and a valuation function $V: P \rightarrow \mathcal{P}(W)$ giving for every propositional $p$ the set of worlds in which $p$ is true:  $V(p) \subseteq W$ for every $p$.
    The relation $R$ is called an $S5$-relation if it is an equivalence relation (i.e., reflexive, transitive, and Euclidean). An $S5$-frame is a frame $F = \langle W, R_A\rangle$ in which all $R_a \in R_A$ are $S5$-relations and an $S5$-model is a model $M=\langle F,V\rangle$ in which $F$ is an $S5$-frame. 
    \fillBox
\end{definition} \medskip

Given this, we can give the usual truth conditions for formulae at a world. The validity of a formula is then defined by its being true at all worlds in every model. As we are talking about $S5$ modal logic, we restrict this to $S5$-models.

\medskip 

\begin{definition}
    \label{KripkeValidity}
    Let $F = \langle W, R\rangle$ be an $S5$-frame, $M=\langle F,V\rangle$ be an $S5$-model, and $w\in W$ be a world. That a formula $\phi$ is true at $(M,w)$ --- denoted $M,w\vDash \phi$ --- is defined as follows:

\[
\begin{array}{l@{\quad}c@{\quad}l}
M,w\vDash p   & \mbox{iff} & w\in V(p) \\
M,w\vDash \phi\to\psi & \mbox{iff} & \mbox{if $M,w\vDash \phi$, then $M,w\vDash \psi$} \\ 
M,w\vDash \neg \phi & \mbox{iff} & \mbox{not $M,w\vDash \phi$} \\ 
M,w\vDash \bot & \mbox{iff} & \mbox{never} \\
M,w\vDash K_a \phi & \mbox{iff} & \mbox{for all $v$ s.t. $R_a wv$, $M,v\vDash \phi$}\\
\end{array} 
\]

    







\noindent If a formula $\phi$ is true at all worlds in a model $M$, we say $\phi$ is true in $M$. A formula is valid iff it is true in all models.
\fillBox
\end{definition} \medskip

This concludes the Kripke semantics for the epistemic logic underlying PAL. We give an axiomatic proof system that is sound and complete with respect to this semantics. Note  that the axioms $T$, $4$, and $5$ correspond to the relation being reflexive, transitive, and Euclidean, respectively. Axiomatic systems for other modal logics can be obtained by only using the corresponding axioms for that relation. 

\medskip 

\begin{definition} \label{def:S5Axioms}
    The proof system $\vdash^{S5}$ for the modal logic $S5$, with $a \in A$ is given by the following axioms and rules:

    \[
    \begin{array}{l@{\quad}l}
    \phi\to(\psi\to\phi) & \mbox{\emph{(1)}} \\\relax
    (\phi\to(\psi\to\chi))\to((\phi\to\psi)\to(\phi\to\chi)) & \mbox{\emph{(2)}} \\\relax
    (\neg\phi\to\neg\psi)\to (\psi\to\phi)  & \mbox{\emph{(3)}} \\\relax
    K_a(\phi\to\psi)\to(K_a\phi\to K_a\psi)  & \mbox{\emph{(K)}} \\\relax   
    K_a \phi\to\phi  & \mbox{\emph{(T)}} \\\relax   
    K_a \phi \to K_a K_a\phi  & \mbox{\emph{(4)}} \\\relax   
    \neg K_a \phi\to K_a\neg K_a\phi  & \mbox{\emph{(5)}} \\

& \\
    
    \mbox{If $\phi$ and $\phi\to\psi$, then $\psi$} & \mbox{Modus Ponens \emph{(MP)}}\\\relax
    \mbox{If $\psi$, then $K_a \psi$} & \mbox{Necessitation \emph{(NEC)}}
    \end{array}\]

    \noindent A formula is provable in S5, written $\vdash^{S5} \phi$, if it follows from these axioms and rules. Just the axioms \emph{(1)--(3)} together with the rule \emph{(MP)} constitute a proof system for classical logic. Proof systems for the modal logics other than $S5$ are obtained by taking only the corresponding axioms (e.g., modal logic $K4$ results from adding \emph{(K)}, \emph{(4)}, and \emph{(NEC)} to the classical proof system). \fillBox 
\end{definition} \medskip

\subsection{Public Announcement Logic}
\label{sec:PAL-intro}

To obtain the language for PAL, the language of multi-agent $S5$ modal logic is enriched with announcement operators of the form $[\phi]\psi$, which are read as `after a (truthful and public) announcement of $\phi$, $\psi$ holds'. Note that $[\phi]$ is a box-like operator and so it should read `after {\it every} announcement of $\phi$'; however, these two readings are equivalent as announcements are partial functions (i.e., there is at most one announcement of $\phi$)\footnote{This will no longer be the case in our base-extension semantics (see Section \ref{sec:B-eS-PAL}).}.

\medskip 

\begin{definition}
\label{def:PALLanguage}
    For a set of atomic formulae $P$, atomic formulae $p \in P$ and agents $a$, the language of $PAL$ is generated by the following grammar:
    \begin{center}
    $\phi := p \mid \bot \mid  \neg \phi \mid \phi\to\phi \mid K_a \phi \mid [\phi]\phi$
    \end{center}
\vspace{-7mm}
\fillBox
\end{definition} \medskip

The Kripke semantics of PAL is obtained by enriching that of multi-agent $S5$ with an update operation that restricts the worlds to only those in which the announced formula is true. Announcing a formula is therefore seen as a 
model-transformation operation that results in a new Kripke model representing the state of knowledge of the agents after the announcement takes places.

Following that idea, we write $M|\phi$ for the model corresponding to $M$ restricted to worlds at which $\phi$ holds.

\medskip 

\begin{definition}
    For any model $M= \langle F, V\rangle$ with $F=\langle W,R\rangle$ and formula $\phi$, we define $M|\phi = \langle F', V'\rangle$ with $F'=\langle W', R'\rangle$ in the following way:
\begin{center}
$\begin{array}{l@{\quad}c@{\quad}l}   
W' & = & \{w \mid w\in W \mbox{ and } M,w\vDash \phi \}\\
R' & = & R \cap (W' \times W')\\
V(p)' & = & V(p) \cap W'\\
    \end{array}$
\end{center} 
\vspace{-7mm}
\fillBox
\end{definition} \medskip

Given this, we can define validity, with a formula of the form $[\phi]\psi$ being true at a world if $\psi$ is true at that world after the announcement of $\phi$. We add this condition to our truth conditions in Definition \ref{KripkeValidity}.

\medskip 

\begin{definition} \label{KripkePALValidity}
    Let $F = \langle W, R\rangle$ be an $S5$-frame, $M=\langle F,V\rangle$ an $S5$-model and $w\in W$ a world. That a formula $\phi$ is true at $(M,w)$ --- denoted $M,w\vDash \phi$ --- is defined as follows:
\[
\begin{array}{l@{\quad}c@{\quad}l}
M,w\vDash p   & \mbox{iff} & w\in V(p) \\
M,w\vDash \phi\to\psi & \mbox{iff} & \mbox{if $M,w\vDash \phi$, then $M,w\vDash \psi$} \\ 
M,w\vDash \neg \phi & \mbox{iff} & \mbox{not $M,w\vDash \phi$} \\ 
M,w\vDash \bot & \mbox{iff} & \mbox{never} \\
M,w\vDash K_a \phi & \mbox{iff} & \mbox{for all $v$ s.t. $R_a wv$, $M,v\vDash \phi$}\\
M,w\vDash [\phi]\psi & \mbox{iff} & \mbox{$M,w\vDash \phi$ implies $M|\phi,w\vDash\psi$}
\end{array} 
\]

    







\noindent If a formula $\phi$ is true at all worlds in a model $M$, we say $\phi$ is true in $M$. A formula is valid iff it is true in all models.
\fillBox 
\end{definition} \medskip

A sound and complete proof-system for PAL is given by the following axiomatic system: 

\medskip 

\begin{definition}
We add the following axioms to the axiomatic system for $S5$ modal logic in Definition \ref{def:S5Axioms}:

    \[
    \begin{array}{l@{\quad}l}
    [\phi]p \leftrightarrow \phi\to p & \mbox{atomic permanence} \\\relax
    [\phi]\bot\leftrightarrow \phi\to\bot & \mbox{announcement and bot} \\\relax
    [\phi](\psi\to\chi)\leftrightarrow [\phi]\psi \to [\phi]\chi  & \mbox{announcement and implication} \\\relax
    [\phi]K_a \psi\leftrightarrow \phi\to K_a[\phi]\psi  & \mbox{announcement and knowledge} \\\relax   
    [\phi][\psi]\chi\leftrightarrow [\phi\wedge[\phi]\psi]\chi  & \mbox{announcement composition}\\

    \mbox{If $\psi$, then $[\phi]\psi$} & \mbox{necessitation of announcements (NOA)}
    \end{array}\]
    A formula is provable in PAL, written $\vdash \phi$, if it follows from these axioms and rules.\footnote{We take $\wedge$ here to be defined from $\to$ and $\neg$ in the usual way: $\phi\wedge\psi = \neg(\phi\to\neg\psi)$} 
    \fillBox
\end{definition} \medskip

\section{Examples of Public Announcements} \label{sec:Examples}

It is useful here to discuss some examples of how the reasoning of PAL can be applied. This shows how the standard semantics for PAL works and will allow us to highlight, when we return to the examples in Section~\ref{sec:ExamplesRevisited}, the differences with the base-extension semantics treatment of PAL developed later in this paper. 

The examples and their treatment in the standard semantics are taken from \cite{Ditmarsch2007}. 

Our first example will be a simple card game between three agents. Our second example, the most famous puzzle concerning PAL, is the muddy-children puzzle. We have chosen these examples because they will highlight, 
as we discuss in Section~\ref{sec:ExamplesRevisited}, how informational content is represented in our base-extension semantics and how that differs from the treatment of these examples in Kripke semantics.

\subsection{Three-player Card Game}
\label{subsec:standard3pcg}
\begin{example}[Card game]
    \label{ex:cardgame1}
    Three players, Anne, Bob, and Cath, each draw a card from a deck of three cards, 0, 1, and 2. This is known among the agents. Anne draws 0, Bob 1, and Cath 2. Only the players themselves know which card they have drawn. Anne announces `I do not have card 1'. \fillBox
\end{example}

Before it is decided who draws which card, there are six possible distributions of the card draw. We write, for example, 012 for the case in which Anne draws 0, Bob draws 1, and Cath 2 (i.e., the actual case in our example).

However, before looking at the thought processes of Anne, Bob, and Cath, we need to analyse the framework in which the game takes place. There are some important rules that must be obeyed in interpretation of the example. Every agent draws exactly one card. By the rules of the game, no agent can draw two cards or no card at all. Conversely, every card is drawn by exactly one agent. There are no cards left out of the deck and there are no multiples of any card. Lastly, every agent can see their own card but not the cards of the other agents.

Given these rules, we can start drawing some conclusions. For every agent, there are certain draws between which they cannot decide. For example, Anne cannot decide between 012 and 021, as they can only say that they themselves have 0 and that Bob and Cath each have one of the other two cards. For the same reasons, Cath cannot decide between 012 and 102. In fact, these are the only draws she considers possible at 012.

Now we go forward to Anne's announcement `I do not have card 1'. This announcement contradicts 102 and so Cath no longer considers it. This leaves 012 as the only possible draw Cath considers and they now know exactly who has which card.

In Kripke semantics, this example can be modelled in a straightforward way: we simply take a model with six worlds, one for each of the possible draws, and define relations for each of the agents that fit the idea that the agents only know their own card. Let 012 be the world that corresponds to that specific draw and $R_a$ Anne's relation, $R_b$ Bob's, etc.. So we want our initial model to have $R_a 012\; 021$ and so on. Given this, we can draw a model $M$ that encompasses all possible draws (Figure \ref{KripkeCardgame1}).

\begin{figure}[ht]
\hrule
\begin{center}

\begin{tikzpicture}[-,shorten >=1pt,node distance =2cm, thick, baseline={([yshift=-1.8ex]current bounding box.center)}] 

\node (A) at (0,1) {$102$};
\node (B) at (1,0) {$201$};
\node (C) at (1,2) {$012$};
\node (D) at (4,1) {$120$};
\node (E) at (3,0) {$210$};
\node (F) at (3,2) {$021$};

\draw[-,near start, above] (A) to node {$a$} (D);
\draw[-,above] (B) to node {$a$} (E);
\draw[-,above] (C) to node {$a$} (F);

\draw[-,above] (A) to node {$b$} (B);
\draw[-,near start, above] (C) to node {$b$} (E);
\draw[-,above] (D) to node {$b$} (F);

\draw[-,above] (A) to node {$c$} (C);
\draw[-,near end, above] (B) to node {$c$} (F);
\draw[-,above] (D) to node {$c$} (E);

\end{tikzpicture}

\caption{\label{KripkeCardgame1} The initial Kripke model $M$ for the card game example.}
\end{center}
\vspace{1mm}
\hrule
\end{figure}

Following \cite{Ditmarsch2007}, we use $N_i$ for the atomic sentences so that `agent $i$ has card $N$', for example at world $012$, we have $0_a, 1_b,$ and $2_c$. Note that the way our model is set-up guarantees that the rules of the game, as discussed above, hold at all the worlds. 

Anne's announcement of `I do not have card 1' corresponds to $\neg 1_a$. So the announcement, as expected, will restrict our model to those worlds at which Anne does not have card $1$, meaning we delete the worlds $102$ and $120$. Similarly, the relation is restricted to only the remaining worlds. This results in the model $M'=M|\neg1_a$ of Figure \ref{KripkeCardgame2}.

\begin{figure}[ht]
\hrule
\begin{center}

\begin{tikzpicture}[-,shorten >=1pt,node distance =2cm, thick, baseline={([yshift=-1.8ex]current bounding box.center)}] 

\node (B) at (1,0) {$201$};
\node (C) at (1,2) {$012$};
\node (E) at (3,0) {$210$};
\node (F) at (3,2) {$021$};

\draw[-,above] (B) to node {$a$} (E);
\draw[-,above] (C) to node {$a$} (F);

\draw[-,near start, above] (C) to node {$b$} (E);

\draw[-,near end, above] (B) to node {$c$} (F);

\end{tikzpicture}

\caption{\label{KripkeCardgame2} The Kripke model $M'$ for the card game example after Anne's announcement `I do not have card 1'.}
\end{center}
\vspace{1mm}
\hrule
\end{figure}

Given this we can now show that after Anne's announcement Cath has enough information to identify the draw that took place. We can see that $M',012\models K_c (0_a \wedge 1_b\wedge 2_c)$ and so $M,012 \models [\neg 1_a] K_c (0_a \wedge 1_b\wedge 2_c)$.

As we can see, Anne's announcement leads to a model in which Cath can determine 
who has which card, just as we expected.

\subsection{The Muddy Children Puzzle}
\label{subsec:standardmcp}

Again, like with the card game example, we start by informally describing the muddy children puzzle and then show how this puzzle can be solved using Kripke models.

\medskip 

\begin{example}[Muddy Children]
\label{ex:muddychildren}
    A group of children playing outside are called by their father. Some of them have become dirty from playing outside and they have mud on their foreheads. The children cannot see their own forehead but the foreheads of all the other children. The children are, of course, perfect logicians.

    The father proclaims: `At least one of you has mud on their forehead. Will those with mud on their forehead come forward to get cleaned?'. The children do not like being cleaned and will avoid so if possible (i.e., they do not have muddy foreheads). If no child steps forward the father will repeat this request until the muddy children step forward.

    If $m$ of the $n$ children are muddy, the children will step forward after the father has made his request $m$ times.
\fillBox
\end{example}

    We are going to focus on the case in which there are three children, two of which are muddy. Let us call the children Anne, Bob, and Cath as in the card game example above and let Anne and Bob be muddy. This gives us a specific example to reason through and the general case follows the exact same reasoning. 

    So how are the children capable of finding out whether they are muddy or not without seeing their own foreheads? Let us take Anne's point of view (Bob's, of course, will be analogous). Anne sees that Bob is muddy and that Cath is not muddy. The father now states that at least one of them is muddy. So Anne considers that Bob can also see that Cath is not muddy and so if Anne was not muddy, Bob would realize that they were the only muddy one. When the father repeats his request without Bob having stepped forward, Anne realizes that Bob has to see another muddy person. That can only be Anne herself though and so Anne (and Bob) step forward as they now know that they are the muddy ones.

    Similarly to the card game example, we can represent this using Kripke models and by treating the actions of the father and the children as public announcements. The atomic formulae are $m_a, m_b$, and $m_c$ for Anne, Bob and Cath are muddy, respectively. Given that we have three atomic formulae we can represent the different worlds in a cube (see Figure \ref{KripkeMC1}).

\begin{figure}[ht]
\hrule 
\begin{center}

\begin{tikzpicture}[-,shorten >=1pt,node distance =2cm, thick, baseline={([yshift=-1.8ex]current bounding box.center)}] 

\node (A) at (0,0) {$\emptyset$};

\node (C) at (0,3) {$m_b$};
\node (D) at (1.5,1.5) {$m_c$};
\node (B) at (3,0) {$m_a$};

\node (E) at (3,3) {$m_a,m_b$};
\node (F) at (4.5,1.5) {$m_a,m_c$};
\node (G) at (1.5,4.5) {$m_b,m_c$};

\node (H) at (4.5,4.5) {$m_a,m_b,m_c$};

\draw[-,above] (A) to node {$a$} (B);
\draw[-,left] (A) to node {$b$} (C);
\draw[-,above] (A) to node {$c$} (D);

\draw[-, near start,left] (B) to node {$b$} (E);
\draw[-,above] (B) to node {$c$} (F);

\draw[-,near end, above] (C) to node {$a$} (E);
\draw[-,above] (C) to node {$c$} (G);

\draw[-,near start, above] (D) to node {$a$} (F);
\draw[-,near start, left] (D) to node {$b$} (G);

\draw[-,above] (E) to node {$c$} (H);
\draw[-,left] (F) to node {$b$} (H);
\draw[-,above] (G) to node {$a$} (H);
\end{tikzpicture}

\caption{\label{KripkeMC1} The initial Kripke model for the muddy children puzzle ahead of the father's announcements.}
\end{center}
\vspace{1mm}
\hrule
\end{figure}

We can model the father's announcement that some of the children are muddy by an announcement of $m_a\vee m_b \vee m_c$. For readability, let $\phi= m_a\vee m_b\vee m_c$. This removes the world in which no child is muddy and all relations between it and the other worlds (see Figure \ref{KripkeMC2}).

\begin{figure}[ht]

\hrule
\begin{center}

\begin{tikzpicture}[-,shorten >=1pt,node distance =2cm, thick, baseline={([yshift=-1.8ex]current bounding box.center)}] 


\node (C) at (0,3) {$m_b$};
\node (D) at (1.5,1.5) {$m_c$};
\node (B) at (3,0) {$m_a$};

\node (E) at (3,3) {$m_a,m_b$};
\node (F) at (4.5,1.5) {$m_a,m_c$};
\node (G) at (1.5,4.5) {$m_b,m_c$};

\node (H) at (4.5,4.5) {$m_a,m_b,m_c$};


\draw[-, near start,left] (B) to node {$b$} (E);
\draw[-,above] (B) to node {$c$} (F);

\draw[-,near end, above] (C) to node {$a$} (E);
\draw[-,above] (C) to node {$c$} (G);

\draw[-,near start, above] (D) to node {$a$} (F);
\draw[-,near start, left] (D) to node {$b$} (G);

\draw[-,above] (E) to node {$c$} (H);
\draw[-,left] (F) to node {$b$} (H);
\draw[-,above] (G) to node {$a$} (H);

\end{tikzpicture}

\caption{\label{KripkeMC2} The Kripke model for the muddy children puzzle after the announcement of $m_a\vee m_b\vee m_c$ (i.e., that some children are muddy).}
\end{center}
\vspace{1mm}
\hrule
\end{figure}

Now the father repeats his request for the children to get cleaned. Above we treated the children not stepping forward to get cleaned as an announcement that they do not know whether they are muddy or not. This corresponds to the announcement of the formula $\psi =\neg (K_a m_a\vee K_a \neg m_a) \wedge \neg (K_b m_b\vee K_b \neg m_b) \wedge \neg (K_c m_c\vee K_c \neg m_c)$. So, when the children fail to step up after the first request by the father, we update our model of Figure \ref{KripkeMC2} with that announcement. This results in Figure \ref{KripkeMC3}.

\begin{figure}[ht]
\hrule 
\begin{center}

\begin{tikzpicture}[-,shorten >=1pt,node distance =2cm, thick, baseline={([yshift=-1.8ex]current bounding box.center)}] 



\node (E) at (3,3) {$m_a,m_b$};
\node (F) at (4.5,1.5) {$m_a,m_c$};
\node (G) at (1.5,4.5) {$m_b,m_c$};

\node (H) at (4.5,4.5) {$m_a,m_b,m_c$};





\draw[-,above] (E) to node {$c$} (H);
\draw[-,left] (F) to node {$b$} (H);
\draw[-,above] (G) to node {$a$} (H);

\end{tikzpicture}

\caption{\label{KripkeMC3} The Kripke model for the muddy children puzzle after the father has to repeat his request (modelled by the announcement of $\neg (K_a m_a\vee K_a \neg m_a) \wedge \neg (K_b m_b\vee K_b \neg m_b) \wedge \neg (K_c m_c\vee K_c \neg m_c)$)}
\end{center}
\vspace{1mm}
\hrule
\end{figure}

In the model of Figure \ref{KripkeMC3}, we can see that at the world with $m_a, m_b$ Anne and Bob do not have any other worlds they consider and so $K_a m_a$ and $K_b m_b$. So from our initial model $[\phi][\psi] K_a m_a \wedge K_b m_b$ holds. Cath, of course, still believes it is possible that they themselves are muddy as well, as they still think all of them might be muddy (at least until Anne and Bob step forward).

Note that this reasoning is not depending on the number of total children. Regardless of $n$, this reasoning allows the $m$ actually muddy children to realize that they are muddy in $m$ steps, as every repetition of the father's request will remove those worlds with the least muddy children until there are no more alternatives to the actual world.

\section{Base-extension Semantics for $S5$} \label{sec:B-eS-S5}


Here we give, in Section~\ref{subsec:besS5EP}, a base-extension semantics for multi-agent $S5$, as presented in \cite{EckhardtPymS5}. This approach builds on that for the systems $K$, $KT$, $K4$, and $S4$, as presented in \cite{Eckhardt2024}. The correctness of base-extension semantics for $S5$ is established relative to its Kripke semantics \cite{EckhardtPymS5}. This is the starting point from which we develop the base-extension semantics for PAL.


To set up our semantics for PAL, we modify the base-extension semantics of $S5$ to accommodate the notion of update that is required for its use as a basis for PAL. This requires the notion of an `update set', which is given in Section~\ref{subsec:update-sets}. 

\subsection{Base-extension semantics} 
\label{subsec:besS5EP}

We adapt the language used in Section~\ref{sec:Kripke-intro} for $S5$ given in Definition~\ref{def:KripkeLanguage} without $\neg$, as in our set-up $\neg \phi$ will be defined as $\phi\to\bot$.

\begin{definition}
\label{def:BesS5Language}
For a set of atomic formulae $P$, atomic formulae $p\in P$ and agents $a$, the language  for modal logic $S5$ is generated by the following grammar:
    \begin{center}
    $\phi := p \mid \bot \mid \phi\to\phi \mid K_a \phi$
    \end{center}
\vspace{-7mm}
\fillBox 
\end{definition} \medskip

Base rules and bases build the foundation of validity in base-extension semantics: we must define them before moving to validity conditions. 

\medskip 

\begin{definition} \label{ClassicalRule}
A \textit{base rule} for a countably infinite set of propositional atoms $P$ is a pair $(L_j, p)$, where $L_j  = \{p_1, \dots, p_n \}$, is a finite (possibly empty) set of atomic formulae and $p$ is also an atomic formula in $P$. Generally, a base rule will be written as $p_1, \dots, p_n \Rightarrow p$ ($\Rightarrow p$ for axioms). A \emph{base} $\mathscr{B}$ is any countable collection of base rules. Given a set of atoms $P$, we call the set of all bases $\Omega_P$. $\overline{\mathscr{B}}$ is the closure of the empty set under the rules in $\mathscr{B}$.
A base is called \emph{inconsistent} for a set of atoms $P$, if $p\in \overline{\mathscr{B}}$ for all $p\in P$. All other bases are \textit{consistent}.
To avoid clutter, we generally omit stating $P$. \fillBox 
\end{definition} \medskip

Before we define validity, we define the notion of \emph{modal relation}. We impose four conditions on modal relations. Conditions (a) and (b) separate consistent and inconsistent bases. Condition (c) guarantees that if there is a relation from a base $\mathscr{B}$ to base $\mathscr{C}$, then for every consistent superset base of $\mathscr{B}$, there is a superset base of $\mathscr{C}$ between which the relation holds. Analogously, condition (d) guarantees that if there is a relation from a base $\mathscr{B}$ to base $\mathscr{C}$, then for every subset base of $\mathscr{B}$, there is a subset base of $\mathscr{C}$ between which the relation holds. So (c) and (d) can be thought of as fixing the interaction between the modal relations and the subset relation. See \cite{EckhardtPymS5} for a more detailed discussion and illustrations of these conditions.

\medskip 

\begin{definition} \label{modalRelation}
A relation $\mathfrak{R}_a$ on the set of bases $\Omega_P$ is called an $S5$-modal relation iff, for all $\mathscr{B}$,


\begin{enumerate}[label=(\alph*)]
    \item if $\mathscr{B}$ is inconsistent, then there is an inconsistent $\mathscr{C}$ s.t. $\mathfrak{R}_a\mathscr{B}\mathscr{C}$ and, for all $\mathscr{D}$, $\mathfrak{R}_a\mathscr{B}\mathscr{D}$ implies that $\mathscr{D}$ is also inconsistent. 

    \item if $\mathscr{B}$ is consistent, then for all $\mathscr{C}$ s.t. $\mathfrak{R}_a \mathscr{B}\mathscr{C}$, $\mathscr{C}$ is also consistent.

    \item for all $\mathscr{C}$, if $\mathfrak{R}_a\mathscr{B}\mathscr{C}$, then, for all consistent $\mathscr{D}\supseteq \mathscr{B}$, there is an $\mathscr{E}\supseteq\mathscr{C}$ s.t. $\mathfrak{R}_a\mathscr{D}\mathscr{E}$.

    \item for all consistent $\mathscr{C}$, if $\mathfrak{R}_a\mathscr{B}\mathscr{C}$, then, for all $\mathscr{D}\subseteq\mathscr{B}$, there is an $\mathscr{E}\subseteq\mathscr{C}$ s.t. $\mathfrak{R}_a\mathscr{D}\mathscr{E}$.


\end{enumerate}

\noindent A modal relation $\mathfrak{R}$ is called an $S5$-modal relation iff $\mathfrak{R}$ is reflexive, transitive, and Euclidean. \fillBox 
\end{definition} \medskip

We define validity at a base given a set of modal relations. As mentioned in Section \ref{sec:introduction}, validity is defined inductively with the base case for atoms is defined in terms of provability from the rules in the base. The modal relations are used for formulae of the form $K_a \phi$ in a manner analogous to Kripke semantics. General validity is defined as being valid at all bases given any modal relation.

\medskip 

\begin{definition} \label{EXTValidity}

For the epistemic modal logic, we define validity at a base $\mathscr{B}$ given a set of agents $A$, a set of atomic formulae $P$ and a set of $S5$-modal relations $\mathfrak{R}_a$, one for every agent $a \in A$, denoted by $\mathfrak{R}_A$ as follows:
\[
\begin{array}{l@{\quad}c@{\quad}l}
\Vdash_{\mathscr{B},\mathfrak{R}_A} p   & \mbox{iff} & \mbox{$p$ is in every set of basic sentences closed under $\mathscr{B}$} \\  
        & & \mbox{(i.e., iff $p\in \overline{\mathscr{B}}$)} \\
       
\Vdash_{\mathscr{B},\mathfrak{R}_A} \phi\to\psi & \mbox{iff} & \mbox{$\phi\Vdash_{\mathscr{B},\mathfrak{R}_A} \psi$} \\ 
\Vdash_{\mathscr{B},\mathfrak{R}_A} \bot & \mbox{iff} & \mbox{$\Vdash_{\mathscr{B},\mathfrak{R}_A} p$ for every basic sentence $p\in P$} \\
\Vdash_{\mathscr{B},\mathfrak{R}_A} K_a \phi & \mbox{iff} & \mbox{for all $\mathscr{C}$ s.t. $\mathfrak{R}_a\mathscr{B}\mathscr{C}$, $\Vdash_{\mathscr{C},\mathfrak{R}_A} \phi$}\\
& & \\
 \mbox{For non empty $\Gamma$:} & &\\
\Gamma\Vdash_{\mathscr{B},\mathfrak{R}_A} \phi & \mbox{iff} & \mbox{for all $\mathscr{C}\supseteq \mathscr{B}$, if $\Vdash_{\mathscr{C},\mathfrak{R}_A} \psi$ for all $\psi \in \Gamma$, then $\Vdash_{\mathscr{C},\mathfrak{R}_A} \phi$}\\ 
\end{array} 
\]








\noindent Given a set of agents $A$ and a set of atomic formulae $P$, a formula $\phi$ is \textit{valid}, written as $\Vdash \phi$, iff $\Vdash_{\mathscr{B}, \mathfrak{R}_A} \phi$ for all bases $\mathscr{B}\in\Omega_P$ and set of modal relations $\mathfrak{R}_A$. 
We write $\Gamma \Vdash \phi$ iff for all $\mathscr{B}$ and 
all $\mathfrak{R}_A$, if $\Vdash_{\mathscr{B}, \mathfrak{R}_A} \psi$, for all $\psi \in \Gamma$, then $\Vdash_{\mathscr{B}, \mathfrak{R}_A} \phi$.\footnote{In \cite{EckhardtPymS5}, although this definition of $\Gamma \Vdash \phi$, which is the standard one (see, for example, \cite{Sandqvist2015}), is assumed in the necessary proofs, it is omitted from the corresponding formal definition in \cite{EckhardtPymS5}.}
\fillBox
\end{definition} \medskip

The following soundness and completeness theorem for multi-agent $S5$'s base-extension semantics relative to its Kripke semantics has been established in \cite{EckhardtPymS5}: 

\medskip 

\begin{theorem} \label{def:S5completeness}
    The base-extension semantics for multi-agent modal logic $S5$ is sound and complete with respect to the multi-agent modal logic $S5$.
\end{theorem}

\subsection{Update Sets} \label{subsec:update-sets}

 Even though the semantics established above is sound and complete for the epistemic modal logic underlying PAL, we will see below, that it will cause certain updates to not be possible. We further restrict our set of modal relations to avoid those cases. This differs from the conditions $(a)-(c)$ of Definition \ref{modalRelation} in that it is not a condition on the single modal relations but rather on the set of them.  First we define the idea of reachability given a set of relations.

\medskip 

\begin{definition} \label{def:reachable}
    A base $\mathscr{C}$ is considered reachable from a base $\mathscr{B}$ by a set of relations $\mathfrak{R}_A$ iff there are bases $\mathscr{C}_1, \dots, \mathscr{C}_n$ and agents $a_0, \dots a_n \in A$ s.t. $\mathfrak{R}_{a_0}\mathscr{B}\mathscr{C}_1$, $\mathfrak{R}_{a_1} \mathscr{C}_1\mathscr{C}_2$, $\dots$, $\mathfrak{R}_{a_n}\mathscr{C}_n\mathscr{C}$ (i.e., there is a way to get from $\mathscr{B}$ to $\mathscr{C}$ using some combination of the relations in $\mathfrak{R}_A$).

    Similarly, a base $\mathscr{C}$ is reachable from a base $\mathscr{B}$ using the sub/superset relation iff there are bases $\mathscr{C}_1, \dots, \mathscr{C}_n$ s.t. $\mathscr{B}\bullet_1\mathscr{C}_1,\mathscr{C}_1\bullet_2\mathscr{C}_2, \dots, \mathscr{C_n}\bullet_n\mathscr{C}$ where $\bullet_i\in\{\subseteq,\supseteq\}$. 
    \fillBox 
\end{definition} \medskip

Besides dealing with modal relations, we put two further restrictions on update sets. If we have two bases that are reachable from each other using the set of relations and are also reachable from each other using the sub/superset relation, we require them to have access to the same bases. Additionally, for any consistent base we can divide up the bases that are reachable from it into a finite amount of sets s.t. all members of one of these sets has access to the same bases. We show in Section \ref{sec:B-eS-PAL} that this guarantees that the set of modal relations can be updated according to an announcement as long as the announced formula holds at that base and we give an example of how this might fail otherwise. 

\medskip 

\begin{definition}
    \label{def:UpdateRelations}

    Given a set of atoms $P$, a set of relations $\mathfrak{R}_A$ is called an \text{update} set iff

    \begin{itemize}
        \item[—] all $\mathfrak{R}_a \in \mathfrak{R}_A$ are $S5$-modal relations, and
        \item[—] for all consistent $\mathscr{B}$ and $\mathscr{B}'$ s.t. $\mathscr{B}'$ is reachable from $\mathscr{B}$ and $\mathscr{B}'\subseteq\mathscr{B}$, and all $\mathfrak{R}_a\in\mathfrak{R}_A$ and $\mathscr{C}$, \[\mathfrak{R}_a\mathscr{B}\mathscr{C} \text{ iff }\mathfrak{R}_a\mathscr{B}'\mathscr{C} \]
        \item[—] for all consistent bases $\mathscr{B}$, there is a finite set $\pi({\mathscr{B}}) = \{\pi_1,\dots, \pi_n\}$ s.t
        \begin{itemize}
            \item[—] for all $i$, $\pi_i\subseteq \Omega_P$
            \item[—] for all $i$, and $\mathscr{C},\mathscr{D}\in \pi_i$, $\mathscr{D}$ is reachable from $\mathscr{C}$ using the sub/superset relation
            \item[—] for all $\mathscr{C}$, $\mathscr{C}$ is reachable from $\mathscr{B}$ using $\mathfrak{R}_A$ iff there is a $\pi_i$ s.t. $\mathscr{C}\in \pi_i$.
        \end{itemize}
        
    \end{itemize}
    \vspace{-7.5mm}
    \fillBox
\end{definition} 

We can now restate validity in terms of update sets simply by adapting Definition \ref{EXTValidity}.

\medskip 

\begin{definition}
 \label{UpdateValidity}

For the epistemic modal logic, we define validity at a base $\mathscr{B}$ given a set of agents $A$, a set of atomic formulae $P$, and an {\bf update set} $\mathfrak{R}_A$ consisting of $S5$-modal relations $\mathfrak{R}_a$, one for every agent $a \in A$, as follows:
\[
\begin{array}{l@{\quad}c@{\quad}l}
\Vdash_{\mathscr{B},\mathfrak{R}_A} p   & \mbox{iff} & \mbox{$p$ is in every set of basic sentences closed under $\mathscr{B}$} \\  
        & & \mbox{(i.e., iff $p\in \overline{\mathscr{B}}$)} \\
       
\Vdash_{\mathscr{B},\mathfrak{R}_A} \phi\to\psi & \mbox{iff} & \mbox{$\phi\Vdash_{\mathscr{B},\mathfrak{R}_A} \psi$} \\ 
\Vdash_{\mathscr{B},\mathfrak{R}_A} \bot & \mbox{iff} & \mbox{$\Vdash_{\mathscr{B},\mathfrak{R}_A} p$ for every basic sentence $p\in P$} \\
\Vdash_{\mathscr{B},\mathfrak{R}_A} K_a \phi & \mbox{iff} & \mbox{for all $\mathscr{C}$ s.t. $\mathfrak{R}_a\mathscr{B}\mathscr{C}$, $\Vdash_{\mathscr{C},\mathfrak{R}_A} \phi$}\\
& & \\
 \mbox{For non empty $\Gamma$:} & &\\
\Gamma\Vdash_{\mathscr{B},\mathfrak{R}_A} \phi & \mbox{iff} & \mbox{for all $\mathscr{C}\supseteq \mathscr{B}$, if $\Vdash_{\mathscr{C},\mathfrak{R}_A} \psi$ for all $\psi \in \Gamma$, then $\Vdash_{\mathscr{C},\mathfrak{R}_A} \phi$}\\ 
\end{array} 
\]

\noindent Given a set of agents $A$ and set of atomic formulae $P$, a formula $\phi$ is \textit{valid}, written as $\Vdash \phi$, iff $\Vdash_{\mathscr{B}, \mathfrak{R}_A} \phi$ for all bases $\mathscr{B}\in\Omega_P$ and update sets $\mathfrak{R}_A$. We write $\Gamma \Vdash \phi$ iff for all $\mathscr{B}$ and 
all $\mathfrak{R}_A$, if $\Vdash_{\mathscr{B}, \mathfrak{R}_A} \psi$, for all $\psi \in \Gamma$, then $\Vdash_{\mathscr{B}, \mathfrak{R}_A} \phi$.  \fillBox 
\end{definition} \medskip

\medskip 

We also define maximally-consistent bases here. They behave interestingly and will be referenced throughout the rest of the paper.

\begin{definition}
    \label{MaxConDef}
    A base $\mathscr{B}$ is maximally-consistent iff it is consistent and for every base rule $\delta$, either $\delta \in \mathscr{B}$ or $\mathscr{B}\cup\{\delta\}$ is inconsistent. \fillBox
\end{definition} \medskip

\begin{theorem}
    \label{theo:UpdateCompleteness}

    The base-extension semantics for multi-agent modal logic $S5$ based on update sets is sound and complete with respect to the multi-agent modal logic $S5$.
    
\end{theorem}

\begin{proof}
    We establish that any formula that is valid in the Kripke semantics for multi-agent $S5$ is also valid in the base-extension semantics for multi-agent $S5$ without the update condition requirement in \cite{EckhardtPymS5}. This will obviously still hold if we add additional conditions.
    
    For the other direction, we show in \cite{EckhardtPymS5} that any formula that is false on a Kripke modal can be made false in some base given some set of relations. As it turns out, this construction can be adapted to update sets and so used to proof the same given some update set. The adapted construction and proof that it results in an update set is in Appendix \ref{App1}.
\end{proof}

Some lemmas important in establishing the soundness and completeness results can be found in Appendix \ref{App1}. They will be generalized to the semantics for PAL in Section \ref{sec:B-eS-PAL}. More details on how these are used in the proof of soundness and completeness can be found in \cite{EckhardtPymS5} as they remain fundamentally unchanged.

\section{Base-extension Semantics for PAL} \label{sec:B-eS-PAL}

In order to develop a base-extension for PAL, we must first give a way of representing the model-update of PAL's Kripke semantics in base-extension semantics. The set of bases, however, is given by the propositional letters and the shape of our base rules. So, we look at updating the modal relations $\mathfrak{R}_a$ instead. As in Section~\ref{sec:PAL-intro}, we add the announcement operator to the underlying language of $S5$ (Definition~\ref{def:BesS5Language}). Note that we simply remove $\neg$ from the language used for the Kripke semantics of PAL in Section~\ref{sec:PAL-intro} (Definition~\ref{def:PALLanguage}).

\begin{definition}
\label{def:BesPALLanguage}
    For a set of atomic formulae $P$, atomic formulae $p \in P$ and agents $a$, the language of $PAL$ is generated by the following grammar:
    \begin{center}
    $\phi := p \mid \bot \mid \phi\to\phi \mid K_a \phi \mid [\phi]\phi$
    \end{center} \vspace{-7mm}
\fillBox
\end{definition} \medskip

Given a relation $\mathfrak{R}_a$ and an announcement of a formula $\phi$, we want to create a new relation $\mathfrak{R}_a |\phi$ that corresponds to $\mathfrak{R}_a$ after the announcement of $\phi$ much like $M|\phi$ corresponds to $M$.

An obvious candidate for such an $\mathfrak{R}_a|\phi$ might be that subset of $\mathfrak{R}_a$ where for all $\mathscr{B}$, $\mathscr{C}$ s.t. $\mathfrak{R}_a|\phi \mathscr{B}\mathscr{C}$, $\Vdash_{\mathscr{B}, \mathfrak{R}_A} \phi$ iff $\Vdash_{\mathscr{C}, \mathfrak{R}_A} \phi$. This is, however, not guaranteed to be a modal relation. 

For example, take some $\mathscr{B}'\subset \mathscr{B}$ and $\mathscr{C}'\subset\mathscr{C}$ s.t. $\mathscr{B}'\not\subseteq \mathscr{C}'$, $\mathfrak{R}_a\mathscr{B}'\mathscr{C}'$ and where $\Vdash_{\mathscr{B}', \mathfrak{R}_A} \phi$, but $\nVdash_{\mathscr{C}', \mathfrak{R}_A} \phi$. Let $\mathscr{C}'$ be the only base s.t. $\mathfrak{R}_a\mathscr{B}'\mathscr{C}'$ other than $\mathscr{B}'$ itself. So we will not have $\mathfrak{R}_a|\phi \mathscr{B}'\mathscr{C}'$.  In this case $(d)$ of Definition \ref{modalRelation} does not hold and $\mathfrak{R}_a|\phi$ is not a modal relation (see Figure \ref{SimpleDisconnect}).

\begin{figure}[ht]
\hrule
\begin{center}
\begin{tabular}{c p{.7cm} c}

\begin{tikzpicture}[-,shorten >=1pt,node distance =2cm, thick, baseline={([yshift=-1.8ex]current bounding box.center)}] 

\node[label=left:{$\phi$}] (A) {$\mathscr{B}'$};
\node[label=right:{$\cancel{\phi}$}] (B) [right of=A] {$\mathscr{C}'$};
\node[label=left:{$\phi$}] (C) [above of=A] {$\mathscr{B}$};
\node[label=right:{$\phi$}] (D) [right of=C] {$\mathscr{C}$};

\draw[-,below] (C) to node {$a$} (D);
\draw[-,above] (A) to node {$a$} (B);
\draw[dotted, right] (A) to node {} (C);
\draw[dotted, right] (B) to node {} (D);

\end{tikzpicture}

& 

$\Longrightarrow$

&

\begin{tikzpicture}[-,shorten >=1pt,node distance =2cm, thick, baseline={([yshift=-1.8ex]current bounding box.center)}] 

\node[label=left:{$\phi$}] (A) {$\mathscr{B}'$};
\node[label=right:{$\cancel{\phi}$}] (B) [right of=A] {$\mathscr{C}'$};
\node[label=left:{$\phi$}] (C) [above of=A] {$\mathscr{B}$};
\node[label=right:{$\phi$}] (D) [right of=C] {$\mathscr{C}$};

\draw[dotted, right] (A) to node {} (C);
\draw[-, below] (C) to node {$a$} (D);
\draw[dotted, right] (B) to node {} (D);

\end{tikzpicture}
\end{tabular}

\caption{\label{SimpleDisconnect} Updating $\mathfrak{R}_a$ (left) by simply disconnecting $\phi$ and non-$\phi$ bases need not result in a $\mathfrak{R}_a|\phi$ that is a modal relation (right).}
\end{center}
\vspace{1mm}
\hrule
\end{figure}

There is no clear way of forcing a relation like $\mathfrak{R}_a|\phi$ to be a modal relation without adding relations between bases that disagree on $\phi$ or by terminating relations between bases that agree on $\phi$. So, our solution to this problem is to consider different $\mathfrak{R}_a|\phi$ depending on at which base we are evaluating the announcement. Given the example above, we want a relation $\mathfrak{R}_a|\phi^\mathscr{B}$ with $\mathfrak{R}_a|\phi^\mathscr{B} \mathscr{B}\mathscr{C}$, because $\phi$ holds at both of these bases, and $\mathfrak{R}_a|\phi^\mathscr{B} \mathscr{B}'\mathscr{C}'$, so that $(d)$ holds, and a relation $\mathfrak{R}_a|\phi^{\mathscr{B}'}$ without $\mathfrak{R}_a|\phi^{\mathscr{B}'} \mathscr{B}'\mathscr{C}'$, because they disagree on $\phi$, and $\mathfrak{R}_a|\phi^{\mathscr{B}'} \mathscr{B}\mathscr{C}$, so that $(c)$, the converse of $(d)$, holds (see Figure \ref{TargettedUpdated}).

\begin{figure}[ht]
\hrule 
\begin{center}
\begin{tabular}{c p{.7cm} c}

\begin{tikzpicture}[-,shorten >=1pt,node distance =2cm, thick, baseline={([yshift=-1.8ex]current bounding box.center)}] 

\node[label=left:{$\phi$}] (A) {$\mathscr{B}'$};
\node[label=right:{$\not{\phi}$}] (B) [right of=A] {$\mathscr{C}'$};
\node[label=left:{$\phi$}] (C) [above of=A] {$\mathscr{B}$};
\node[label=right:{$\phi$}] (D) [right of=C] {$\mathscr{C}$};

\draw[-,below] (C) to node {$a$} (D);
\draw[-,above] (A) to node {$a$} (B);
\draw[dotted, right] (A) to node {} (C);
\draw[dotted, right] (B) to node {} (D);

\end{tikzpicture}

&

&

\begin{tikzpicture}[-,shorten >=1pt,node distance =2cm, thick, baseline={([yshift=-1.8ex]current bounding box.center)}] 

\node[label=left:{$\phi$}] (A) {$\mathscr{B}'$};
\node[label=right:{$\not{\phi}$}] (B) [right of=A] {$\mathscr{C}'$};
\node[label=left:{$\phi$}] (C) [above of=A] {$\mathscr{B}$};
\node[label=right:{$\phi$}] (D) [right of=C] {$\mathscr{C}$};

\draw[dotted, right] (A) to node {} (C);
\draw[dotted, right] (B) to node {} (D);

\end{tikzpicture}

\end{tabular}

\caption{\label{TargettedUpdated} Updating $\mathfrak{R}_a$ from Figure \ref{SimpleDisconnect} by the announcement of $\phi$ results in different modal relations: $\mathfrak{R}_a|\phi^\mathscr{B}$ (left) and $\mathfrak{R}_a|\phi^\mathscr{B'}$ (right). }
\end{center}
\vspace{1mm}
\hrule
\end{figure}

As it turns out, multiple modal relations can fit the bill and work as the $\mathfrak{R}_A|\phi$s we require and we do not have any reason to prefer one over the others. For that reason we treat announcement as a box-like operator that ranges over any such update. Note that, as mentioned in Section~\ref{sec:PAL-intro}, it is not uncommon to treat announcements as box-operators, however, in Kripke semantics, there is exactly one modal update that corresponds to each announcement. This is a difference between the two semantic treatments of announcements. Updates in our base-extension semantics are no longer partial functions.

We call the relations we want to consider, together with an extension of our language by propositional atoms, \emph{effective updates} (a formal definition follows in Definition \ref{def:effective}), not to be confused with the notion of successful announcements used in PAL literature.\footnote{A successful announcement is one after which the announced formula is true.} Informally, an effective update at a base is an update such that we cannot get from that base to a base at which the announced formula did not originally hold. We allow adding new propositional atoms to our bases (see Definition \ref{def:UpdateRelations}) as that guarantees that for any combination of bases reachable from the initial base, bases exist that are only their super- or subsets; that is, for bases $\mathscr{C}$ and $\mathscr{D}$ there exists a base that is the super- or subset of them, but no other base reachable from the initial base. This guarantees that there exists a set of relations that is an update set and corresponds to the update (see Definition \ref{def:EffectiveUpdateRelation} and Lemmas \ref{PALrelationIsModal} and \ref{lem:effective}).



 

We simultaneously define effective updates (Definition \ref{def:effective}) and the support relation  for PAL (Definition \ref{PALValidity}). However, we give the definitions separately to make them easier to reference. Note that, although the definition of effective updates makes use of the support relation (and vice-versa), the effective updates invoked in the definition of the support relation are on formulae of lesser height and so do not cause problems with the definitions.

As might be expected, the case for formulae of the form $[\phi]\psi$ depends on whether $\psi$ holds given effective updates by $\phi$. Accordingly, we define the support relation $\Vdash^\phi_{\mathscr{B},\mathfrak{R}_A}$ for after an announcement of $\phi$. We use $\Delta$ for a sequence of announcements. Maybe more surprisingly, the $K_a\phi$ case needs adjustment from the $S5$ version. If we are evaluating $K_a\phi$, the condition remains unchanged from the modal condition in Definition \ref{UpdateValidity}. However, if we are evaluating $K_a \phi$ after an announcement, then we cannot just consider all modal relations that are effective updates of the initial relation. We also have to check all superset bases of our initial base. This is exactly because we are no longer considering a single modal relation, but rather multiple and so condition (d) of Definition \ref{modalRelation} does not guarantee that all bases reachable by a superset of our initial base also have corresponding subset bases that are reachable from the initial base. Take the example as shown in Figure \ref{TargettedUpdated}, even though for some $\mathfrak{R}_a$ we have $\mathfrak{R}_a\mathscr{B}\mathscr{C}$, we do not necessarily have $\mathfrak{R}'_a\mathscr{B}'\mathscr{C}'$ for some corresponding $\mathfrak{R}'_a$.

Since announcements add propositional atoms, the case for $\bot$ is restricted to our initial language. Otherwise, it and the cases for propositional atoms and implication remain unchanged from their $S5$ versions, given in Definition \ref{UpdateValidity}.






\medskip 

\begin{definition}[Effective Updates]
    \label{def:effective}

    Given a set of atomic formulae $P$, a base $\mathscr{B}\in\Omega_P$, an update set  $\mathfrak{R}_A$ and a formula $\phi$, an update set $\mathfrak{R}'_A$ on $\Omega_{P'}$ is an {\emph effective update} of $\mathfrak{R}_A$ by $\phi$ at $\mathscr{B}$ iff 

    \begin{itemize}

    \item[—] $P' \supseteq P$
    
    \item[—] for every $\mathscr{C}$ reachable from $\mathscr{B}$ there is a $\mathscr{C}^+\supseteq\mathscr{C}$ s.t. $\mathscr{C}^+\in\Omega_{P'}$ and for all classical formulae $\phi$ containing only atomic formulae in $P$ $\Vdash_{C^+,\mathfrak{R}'_A} \phi$ iff $\Vdash_{C,\mathfrak{R}_A} \phi$ and

    \item[—] for all agents $a\in A$ and bases $\mathscr{C}$ and $\mathscr{D}$ reachable by $\mathfrak{R}_A$ from $\mathscr{B}$, 
    \begin{center} 
    $\mathfrak{R}'_a\mathscr{C}^+\mathscr{D}^+$ iff $\mathfrak{R}_a\mathscr{C}\mathscr{D}$, $\Vdash_{\mathscr{C},\mathfrak{R}_A} \phi$, and $\Vdash_{\mathscr{D},\mathfrak{R}_A} \phi$.
    \end{center}
    \end{itemize}

    
    Such a $\mathscr{C}^+$ is called an update of $\mathscr{C}$ for $\mathfrak{R}'_A$.

    
    We let $\mathfrak{R}_A|\phi^\mathscr{B}$ denote the set of effective updates of $\mathfrak{R}_A$ by $\phi$ at $\mathscr{B}$ and $\mathfrak{R}_A|\phi:\psi^\mathscr{B}$ the set obtained by updating the relations in $\mathfrak{R}_A|\phi^\mathscr{B}$ by $\psi$ at the corresponding updates $\mathscr{B}'$ of $\mathscr{B}$. 
    \fillBox 
\end{definition} \medskip

\medskip 

\begin{definition}[Validity]
 \label{PALValidity}
For the logic of public announcements, given a set of agents $A$, a set of atomic formulae $P$, and an update set $\mathfrak{R}_A$ consisting of $S5$-modal relations $\mathfrak{R}_a$, the support relation $\Vdash_{\mathscr{B},\mathfrak{R}_A}^{\Delta}$, where $\Delta$ is a (possibly empty) sequence of formulae, is defined as follows:

\[
\begin{array}{l@{\quad}c@{\quad}l}
\Vdash_{\mathscr{B},\mathfrak{R}_A}^{\Delta} p   & \mbox{iff} & \mbox{$p$ is in every set of basic sentences closed under $\mathscr{B}$} \\  
        & & \mbox{(i.e., iff $p\in \overline{\mathscr{B}}$)} \\
       
\Vdash_{\mathscr{B},\mathfrak{R}_A}^{\Delta} \phi\to\psi & \mbox{iff} & \mbox{$\phi\Vdash_{\mathscr{B},\mathfrak{R}_A}^{\Delta} \psi$} \\ 
\Vdash_{\mathscr{B},\mathfrak{R}_A}^{\Delta} \bot & \mbox{iff} & \mbox{$\Vdash_{\mathscr{B},\mathfrak{R}_A}^{\Delta} p$ for every basic sentence $p\in P$} \\

\Vdash_{\mathscr{B},\mathfrak{R}_A}^{\Delta} K_a \phi & \mbox{iff} & \mbox{-- if $\Delta$ is empty, then for all $\mathscr{B}'$ s.t. $\mathfrak{R}_a\mathscr{B}\mathscr{B}'$, $\Vdash_{\mathscr{B}',\mathfrak{R}_A}^{\Delta} \phi$}\\

& & \mbox{-- if $\Delta$ is non-empty, then for all $\mathscr{C}\supseteq\mathscr{B}$,  
$\mathfrak{R}_A' \in \mathfrak{R}_A|\Delta^\mathscr{C}$,}\\

& & \mbox{\,\; updates $\mathscr{C}^+$ of $\mathscr{C}$ and $\mathscr{C}'$ s.t. $\mathfrak{R}'_a \mathscr{C}^+\mathscr{C'}$, $\Vdash_{\mathscr{C}',\mathfrak{R}_A}^{\Delta} \phi$}\\

\Vdash_{\mathscr{B},\mathfrak{R}_A}^{\Delta} [\phi]\psi & \mbox{iff} &\mbox{for all $\mathscr{C}\supseteq\mathscr{B}$, if $\Vdash_{\mathscr{C},\mathfrak{R}_A}^{\Delta}\phi$, then $\Vdash_{\mathscr{C},\mathfrak{R}_A}^{\Delta , \phi} \psi$}\\

& & \\
\Gamma\Vdash_{\mathscr{B},\mathfrak{R}_A}^{\Delta} \phi & \mbox{iff} & \mbox{for all $\mathscr{C}\supseteq \mathscr{B}$, if $\Vdash_{\mathscr{C},\mathfrak{R}_A}^{\Delta} \psi$ for all $\psi \in \Gamma$, then $\Vdash_{\mathscr{C},\mathfrak{R}_A}^{\Delta} \phi$}\\ 
\end{array} 
\]

\noindent We write $\Vdash_{\mathscr{B}, \mathfrak{R}_A} \phi$ for $\Vdash_{\mathscr{B}, \mathfrak{R}_A}^\Delta \phi$ with empty $\Delta$. Given a set of agents $A$ and a set of atomic formulae $P$, a formula $\phi$ is \textit{valid}, written as $\Vdash \phi$, iff $\Vdash_{\mathscr{B}, \mathfrak{R}_A} \phi$ for all bases $\mathscr{B} \in \Omega_P$ and sets of $S5$-modal relations $\mathfrak{R}_A$. We write $\Gamma \Vdash \phi$ iff for all $\mathscr{B}$ and 
all $\mathfrak{R}_A$, if $\Vdash_{\mathscr{B}, \mathfrak{R}_A} \psi$, for all $\psi \in \Gamma$, then $\Vdash_{\mathscr{B}, \mathfrak{R}_A} \phi$.  
\fillBox 
\end{definition} \medskip

\begin{figure}[h]
\hrule 
\begin{center}
\begin{tikzpicture}[-,shorten >=1pt,node distance =2cm, thick, baseline={([yshift=-1.8ex]current bounding box.center)}] 

\node[label=left:{$\phi$}] (A) {$\mathscr{B}'$};
\node[label=right:{$\xcancel{\phi}$}] (B) [right of=A] {$\mathscr{C}'$};
\node[label=left:{$\phi$}] (C) [above of=A] {$\mathscr{B}$};
\node[label=right:{$\phi$}] (D) [right of=C] {$\mathscr{C}$};

\draw[-,below] (C) to node {$a$} (D);
\draw[-,above] (A) to node {$a$} (B);
\draw[-, right] (A) to node {$b$} (C);
\draw[dotted, right] (B) to node {} (D);

\end{tikzpicture}
\end{center}
\caption{\label{fig:FailedUpdate} An example of how updates can fail when $\mathfrak{R}_A$ is not an update sets.}
\vspace{1mm}
\hrule
\end{figure}

Here our definition of update sets in Definition \ref{def:UpdateRelations} becomes important. Take the set-up in Figure \ref{fig:FailedUpdate} that can arise using just modal relations rather than full update sets.

If we were to update $\mathscr{B}$ by $\phi$, there would be no subset of $\mathscr{C}$ to which $\mathscr{B'}$ can reasonably have a relation for agent $a$ to, as a relation between $\mathscr{B}'$ and $\mathscr{C}$ might change what formulae are supported at these bases (and at $\mathscr{B}$ for that matter) in ways that not directly relate to the update. If, however, $\mathscr{B}'$ does not have access to some subset of $\mathscr{C}$, condition (d) no longer holds and we no longer have a modal relation.

 To show that there always is at least one such effective update of $\mathfrak{R}_A$ by $\phi$ at a base $\mathscr{B}$ whenever $\Vdash_{\mathscr{B},\mathfrak{R}_A}\phi$, we define a construction that results in such a relation.

We define a construction of a relation after an announcement takes place. We do so by first fixing the relations between bases reachable from our initial base and then fixing the relations between subsets and supersets of those reachable bases in such a way that guarantees that the resulting set of relations is an update set. 

To guarantee that we can always do that, we choose our extension of the language and the corresponding bases after the announcement in such a way that for every combination of bases reachable from the base at which the update takes place there exist bases that are their subsets (supersets) but which are not subsets (supersets) of the other reachable bases. This means that, for example, if the update takes place at a base $\mathscr{B}$ and $\mathscr{C}$ and $\mathscr{D}$ are reachable from $\mathscr{B}$, then there exists a base that is the subset (superset) of both $\mathscr{C}^+$ and $\mathscr{D}^+$ but of no other $\mathscr{E}^+$ with $\mathscr{E}$ reachable from  $\mathscr{B}$. This gives us the the ability to connect subsets and supersets of the reachable bases exactly in accordance with how they are connected.



We start by defining relations restricted to the bases reachable from the base at which the effective update is taking place and carry out the update on those relations. We then restrict these relations again to ones reachable from the initial base as we have removed relations and so the bases reachable from the initial base have changed. We also define a set of all bases in the range and domain of a relation. This allows us to identify which bases are relevant for evaluating formulae at our initial base and will be used to construct the effective update.

\medskip 

\begin{definition}
    \label{def:UpdateS}

Given a set of atomic formulae $P$, an update set $\mathfrak{R}_A$, a consistent base $\mathscr{B}$ and a formula $\phi$ such that  $\Vdash_{\mathscr{B},\mathfrak{R}_A} \phi$ 

\begin{itemize}
    \item[-] $S_a$ is $\mathfrak{R}_a$ restricted to bases reachable from $\mathscr{B}$ by $\mathfrak{R}_A$
    \item[-] $S_a|\phi \mathscr{C}\mathscr{D}$ iff $S_a \mathscr{C}\mathscr{D}$, $\Vdash_{\mathscr{C},\mathfrak{R}_a} \phi$ and $\Vdash_{\mathscr{D},\mathfrak{R}_a} \phi$ 
    \item[-] $S_A|\phi$ is the set of $S_a|\phi$,  for all $a\in A$, as normal
    \item[-] $S^*_a|\phi$ is $S_a|\phi$ restricted to bases reachable from $\mathscr{B}$ by $S_A|\phi$ 
    \item[-] $S^*_A|\phi$, again, is the set of all $S^*_a|\phi$ 
    \item[-] $\overline{S^*_A|\phi}$ is the set of all bases in the range and domain of the relations in $S^*_A|\phi$, that is, for all $\mathscr{B}$, $\mathscr{B}\in \overline{S^*_A|\phi}$ iff there exists a $\mathscr{C}$ s.t. $S^*_A|\phi \mathscr{B}\mathscr{C}$ or $S^*_A|\phi\mathscr{C}\mathscr{B}$. 
\end{itemize}
\vspace{-7mm}
\fillBox 
\end{definition}  \medskip

\noindent Next, we define the extension of our language with a $p_\emptyset$ and a $p_i$ for every $\pi_i$ in a $\pi(\mathscr{B})$ (see Definition~\ref{def:UpdateRelations}).

\medskip 

\begin{definition}
    \label{def:UpdateP}
Given a set of atomic formulae $P$, an update set of relations $\mathfrak{R}_A$, a consistent base $\mathscr{B}$ and a formula $\phi$ s.t. $\Vdash_{\mathscr{B},\mathfrak{R}_A} \phi$, for some finite $\pi(\mathscr{B})= \{\pi_1,\dots, \pi_n\}$, where 
\begin{itemize}
        \item[—]for all $i$, $\pi_i\subseteq \Omega$
        \item[—]for all $i$, and $\mathscr{C},\mathscr{D}\in \pi_i$, $\mathscr{D}$ is reachable from $\mathscr{C}$ using the sub/superset relation
        \item[—]for all $\mathscr{C}$, $\mathscr{C}$ is reachable from $\mathscr{B}$ using $\mathfrak{R}_A$ iff there is a $\pi_i$ s.t. $\mathscr{C}\in \pi_i$
\end{itemize}

let $P^+ = P \cup \{p_\emptyset\}\cup\{p_1,\dots,p_n\}$. Note that we know such a $\pi(\mathscr{B})$ exists as $\mathfrak{R}_A$ is an update set. \fillBox
\end{definition} \medskip

We use this extended language to define the bases that are going to be the updates of the initial bases ensuring that each combination of these bases share rules that are not shared with any other such base.

\medskip 

\begin{definition}
    \label{def:UpdateBases}

Given a set of atomic formulae $P$, an update set of relations $\mathfrak{R}_A$, a consistent base $\mathscr{B}$, and a formula $\phi$ such that $\Vdash_{\mathscr{B},\mathfrak{R}_A} \phi$, for any base $\mathscr{C}$ reachable from $\mathscr{B}$ by $S^*_A|\phi$, let $\mathscr{C}^+ = \mathscr{C}\cup \{p'_{0},\dots, p'_i\Rightarrow p_\emptyset: i\geq 0 \text{ and for all } j\leq i \text{ and $k$ s.t.  } \mathscr{C} \in \pi_k, p'_j \in \{p_1,\dots,p_n\} \setminus \{p_k\}\} \cup
\{q\Rightarrow p_i: q\in P \text{ and }p_i \in \{p_1,\dots,p_n\} \cup \{p_\emptyset\}\}$.
\fillBox 
\end{definition} \medskip

Using these update bases, we define a function $\delta_{\pi^+}$ that tells us for any base with which of these $\mathscr{B}^+$ they share subsets.

\medskip 

\begin{definition}
    \label{def:UpdateDelta}
Given a set of atomic formulae $P$, an update set of relations $\mathfrak{R}_A$, a consistent base $\mathscr{B}$ and a formula $\phi$ s.t. $\Vdash_{\mathscr{B},\mathfrak{R}_A} \phi$,
    
For all $\pi_i \in \pi(\mathscr{B})$, let $\mathscr{\pi}_i^+ = \{\mathscr{C}: \mathscr{C} \in \pi_i \text{ and } \mathscr{C}\in \overline{S^*_a|\phi}\}$ and let $\pi^+= \{\pi^+_i : \pi_i\in \pi \text{  and }\pi^+_i\neq\emptyset\}$.

For a base $\mathscr{C} \in \Omega_{P^+}$,  
\begin{center}
    $\delta_{\pi^+}(\mathscr{C}) =\left\{ \{\pi^+_1,\dots,\pi^+_j\}\middle\vert \begin{array}{l} 
    \{\pi^+_1,\dots,\pi^+_j\}\subseteq \pi^+ \text{ and there is a } \mathscr{C}'\subseteq\mathscr{C}\\ \text{ s.t. for all }\pi^+_i \in \{\pi^+_1,\dots,\pi^+_j\}, \text{there is a} \\  \mathscr{D}\in \pi_i \text{ and } \mathscr{C}'\subseteq\mathscr{D} 
    \end{array}\right\}$
\end{center}
\vspace{-7mm}
\fillBox 
\end{definition} \medskip

Finally, using the notation established in Definitions \ref{def:UpdateS} - \ref{def:UpdateDelta}, we construct an effective update of an update set $\mathfrak{R}_A$ by a formula $\phi$ at a base $\mathscr{B}$ called $\mathfrak{U}^\mathscr{B}_A|\phi$.

\medskip 

\begin{definition}
    \label{def:EffectiveUpdateRelation}

Given a set of atomic formulae $P$, an update set of relations $\mathfrak{R}_A$, a consistent base $\mathscr{B}$ and a formula $\phi$ s.t. $\Vdash_{\mathscr{B},\mathfrak{R}_A} \phi$, let $\mathfrak{U}_A^\mathscr{B}|\phi$ be the set of relations such that, for every agent $a\in A$, $\mathfrak{U}_a^\mathscr{B}|\phi \mathscr{C}\mathscr{D}$ only by the following: 

\begin{enumerate}[label=(\arabic{enumi})]
    \item if $S^*_a|\phi\mathscr{C}\mathscr{D}$, then $\mathfrak{U}_a^\mathscr{B}|\phi\mathscr{C}^+\mathscr{D}^+$
    \item if $\mathscr{C}$ and $\mathscr{D}$ are inconsistent, then $\mathfrak{U}^\mathscr{B}_a|\phi\mathscr{C}\mathscr{D}$
       
    \item If $\mathscr{C}$ and $\mathscr{D}$ \begin{itemize}
        \item[—]are consistent
        
        \item[—]are both not in f $\overline{S^*_A|\phi}$
        
        \item[—]there is an $\mathscr{E}^+ \in \overline{S^*_A|\phi}$ s.t. $\mathscr{C}\subseteq\mathscr{E}^+$ iff there is an $\mathscr{F}^+ \in \overline{S^*_A|\phi}$ s.t. $\mathscr{D}\subseteq\mathscr{F}^+$

        \item[—]there is an $\mathscr{E}^+ \in \overline{S^*_A|\phi}$ s.t. $\mathscr{C}\supseteq\mathscr{E}^+$ iff there is an $\mathscr{F}^+ \in \overline{S^*_A|\phi}$ s.t. $\mathscr{D}\supseteq\mathscr{F}^+$ 

        \item[—]there is a rule $r\in\mathscr{C}$ s.t. $r\notin\mathscr{E}^+$ for all $\mathscr{E}^+ \in \overline{S^*_A|\phi}$ iff there is a rule $r'\in\mathscr{D}$ s.t. $r'\notin\mathscr{E}^+$ for all $\mathscr{E}^+ \in \overline{S^*_A|\phi}$
        
        \item[—]there are bijective functions $f: \delta_{\pi^+}(\mathscr{C}) \to \delta_{\pi^+}(\mathscr{D})$ and $g: \overline{S^*_A|\phi}\rightarrow \overline{S^*_A|\phi}$ s.t. for every $\delta \in \delta_{\pi^+}(\mathscr{C)}$ and $\mathscr{C}\in \delta$, $g(\mathscr{C})\in f(\delta)$ and $S^*_a|\phi \mathscr{C} g(\mathscr{C})$  
    \end{itemize}
        then $\mathfrak{U}^\mathscr{B}_a|\phi\mathscr{C}\mathscr{D}$. 
\end{enumerate}
\vspace{-8mm}\fillBox 
\end{definition} 

\medskip

In Appendix~\ref{App2}, we show $\mathfrak{U}_A^\mathscr{B}|\phi$ is, in fact, an effective update. We do so by showing that, for all $a \in A$, $\mathfrak{U}_a^\mathscr{B}|\phi$ is an $S5$-modal relation according to Definition~\ref{modalRelation} (i.e., conditions (a)--(d) hold and it is an equivalence relation) and that the conditions for update sets of Definition~\ref{def:UpdateRelations} hold for $\mathfrak{U}_A^\mathscr{B}|\phi$ in Lemma \ref{PALrelationIsModal}. We then show that it is an effective update in the sense of Definition~\ref{def:effective} in Lemma~\ref{lem:effective}. It follows that whenever $\Vdash_{\mathscr{B},\mathfrak{R}_A} \phi$ there exists an effective update of $\mathfrak{R}_A$ by $\phi$ at $\mathscr{B}$ (namely, $\mathfrak{U}_A^\mathscr{B}|\phi$).

\begin{theorem}
    For any $\mathfrak{R}_A$, $\mathscr{B}$, and $\phi$ s.t. $\Vdash_{\mathscr{B},\mathscr{R}_A} \phi$, there is an effective update of $\mathfrak{R}_A$ by $\phi$ at $\mathscr{B}$. \fillBox
\end{theorem} 

\medskip 
    
    Another important insight about effective updates is that they are equivalent when it comes to their judgements about formulae restricted to our initial $P$. Lemma~\ref{lem:EquivalenceOfEffectiveUpdates}, which comes next, makes this precise, but is not required for our technical development. However, it is useful when considering examples, as in Section~\ref{sec:ExamplesRevisited}, as it means that we only ever have to consider one effective update when checking whether an announcement formula is supported at a base.    
    
\medskip 

    \begin{restatable}{lemma}{EquivalenceEffUp}
        \label{lem:EquivalenceOfEffectiveUpdates}
        Consider any sequence of formulae $\Delta$, base $\mathscr{B}$, update relations $\mathfrak{R}_A$, and effective updates of $\mathfrak{R}_A$ by $\Delta$ at $\mathscr{B}$ $\mathfrak{R}^1$ and $\mathfrak{R}^2$, with $\mathscr{B}^1$ and $\mathscr{B}^2$ the updates of $\mathscr{B}$, respectively, for any formula $\chi$ restricted to $P$. Then  
        \[
        \Vdash_{\mathscr{B}^1,\mathfrak{R}^1}\chi \text{ iff } \Vdash_{\mathscr{B}^2,\mathfrak{R}^2}\chi    
        \]
    \end{restatable}
    \begin{proof}
    This proof goes by a simple induction on $\chi$. The details can be found in Appendix~\ref{App2}.
    %
    %
    %
    %
    \end{proof}

We now move to show soundness and completeness for our system. We start by showing that the axioms and rules of $S5$ modal logic still hold as they did before adding announcements.

For this we need to adept some lemmas from the modal case from $S5$ and use these to show, in order, that the axioms of classical logic and (MP), the axioms of $S5$ modal logic and (NEC), and, finally, the axioms of PAL and (NOA) hold in semantics. This is largely the same strategy as was employed in \cite{EckhardtPymS5} only with the new cases added for the axioms of PAL and (NOA). The details can be found in Appendix~\ref{App2}.

\begin{restatable}{theorem}{PALsound}
    \label{theorem:Soundness}
    For all $\phi$, $\vdash \phi$ implies $\Vdash \phi$. 
\end{restatable}
\vspace{-7.5mm}\fillBox

\vspace{4mm}


We establish completeness in a known way for PAL by translating PAL formulae to the language of $S5$. We start by defining the necessary translation function based on the axioms for PAL and then show that for every formula in PAL there is an equivalent formula in S5. The result then follows from the Completeness of $S5$ and the Soundness of PAL just shown. This proof is not specific to our semantics and is taken from \cite{Ditmarsch2007}. The details can be found in Appendix~\ref{App2}.   

In fact, it is a known property of PAL without common knowledge that it is essentially a part of $S5$ just expressed differently and so, this completeness proof based on the translation can be used regardless of the particular semantics. 

\begin{restatable}{theorem}{PALcomp}
    \label{PALComp}

    For all formulae of $PAL$, $\Vdash \phi$ implies $\vdash \phi$.

\end{restatable}
  \vspace{-7.5mm}  \fillBox

\vspace{4mm}

Note that we have only shown weak completeness here. As, in \cite{EckhardtPymS5}, we have only proven shown weak completeness for the underlying $S5$ modal logic our proof here that relies on translation to that language cannot be easily extended to strong completeness without first proving it for $S5$. Alternatively, one could seek to give a proof closer to Sandqvist's proof of completeness in Be-S for intuitionistic propositional logic in \cite{Sandqvist2015}, as done for intuitionistic modal logic in \cite{BuzokuPym-Modal2026}. This remains an interesting open question considering that it holds for the Kripke semantics for PAL. 

\section{Examples of Public Announcements Revisited}
\label{sec:ExamplesRevisited}

We have established a B-eS for PAL, but it is important to see how this semantics can be used to analyze the connection between announcements and information. To this end, we discuss the two example puzzles, we have already given in Sections \ref{subsec:standard3pcg} and \ref{subsec:standardmcp}, respectively. 

We have chosen these two examples because they highlight two different aspects of how information is treated differently in the inferentialist set-up of base-extension semantics compared to the modal theoretic view of Kripke semantics. In both cases, we aim to make implicit modelling choices more explicit following \cite{Brandom1994,brandom2009}.


Concretely, in the card game example,  Section~\ref{subsec:card-game-revisted}, we highlight how base rules can be used intensionally to construct bases that best correspond to the situation one is trying to model, in contrast to the extensional character of worlds and valuations in Kripke semantics. In our discussion, in Section~\ref{subsec:muddy-children-revisted}, of the muddy children puzzle, we focus on investigating exactly what base rules are actually necessary to include in bases in order for the example to play out in the expected way or, phrased differently, consider what information do the children need in order to be able to find out that they are muddy. 
In both cases, we give what we call minimal set-ups for a base-extension modelling of the examples. We use minimal here in the rather weak sense that any extension of the bases will result in the desired results. We do not make the stronger claim that these are the minimal possible requirements and any weaker set-up will fail. Although such a result would be desirable, it goes beyond the scope of this paper. 



In Section~\ref{sec:Intensionality}, we reflect in on the differences between the view of the examples given by Kripke semantics and that given by base-extension semantics, stressing how the intensional nature of the semantics enables direct characterizations of the informational content of the games.

\subsection{Three-player Card Game} \label{subsec:card-game-revisted}

As we have seen in Subsection \ref{subsec:standard3pcg}, the Kripke semantics modeling for the card game example (Example \ref{ex:cardgame1}) lines up exactly with our informal interpretation. There is, however, one aspect of the modeling we glossed over and that is the implementation of the rules of the game. Obviously, we want to show that we can model this example in our base-extension semantics and get the same intuitive results as the Kripke semantics. Given the equivalence between maximally-consistent bases and worlds in Kripke semantics, it should come as no surprise that simply taking maximally-consistent bases that correspond to the worlds used in the Kripke models works. 

However, we can also give a more satisfying treatment of the rules of the card game, by embedding them directly into bases. For this reason, we want to construct bases with as few rules as possible, resulting in a minimal model  in the sense given above. This will allow us to make the stronger claim that the necessary reasoning by the agents will work on any model for which these minimal restrictions hold.

As discussed above the basic rules of the card game are that (a) every card is held by exactly one agent and (b) every agent holds exactly one card. We start by constructing a base at which `at most' versions of (a) and (b) hold. This base will be used as a subset for our final bases.

What we mean by `every card is held by at most one agent' is that any base the players are considering at which a card is held by more than one agent must be inconsistent. So, for example, we want $(0_a \wedge 0_b) \to \bot$ to hold at our base. Similarly, for `every agent holds at most one card', we require $(0_a \wedge 1_a)\to\bot$. So we want any combination of two agents holding the same cards and two cards being held by the same agent to lead to contradiction, that is to say cause every propositional atom to hold. We can straightforwardly create rule-schema that cover all those cases. We call the base for which this holds $\mathscr{B}_{cards}$ and give a formal definition as part of Proposition \ref{prop:cardgame}.

We cannot give similar base rules to establish the `at least' parts of (a) and (b) directly. Given the form of our base rules, as given in Definition \ref{ClassicalRule}, it is not possible to give rules to that effect. However, similarly to the Kripke semantics version, we take bases at which the agents have each drawn a card and each card has been drawn. 


There is one more `rule' of the card game we have not yet discussed: Every player can only see their own hand. This is a modal rule that directly impacts what information the players can perceive. As such, this is implemented by our choice of $\mathfrak{R}_A$. Similarly to the Kripke semantics treatment, we choose an $\mathfrak{R}_a$ at which Anne only considers bases at which they drew the same card. The same goes for Bob and Cath.

We do not claim that all rules for any game that can be modelled using PAL can be implemented either as a base rule or as a modal rule in constructing the relevant $\mathfrak{R}_A$, as we have seen with the `at most' versions of (a) and (b). We do however, argue that in cases in which we can this provides a more faithful representation of the example as compared to the abstract choice of valuations in Kripke semantics.


 

    


\medskip 

\begin{prop}
\label{prop:cardgame}

     For all $M, N, \in \{0,1,2\}$ with $M\neq N$, $i, j \in \{a, b, c\}$ with $i\neq j$, and $p$, 
 
 \begin{center}
     $\mathscr{B}_{cards} = \{(N_i, N_j \rightarrow p), (N_i, M_i \rightarrow p)\}$.
 \end{center}
    Let $\mathscr{B}_{MNP}$ be any consistent superset of $\mathscr{B}_{cards}$ s.t. $\Vdash_{\mathscr{B}_{MNP}} M_a \wedge N_b \wedge P_c$.
   For any update set $R_{A}$ with $A=\{a,b,c\}$ s.t. there are 
\begin{center}   $\mathscr{B}_{012}, \mathscr{B}_{021}, \mathscr{B}_{102}, \mathscr{B}_{120}, \mathscr{B}_{201}$ and $\mathscr{B}_{210}$ 
\end{center}

\noindent with 

    \begin{center}
\begin{tabular}{lcr}
   $\mathfrak{R}_a \mathscr{B}_{012} \mathscr{B}_{021}$,  & $\mathfrak{R}_a \mathscr{B}_{102} \mathscr{B}_{120}$, & $\mathfrak{R}_a \mathscr{B}_{201} \mathscr{B}_{210}$, \\
       $\mathfrak{R}_b \mathscr{B}_{102} \mathscr{B}_{201}$, & $\mathfrak{R}_b \mathscr{B}_{012} \mathscr{B}_{210}$, & $\mathfrak{R}_b \mathscr{B}_{021} \mathscr{B}_{120}$,\\
$\mathfrak{R}_c \mathscr{B}_{120} \mathscr{B}_{210}$, & $\mathfrak{R}_c \mathscr{B}_{021} \mathscr{B}_{201}$, & $\mathfrak{R}_c \mathscr{B}_{012} \mathscr{B}_{102}$,
\end{tabular}
\end{center}
and for which these, together with the necessary relations to make these $S5$ (i.e., reflexive, transitive, and Euclidean), are the only relations involving these bases, 
  \begin{center}
      $\Vdash_{\mathscr{B}_{012}, \mathfrak{R}_{A}} [\neg 1_a] K_c(0_a\wedge 1_b\wedge 2_c)$. 
  \end{center}
\end{prop}

Note that besides building the subset of the bases we require to model the card game example, $\mathscr{B}_{cards}$ also represents the state of the game before any card has been drawn exactly because it directly corresponds to the rules of the game.

By Definition \ref{PALValidity}, $\Vdash_{\mathscr{B}_{012}, \mathfrak{R}_{A}} [\neg 1_a] K_c(0_a\wedge 1_b\wedge 2_c)$ iff for all $\mathscr{C}_{012}\supseteq\mathscr{B}_{012}$, if $\Vdash_{\mathscr{C}_{012}, \mathfrak{R}_{A}} \neg 1_a$, then $\Vdash_{\mathscr{C}_{012}^+, \mathfrak{R}_{A}}^{\neg 1_a} K_c (0_a\wedge 1_b \wedge 2_c)$. As $\mathfrak{R}_c$ is a modal relation we know that any $\mathscr{D}$, s.t. $\mathfrak{R}_c \mathscr{C}_{012} \mathscr{D}$, has to be a superset of $\mathscr{B}_{012}$ or $\mathscr{B}_{102}$, by condition (d) of Definition \ref{modalRelation}. We call $\mathscr{D}_{012}$ those $\mathscr{D}\supseteq\mathscr{B}_{012}$  and $\mathscr{D}_{102}$ those $\mathscr{D}\supseteq\mathscr{B}_{102}$. We cannot have any $\mathfrak{R}'_A \in \mathfrak{R}_A|1_a^{\mathscr{C}_{012}}$ with $\mathfrak{R}'_c\mathscr{C}^+_{012}\mathscr{D}^+_{102}$, as $\nVdash_{\mathscr{D}_{102}, \mathfrak{R}_{A}} \neg 1_a$. So we only need to consider $\mathscr{D}_{012}$. As $\mathscr{D}_{012} \supseteq\mathscr{B}_{012}$, we have $\Vdash_{\mathscr{D}^+_{012}, \mathfrak{R}_{A}}^{1_a} 0_a \wedge 1_b \wedge 2_c$ and so also $\vdash_{\mathscr{B}_{012}, \mathfrak{R}_{A}} [\neg 1_a] K_c(0_a\wedge 1_b\wedge 2_c)$. So, the example plays out in exactly the same way as we expected it to.


  In fact,  the basic structure of the model remains largely unchanged from the Kripke semantics version illustrated in Figure \ref{KripkeCardgame1}. But  now we have to consider supersets and sometimes multiple such supersets for the same minimal base (e.g., $\mathscr{B}_{012}$). However, due to condition (d) of modal relations, we know that no other bases are reachable by $\mathfrak{R}_{A}$ and so the example plays out in the same way. We illustrate this in Figure \ref{BeSCardgame}.

\begin{figure}[ht]
\hrule 
\begin{center}
\begin{tabular}{cc}

\begin{tikzpicture}[-,shorten >=1pt,node distance =2cm, thick, baseline={([yshift=-1.8ex]current bounding box.center)}] 

\node (A) at (0,1) {$\{102\}$};
\node (B) at (1,0) {$\{201\}$};
\node (C) at (1,2) {$\mathscr{C}_{012} , \{012\}$};
\node (D) at (4,1) {$\{120\}$};
\node (E) at (3,0) {$\{210\}$};
\node (F) at (3,2) {$\{021\}$};

\draw[-,near start, above] (A) to node {$a$} (D);
\draw[-,above] (B) to node {$a$} (E);
\draw[-,above] (C) to node {$a$} (F);

\draw[-,above] (A) to node {$b$} (B);
\draw[-,near start, above] (C) to node {$b$} (E);
\draw[-,above] (D) to node {$b$} (F);

\draw[-,above] (A) to node {$c$} (C);
\draw[-,near end, above] (B) to node {$c$} (F);
\draw[-,above] (D) to node {$c$} (E);

\end{tikzpicture}

&

\begin{tikzpicture}[-,shorten >=1pt,node distance =2cm, thick, baseline={([yshift=-1.8ex]current bounding box.center)}] 

\node (B) at (1,0) {$\{201\}^+$};
\node (C) at (1,2) {$\mathscr{C}^+_{012} , \{012\}^+$};
\node (E) at (3,0) {$\{210\}^+$};
\node (F) at (3,2) {$\{021\}^+$};

\draw[-,above] (B) to node {$a$} (E);
\draw[-,above] (C) to node {$a$} (F);

\draw[-,near start, above] (C) to node {$b$} (E);

\draw[-,near end, above] (B) to node {$c$} (F);

\end{tikzpicture}

\end{tabular}

\caption{\label{BeSCardgame} The structure of $\mathfrak{R}_{A}^{\mathscr{C}_{012}}$ (left) and any $\mathfrak{R}'_A \in \mathfrak{R}_{A}|\neg 1_a^{\mathscr{C}_{012}}$ (right) for bases $\mathscr{C}_{012}\supseteq\mathscr{B}_{012}$ with $\{012\}$ standing for supersets of $\mathscr{B}_{012}$ and $\{012\}^+$ for their updates etc.}
\end{center}
\vspace{1mm}
\hrule
\end{figure}

\subsection{The Muddy Children Puzzle} 
\label{subsec:muddy-children-revisted}

We model the muddy children puzzle (Example \ref{ex:muddychildren}) for three children in which two, Anne and Bob, are muddy in the base-extension semantics for PAL. As with the card game example, we want to give a minimal example. Unlike the card game, however, we do not have any underlying rules that can be implemented into base rules. We simply have one base for each of the possible combinations of muddy children. The fact that the children cannot see their own forehead is, like the players only being able to see their own card, expressed by the corresponding modal relations between those bases.

We use the idea of using minimal bases to once again show the necessary information to get the wanted result. The question we want to investigate is exactly what base rules are necessary for the muddy children to be able to draw the expected conclusions about which of them are, in fact, muddy. For bases in which a specific child, say Anne, is muddy it turns out that it is not enough to for `Anne is muddy' to not be supported, but rather we require for Anne to not be muddy (i.e., the negation has to be supported). We show this by showing that Proposition \ref{prop:MuddyChildren1}, in which we do not make this requirement, fails.




\medskip 

\begin{prop}
\label{prop:MuddyChildren1}
   For a set of agents $A = \{a,b,c\}$ and any $N \in P(A)$ (i.e., the powerset of $A$), let $\mathscr{B}_N$ be any consistent base such that for all $j\in A$ with $m_j \in N$,  $\Vdash_{\mathscr{B}_{N}, \mathfrak{R}_{A}} m_j$.

   For any update set $R_{A}$ with $A=\{a,b,c\}$ s.t. there are
   
   $\mathscr{B}_\emptyset, \mathscr{B}_{\{a\}}, \mathscr{B}_{\{b\}}, \mathscr{B}_{\{c\}}, \mathscr{B}_{\{a,b\}}, \mathscr{B}_{\{a,c\}}, \mathscr{B}_{\{b,c\}}$ and $\mathscr{B}_{\{a,b,c\}}$ with 
\begin{center}
    
\begin{tabular}{lccr}
   $\mathfrak{R}_a \mathscr{B}_{\emptyset} \mathscr{B}_{\{a\}}$,  & $\mathfrak{R}_a \mathscr{B}_{\{b\}} \mathscr{B}_{\{a,b\}}$, & $\mathfrak{R}_a \mathscr{B}_{\{c\}} \mathscr{B}_{\{a,c\}}$, & $\mathfrak{R}_a \mathscr{B}_{\{b,c\}} \mathscr{B}_{\{a,b,c\}}$,  \\
   
       $\mathfrak{R}_b \mathscr{B}_{\emptyset} \mathscr{B}_{\{b\}}$, & $\mathfrak{R}_b \mathscr{B}_{\{a\}} \mathscr{B}_{\{a,b\}}$, & $\mathfrak{R}_b \mathscr{B}_{\{c\}} \mathscr{B}_{\{b,c\}}$, & $\mathfrak{R}_b \mathscr{B}_{\{a, c\}} \mathscr{B}_{\{a, b,c\}}$, \\

$\mathfrak{R}_c \mathscr{B}_{\emptyset} \mathscr{B}_{\{c\}}$, & $\mathfrak{R}_c \mathscr{B}_{\{a\}} \mathscr{B}_{\{a,c\}}$, & $\mathfrak{R}_c \mathscr{B}_{\{b\}} \mathscr{B}_{\{b,c\}}$, & $\mathfrak{R}_c \mathscr{B}_{\{a,b\}} \mathscr{B}_{\{a,b,c\}}$,\\
     
\end{tabular}
   
\end{center}
and for which those, together with the necessary relations to make these $S5$ (i.e., reflexive, transitive, and Euclidean), are the only relations involving these bases, we have 

  \begin{center}
      $\Vdash_{\mathscr{B}_{\{a,b\}}, \mathfrak{R}_{A}} [(m_a\vee m_b\vee m_c)] [\neg (K_a m_a\vee K_a \neg m_a) \wedge \neg (K_b m_b\vee K_b \neg m_b) \wedge \neg (K_c m_c\vee K_c \neg m_c)] (K_a m_a \wedge K_b m_b) $. 
      
  \end{center}

\end{prop}

To show that Proposition \ref{prop:MuddyChildren1} fails, we give a counterexample. To make things simpler and avoid any considerations of the sub- and superset relations, for every $\mathscr{B}_N$, let $p_{N}$ be a propositional atom s.t. the base rule $(\Rightarrow p_N)\in \mathscr{B}_N$ and, for all $\mathscr{B}_M$ s.t. $M\neq N$, $(\Rightarrow p_N)\notin \mathscr{B}_M$. To further keep the example as simple as possible, we add simply the rules corresponding to $(\Rightarrow m_i)$ for all $i\in N$ to a base $\mathscr{B}_N$. So, for example, $\mathscr{B}_{\{a,b\}} = \{ (\Rightarrow p_{\{a,b\}}), (\Rightarrow m_a), (\Rightarrow m_b)\}$. We also still let $\phi = (m_a\vee m_b\vee m_c)$ and $\psi= \neg (K_a m_a\vee K_a \neg m_a) \wedge \neg (K_b m_b\vee K_b \neg m_b) \wedge \neg (K_c m_c\vee K_c \neg m_c)$.

For $\mathscr{B}_\emptyset$, we are taking something a bit different. Remember that the conditions above are minimal conditions and so we can add base rules to our bases unless they contradict them directly. So let $\mathscr{B}_\emptyset = \{(\Rightarrow p_\emptyset), (p_\emptyset\Rightarrow m_a)\}$. 

Given these bases it is easy to show that the example will not play out in the way we expect it to. Note that now $\Vdash_{\mathscr{B}_\emptyset, \mathfrak{R}_A} m_a$ and so the father's initial announcement of $\phi= m_a \vee m_b \vee m_c$ --- that at least one child is muddy --- fails to eliminate the base $\mathscr{B}_\emptyset$ (see Figure \ref{B-eSMC1}).

\begin{center}
\begin{figure}[h]
\hrule 
\begin{tabular}{c p{.7cm} c}
\begin{tikzpicture}[-,shorten >=1pt,node distance =1cm, thick, baseline={([yshift=-1.8ex]current bounding box.center)}] 

\node (A) at (0,0) {$\mathscr{B}_\emptyset$};

\node (C) at (0,3) {$\mathscr{B}_{\{b\}}$};
\node (D) at (1.5,1.5) {$\mathscr{B}_{\{c\}}$};
\node (B) at (3,0) {$\mathscr{B}_{\{a\}}$};

\node[text width=1.4cm, text centered] (E) at (3,3) {$\mathscr{B}_{\{a,b\}}$};
\node[text width=1.4cm, text centered] (F) at (4.5,1.5) {$\mathscr{B}_{\{a,c\}}$};
\node[text width=1.4cm, text centered] (G) at (1.5,4.5) {$\mathscr{B}_{\{b,c\}}$};

\node[text width=1.4cm, text centered] (H) at (4.5,4.5) {$\mathscr{B}_{\{a,b,c\}}$};

\draw[-,above] (A) to node {$a$} (B);
\draw[-,left] (A) to node {$b$} (C);
\draw[-,above] (A) to node {$c$} (D);

\draw[-, near start,left] (B) to node {$b$} (E);
\draw[-,above] (B) to node {$c$} (F);

\draw[-,near start, above] (C) to node {$a$} (E);
\draw[-,above] (C) to node {$c$} (G);

\draw[-,near start, above] (D) to node {$a$} (F);
\draw[-,near start, left] (D) to node {$b$} (G);

\draw[-,above] (E) to node {$c$} (H);
\draw[-,right] (F) to node {$b$} (H);
\draw[-,above] (G) to node {$a$} (H);

\end{tikzpicture}

&

$\Longrightarrow$

&

\begin{tikzpicture}[-,shorten >=1pt,node distance =1cm, thick, baseline={([yshift=-1.8ex]current bounding box.center)}] 

\node (A) at (0,0) {$\mathscr{B}_\emptyset^+$};

\node (C) at (0,3) {$\mathscr{B}_{\{b\}}^+$};
\node (D) at (1.5,1.5) {$\mathscr{B}_{\{c\}}^+$};
\node (B) at (3,0) {$\mathscr{B}_{\{a\}}^+$};

\node[text width=1.4cm, text centered] (E) at (3,3) {$\mathscr{B}_{\{a,b\}}^+$};
\node[text width=1.4cm, text centered] (F) at (4.5,1.5) {$\mathscr{B}_{\{a,c\}}^+$};
\node[text width=1.4cm, text centered] (G) at (1.5,4.5) {$\mathscr{B}_{\{b,c\}}^+$};

\node[text width=1.4cm, text centered] (H) at (4.5,4.5) {$\mathscr{B}_{\{a,b,c\}}^+$};

\draw[dashed,above] (A) to node {$a$} (B);
\draw[dashed,left] (A) to node {$b$} (C);
\draw[dashed,above] (A) to node {$c$} (D);

\draw[-, near start,left] (B) to node {$b$} (E);
\draw[-,above] (B) to node {$c$} (F);

\draw[-,near start, above] (C) to node {$a$} (E);
\draw[-,above] (C) to node {$c$} (G);

\draw[-,near start, above] (D) to node {$a$} (F);
\draw[-,near start, left] (D) to node {$b$} (G);

\draw[-,above] (E) to node {$c$} (H);
\draw[-,right] (F) to node {$b$} (H);
\draw[-,above] (G) to node {$a$} (H);

\end{tikzpicture}
\end{tabular}
\caption{\label{B-eSMC1}The relevant $S^*_A$ on the left and $S^*_A|\phi$ (Definition \ref{def:UpdateS}) on the right. The announcement of $\phi$ fails to eliminate $\mathscr{B}_\emptyset$ and the relations to it (highlighted by dashed lines).}

\end{figure}   
\vspace{1mm}
\hrule
\end{center}

It should already be clear that the modelling of the example will fail, however, it is useful to see how exactly this failure is going to impact the result.

Normally, we should expect the announcement of $\psi$ to get rid of the bases $\mathscr{B}_{\{a\}}^+, \mathscr{B}_{\{b\}}^+,$ and $\mathscr{B}_{\{c\}}^+$, because if at least one child is muddy, in these bases the sole muddy child would realize that it is muddy. How does that change with $\mathscr{B}^+_\emptyset$ still around? Note that at the bases $\mathscr{B}^+_\emptyset$ and $\mathscr{B}^+_{\{a\}}$, Anne only has access to bases in which $m_a$ is supported. By Lemma \ref{PALMonotonicity}, we know that this will also hold for all superset bases. So $\Vdash^\phi_{\mathscr{B}_\emptyset, \mathfrak{R}_A} K_a m_a$ and $\Vdash^\phi_{\mathscr{B}_{\{a\}}, \mathfrak{R}_A} K_a m_a$. This means that $\psi$ fails at these bases. At all other bases, however, the children do not yet know whether they are muddy and so $\psi$ continues to hold. So, the announcement of $\psi$ eliminates $\mathscr{B}^+_\emptyset$ and $\mathscr{B}^+_{\{a\}}$ and leaves us with the model in Figure \ref{B-eSMC2}.

\begin{center}
\begin{figure}
\hrule 
\begin{tabular}{c p{.7cm} c}
\begin{tikzpicture}[-,shorten >=1pt,node distance =1cm, thick, baseline={([yshift=-1.8ex]current bounding box.center)}] 

\node (A) at (0,0) {$\mathscr{B}^+_\emptyset$};

\node (C) at (0,3) {$\mathscr{B}^+_{\{b\}}$};
\node (D) at (1.5,1.5) {$\mathscr{B}^+_{\{c\}}$};
\node (B) at (3,0) {$\mathscr{B}^+_{\{a\}}$};

\node[text width=1.4cm, text centered] (E) at (3,3) {$\mathscr{B}_{\{a,b\}}^+$};
\node[text width=1.4cm, text centered] (F) at (4.5,1.5) {$\mathscr{B}_{\{a,c\}}^+$};
\node[text width=1.4cm, text centered] (G) at (1.5,4.5) {$\mathscr{B}_{\{b,c\}}^+$};

\node[text width=1.4cm, text centered] (H) at (4.5,4.5) {$\mathscr{B}_{\{a,b,c\}}^+$};

\draw[-,above] (A) to node {$a$} (B);
\draw[-,left] (A) to node {$b$} (C);
\draw[-,above] (A) to node {$c$} (D);

\draw[-, near start,left] (B) to node {$b$} (E);
\draw[-,above] (B) to node {$c$} (F);

\draw[-,near start, above] (C) to node {$a$} (E);
\draw[-,above] (C) to node {$c$} (G);

\draw[-,near start, above] (D) to node {$a$} (F);
\draw[-,near start, left] (D) to node {$b$} (G);

\draw[-,above] (E) to node {$c$} (H);
\draw[-,right] (F) to node {$b$} (H);
\draw[-,above] (G) to node {$a$} (H);

\end{tikzpicture}

&

$\Longrightarrow$

&

\begin{tikzpicture}[-,shorten >=1pt,node distance =1cm, thick, baseline={([yshift=-1.8ex]current bounding box.center)}] 


\node (C) at (0,3) {$\mathscr{B}_{\{b\}}^{++}$};
\node (D) at (1.5,1.5) {$\mathscr{B}_{\{c\}}^{++}$};

\node[text width=1.4cm, text centered] (E) at (3,3) {$\mathscr{B}_{\{a,b\}}^{++}$};
\node[text width=1.4cm, text centered] (F) at (4.5,1.5) {$\mathscr{B}_{\{a,c\}}^{++}$};
\node[text width=1.4cm, text centered] (G) at (1.5,4.5) {$\mathscr{B}_{\{b,c\}}^{++}$};

\node[text width=1.4cm, text centered] (H) at (4.5,4.5) {$\mathscr{B}_{\{a,b,c\}}^{++}$};



\draw[dashed,near start, above] (C) to node {$a$} (E);
\draw[dashed,above] (C) to node {$c$} (G);

\draw[dashed,near start, above] (D) to node {$a$} (F);
\draw[dashed,near start, left] (D) to node {$b$} (G);

\draw[-,above] (E) to node {$c$} (H);
\draw[-,right] (F) to node {$b$} (H);
\draw[-,above] (G) to node {$a$} (H);

\end{tikzpicture}
\end{tabular}
\caption{\label{B-eSMC2} After the announcement of $\psi$ (right side),  the bases $\mathscr{B}_{\{b\}}$ and $\mathscr{B}_{\{c\}}$ fail to be eliminated (with the corresponding relations highlighted by dashed lines).}
\vspace{1mm}
\hrule
\end{figure}    
\end{center}

Remember that we wanted to show that  $\Vdash_{\mathscr{B}_{\{a,b\}}, \mathfrak{R}_{A}} [\phi] [\psi] (K_a m_a \wedge K_b m_b) $. As it turns out  $\Vdash^{\phi,\psi}_{\mathscr{B}^{++}_{\{a,b\}}, \mathfrak{R}_{A}} K_b m_b$, but $\nVdash^{\phi,\psi}_{\mathscr{B}^{++}_{\{a,b\}}, \mathfrak{R}_{A}} K_a m_a$.

Interestingly, Anne being muddy at base $\mathscr{B}_\emptyset$ did not change Bob's ability to find out whether they are muddy or not. Similarly, Cath's reasoning remains largely unchanged in our example as well, even though they were unable to conclude whether they are muddy or not. After the announcements of $\phi$ and $\psi$, Cath still considers $\mathscr{B}_{\{a,b\}}$ and $\mathscr{B}_{\{a,b,c\}}$. This suggests that it is important for Anne that we require $\mathscr{B}_\emptyset$ to not support $m_a$. 

Given that we are trying to establish minimal conditions for these bases, there is an important difference between not having `Anne is muddy' and `Anne is not muddy'. In order for the announcement of $\phi$ to have its intended effect, $\mathscr{B}_\emptyset$ must be a base at which none of the children is muddy. Similarly, for all the bases that represent cases in which a child is not muddy it is not enough to not require the child to be muddy. Take a base $\mathscr{B}_{\{b\}}$ at which $m_a$ is supported. At the bases $\mathscr{B}_{\{b\}}$ and $\mathscr{B}_{\{a, b\}}$, Anne now only has access to bases at which they are muddy and so $K_a m_a$ will be supported even before an announcement even takes place and $\psi$ cannot be announced at $\mathscr{B}_{\{a, b\}}$\footnote{Even though this does not technically cause Proposition \ref{prop:MuddyChildren1} to fail, as $K_a m_a \wedge K_b m_b$ will vacuously hold when $\psi$ cannot be announced, it shows that the model does not behave as expected. Consider that $K_c m_c$ will also hold even though Cath is not muddy.}.

So, in order to give actual minimal conditions for bases corresponding to the muddy children puzzle, we update Proposition \ref{prop:MuddyChildren1} by requiring the children to be not muddy at the relevant bases.

\iffalse{
OLD 

Let $\emptyset$ be the base at which no child is muddy. The other bases are created by adding the rule $\rightarrow m_i$ for every child $i$ that is in fact muddy to $\emptyset$. The relations in $\mathfrak{R}_A$ are created as with the Kripke model version of the puzzle. This relevant fragment of the base-extension model is captured by Figure \ref{B-eSMC1}. We assume that the relations are modal relations as given in Definition \ref{modalRelation}. We still consider the actual base the one in which Anne and Bob are muddy and Cath is not.

\begin{figure}[ht]
\hrule 
\begin{center}

\begin{tikzpicture}[-,shorten >=1pt,node distance =2cm, thick, baseline={([yshift=-1.8ex]current bounding box.center)}] 

\node (A) at (0,0) {$\emptyset$};

\node (C) at (0,3) {$\{\Rightarrow m_b\}$};
\node (D) at (1.5,1.5) {$\{\Rightarrow m_c\}$};
\node (B) at (3,0) {$\{\Rightarrow m_a\}$};

\node[text width=1.4cm, text centered] (E) at (3,3) {$\{\Rightarrow m_a,$ \\ $\Rightarrow m_b\}$};
\node[text width=1.4cm, text centered] (F) at (4.5,1.5) {$\{\Rightarrow m_a,$\\ $\Rightarrow m_c\}$};
\node[text width=1.4cm, text centered] (G) at (1.5,4.5) {$\{\Rightarrow m_b,$\\ $\Rightarrow m_c\}$};

\node[text width=1.4cm, text centered] (H) at (4.5,4.5) {$\{\Rightarrow m_a,$\\ $\Rightarrow m_b,$\\ $\Rightarrow m_c\}$};

\draw[-,above] (A) to node {$a$} (B);
\draw[-,left] (A) to node {$b$} (C);
\draw[-,above] (A) to node {$c$} (D);

\draw[-, near start,left] (B) to node {$b$} (E);
\draw[-,above] (B) to node {$c$} (F);

\draw[-,near end, above] (C) to node {$a$} (E);
\draw[-,above] (C) to node {$c$} (G);

\draw[-,near start, above] (D) to node {$a$} (F);
\draw[-,near start, left] (D) to node {$b$} (G);

\draw[-,above] (E) to node {$c$} (H);
\draw[-,left] (F) to node {$b$} (H);
\draw[-,above] (G) to node {$a$} (H);

\end{tikzpicture}

\caption{\label{B-eSMC1} The relevant fragment of the base-extension version of the muddy children puzzle ahead of the father's announcements.}
\end{center}
\vspace{1mm}
\hrule
\end{figure}

We, once again, model the father's announcement that some of the children are muddy by an announcement of $m_a \vee m_b \vee m_c$. 

As before, we would expect this to remove the emptyset base at which no child is muddy from the children's consideration and that the relations from the other bases to that base are removed. However, as it turns out this does not happen. Our construction of the new relation $\mathfrak{R}_A|(m_a\vee m_b\vee m_c)^{\{\Rightarrow m_a, \Rightarrow m_b\}}$ does not allow for it and the update is ineffective (as defined in Definition \ref{def:effective}. 

For Cath, for example, consider $\{\Rightarrow m_c\} \subset \{\Rightarrow m_b, \Rightarrow m_c \}$ and $\mathfrak{R}_c \{\Rightarrow m_b, \Rightarrow m_c \} \{\Rightarrow m_b\} $, so we would remove the relation between $\{\Rightarrow m_c\}$ and $\emptyset$ for $S_c|(m_a\vee m_b\vee m_c){\{\Rightarrow m_a, \Rightarrow m_b\}}$. Due to the construction of $T_c|(m_a\vee m_b\vee m_c){\{\Rightarrow m_a, \Rightarrow m_b\}}$, however,  the connection would be reestablished and we would have $T_c|(m_a\vee m_b\vee m_c){\{\Rightarrow m_a, \Rightarrow m_b\}}\{\Rightarrow m_c\} \emptyset$ and, finally, $\mathfrak{R}_c|(m_a\vee m_b\vee m_c){\{\Rightarrow m_a, \Rightarrow m_b\}}\{\Rightarrow m_c\} \emptyset$. Analogous for the other agents.

We see the result in Figure \ref{B-eSMC2}.

\begin{figure}[ht]
\hrule 
\begin{center}

\begin{tikzpicture}[-,shorten >=1pt,node distance =2cm, thick, baseline={([yshift=-1.8ex]current bounding box.center)}] 

\node (A) at (0,0) {$\emptyset$};

\node (C) at (0,3) {$\{\Rightarrow m_b\}$};
\node (D) at (1.5,1.5) {$\{\Rightarrow m_c\}$};
\node (B) at (3,0) {$\{\Rightarrow m_a\}$};

\node[text width=1.4cm, text centered] (E) at (3,3) {$\{\Rightarrow m_a,$ \\ $\Rightarrow m_b\}$};
\node[text width=1.4cm, text centered] (F) at (4.5,1.5) {$\{\Rightarrow m_a,$\\ $\Rightarrow m_c\}$};
\node[text width=1.4cm, text centered] (G) at (1.5,4.5) {$\{\Rightarrow m_b,$\\ $\Rightarrow m_c\}$};

\node[text width=1.4cm, text centered] (H) at (4.5,4.5) {$\{\Rightarrow m_a,$\\ $\Rightarrow m_b,$\\ $\Rightarrow m_c\}$};

\draw[-,red, above] (A) to node {$a$} (B);
\draw[-,red, left] (A) to node {$b$} (C);
\draw[-,red, above] (A) to node {$c$} (D);

\draw[-, near start,left] (B) to node {$b$} (E);
\draw[-,above] (B) to node {$c$} (F);

\draw[-,near end, above] (C) to node {$a$} (E);
\draw[-,above] (C) to node {$c$} (G);

\draw[-,near start, above] (D) to node {$a$} (F);
\draw[-,near start, left] (D) to node {$b$} (G);

\draw[-,above] (E) to node {$c$} (H);
\draw[-,left] (F) to node {$b$} (H);
\draw[-,above] (G) to node {$a$} (H);

\end{tikzpicture}

\caption{\label{B-eSMC2} The relevant fragment of the base-extension version of the muddy children puzzle after the father's announcements (i.e., $\mathfrak{R}_A|(m_a\vee m_b\vee m_c)^{\{\Rightarrow m_a, \Rightarrow m_b\}}$). Red lines show the relations we expected to eliminate but could not.}
\end{center}
\vspace{1mm}
\hrule
\end{figure}

It is straightforward to see that Proposition \ref{prop:MuddyChildren1} fails because some of our bases are sub- or super-sets of each other. This is, however, only the most minimal example of this problem. For any set of bases in which one base corresponds to Bob being muddy and another to Anne and Bob being muddy, they will have shared a superset base (the union of those two bases) and condition (d) will force the failure of the example in the same way. 

This is because it does not suffice for $m_a$ to not hold at the base in which Bob is the only muddy child, in fact, we need for $\neg m_a$ and $\neg m_c$ to hold at that base. It is not enough for Anne is muddy and Cath is muddy to not hold, we need Anne and Cath to not be muddy. 

We update the definition of the minimal bases and get a correct version of Proposition \ref{prop:MuddyChildren1}.
}\else

\medskip 

\begin{prop}
\label{prop:MuddyChildren2}
       For a set of agents $A = \{a,b,c\}$ and any $N \in P(A)$ (i.e., the powerset of $A$), let $\mathscr{B}_N$ be any consistent base such that for all $j\in A$ with $m_j \in N$,  $\Vdash_{\mathscr{B}_{N}, \mathfrak{R}_{A}} m_j$ and for all $j\in A$ s.t. $m_j \notin N$, $\Vdash_{\mathscr{B}_{N}, \mathfrak{R}_{A}} m_j\to \bot$. For any update set $R_{A}$ with $A=\{a,b,c\}$ s.t. there are \begin{center}
$\mathscr{B}_\emptyset, \mathscr{B}_{\{a\}}, \mathscr{B}_{\{b\}}, \mathscr{B}_{\{c\}}, \mathscr{B}_{\{a,b\}}, \mathscr{B}_{\{a,c\}}, \mathscr{B}_{\{b,c\}}$ and $\mathscr{B}_{\{a,b,c\}}$
   \end{center}
   with 
\begin{center}
\begin{tabular}{lccr}
   $\mathfrak{R}_a \mathscr{B}_{\emptyset} \mathscr{B}_{\{a\}}$,  & $\mathfrak{R}_a \mathscr{B}_{\{b\}} \mathscr{B}_{\{a,b\}}$, & $\mathfrak{R}_a \mathscr{B}_{\{c\}} \mathscr{B}_{\{a,c\}}$, & $\mathfrak{R}_a \mathscr{B}_{\{b,c\}} \mathscr{B}_{\{a,b,c\}}$,  \\
   
       $\mathfrak{R}_b \mathscr{B}_{\emptyset} \mathscr{B}_{\{b\}}$, & $\mathfrak{R}_b \mathscr{B}_{\{a\}} \mathscr{B}_{\{a,b\}}$, & $\mathfrak{R}_b \mathscr{B}_{\{c\}} \mathscr{B}_{\{b,c\}}$, & $\mathfrak{R}_b \mathscr{B}_{\{a, c\}} \mathscr{B}_{\{a, b,c\}}$, \\

$\mathfrak{R}_c \mathscr{B}_{\emptyset} \mathscr{B}_{\{c\}}$, & $\mathfrak{R}_c \mathscr{B}_{\{a\}} \mathscr{B}_{\{a,c\}}$, & $\mathfrak{R}_c \mathscr{B}_{\{b\}} \mathscr{B}_{\{b,c\}}$, & $\mathfrak{R}_c \mathscr{B}_{\{a,b\}} \mathscr{B}_{\{a,b,c\}}$,\\
\end{tabular}
\end{center}
and for which these, together with the necessary relations to make these $S5$ (i.e., reflexive, transitive, and Euclidean), are the only relations involving these bases, 

  \begin{center}
      $\Vdash_{\mathscr{B}_{\{a,b\}}, \mathfrak{R}_{A}} [(m_a\vee m_b\vee m_c)] [\neg (K_a m_a\vee K_a \neg m_a) \wedge \neg (K_b m_b\vee K_b \neg m_b) \wedge \neg (K_c m_c\vee K_c \neg m_c)] (K_a m_a \wedge K_b m_b) $. 
  \end{center}

\end{prop}

It is straightforward to see that for any $\mathscr{B}_\emptyset$, we have $\nVdash_{\mathscr{B}_\emptyset, R_A} \phi$ and so it is eliminated in the way we expect by the announcement of $\phi$ (Figure \ref{B-eSMC3}).


\begin{center}
\begin{figure}
\hrule 
\begin{tabular}{c p{.7cm} c}
\begin{tikzpicture}[-,shorten >=1pt,node distance =1cm, thick, baseline={([yshift=-1.8ex]current bounding box.center)}] 

\node (A) at (0,0) {$\mathscr{B}_\emptyset$};

\node (C) at (0,3) {$\mathscr{B}_{\{b\}}$};
\node (D) at (1.5,1.5) {$\mathscr{B}_{\{c\}}$};
\node (B) at (3,0) {$\mathscr{B}_{\{a\}}$};

\node[text width=1.4cm, text centered] (E) at (3,3) {$\mathscr{B}_{\{a,b\}}$};
\node[text width=1.4cm, text centered] (F) at (4.5,1.5) {$\mathscr{B}_{\{a,c\}}$};
\node[text width=1.4cm, text centered] (G) at (1.5,4.5) {$\mathscr{B}_{\{b,c\}}$};

\node[text width=1.4cm, text centered] (H) at (4.5,4.5) {$\mathscr{B}_{\{a,b,c\}}$};

\draw[-,above] (A) to node {$a$} (B);
\draw[-,left] (A) to node {$b$} (C);
\draw[-,above] (A) to node {$c$} (D);

\draw[-, near start,left] (B) to node {$b$} (E);
\draw[-,above] (B) to node {$c$} (F);

\draw[-,near start, above] (C) to node {$a$} (E);
\draw[-,above] (C) to node {$c$} (G);

\draw[-,near start, above] (D) to node {$a$} (F);
\draw[-,near start, left] (D) to node {$b$} (G);

\draw[-,above] (E) to node {$c$} (H);
\draw[-,right] (F) to node {$b$} (H);
\draw[-,above] (G) to node {$a$} (H);

\end{tikzpicture}

&

$\Longrightarrow$

&

\begin{tikzpicture}[-,shorten >=1pt,node distance =1cm, thick, baseline={([yshift=-1.8ex]current bounding box.center)}] 


\node (C) at (0,3) {$\mathscr{B}_{\{b\}}^+$};
\node (D) at (1.5,1.5) {$\mathscr{B}_{\{c\}}^+$};
\node (B) at (3,0) {$\mathscr{B}_{\{a\}}^+$};

\node[text width=1.4cm, text centered] (E) at (3,3) {$\mathscr{B}_{\{a,b\}}^+$};
\node[text width=1.4cm, text centered] (F) at (4.5,1.5) {$\mathscr{B}_{\{a,c\}}^+$};
\node[text width=1.4cm, text centered] (G) at (1.5,4.5) {$\mathscr{B}_{\{b,c\}}^+$};

\node[text width=1.4cm, text centered] (H) at (4.5,4.5) {$\mathscr{B}_{\{a,b,c\}}^+$};


\draw[-, near start,left] (B) to node {$b$} (E);
\draw[-,above] (B) to node {$c$} (F);

\draw[-,near start, above] (C) to node {$a$} (E);
\draw[-,above] (C) to node {$c$} (G);

\draw[-,near start, above] (D) to node {$a$} (F);
\draw[-,near start, left] (D) to node {$b$} (G);

\draw[-,above] (E) to node {$c$} (H);
\draw[-,right] (F) to node {$b$} (H);
\draw[-,above] (G) to node {$a$} (H);

\end{tikzpicture}
\end{tabular}
\caption{\label{B-eSMC3}The relevant $S^*_A$ on the left and $S^*_A|\phi$ on the right (Definition \ref{def:UpdateS}). Given the new requirement in Proposition \ref{prop:MuddyChildren2}, the announcement of $\phi$ goes as expected.}
\vspace{1mm}
\hrule
\end{figure}    
\end{center}

Similarly, $\psi$ will fail at the bases $\mathscr{B}^+_{\{a\}}$, $\mathscr{B}^+_{\{b\}}$, and $\mathscr{B}^+_{\{c\}}$ after the announcement of $\phi$ as the child that is muddy at them will know that they are muddy. This leaves us with the final result in Figure \ref{B-eSMC4}. Importantly (and as expected), $\Vdash_{\mathscr{B}_{\{a,b\}}, \mathfrak{R}_{A}} [\phi] [\psi] (K_a m_a \wedge K_b m_b) $. 

\begin{center}
\begin{figure}
\hrule 
\begin{tabular}{c p{.7cm} c}
\begin{tikzpicture}[-,shorten >=1pt,node distance =1cm, thick, baseline={([yshift=-1.8ex]current bounding box.center)}] 


\node (C) at (0,3) {$\mathscr{B}_{\{b\}^+}$};
\node (D) at (1.5,1.5) {$\mathscr{B}_{\{c\}}^+$};
\node (B) at (3,0) {$\mathscr{B}_{\{a\}}^+$};

\node[text width=1.4cm, text centered] (E) at (3,3) {$\mathscr{B}_{\{a,b\}}^+$};
\node[text width=1.4cm, text centered] (F) at (4.5,1.5) {$\mathscr{B}_{\{a,c\}}^+$};
\node[text width=1.4cm, text centered] (G) at (1.5,4.5) {$\mathscr{B}_{\{b,c\}}^+$};

\node[text width=1.4cm, text centered] (H) at (4.5,4.5) {$\mathscr{B}_{\{a,b,c\}}^+$};


\draw[-, near start,left] (B) to node {$b$} (E);
\draw[-,above] (B) to node {$c$} (F);

\draw[-,near start, above] (C) to node {$a$} (E);
\draw[-,above] (C) to node {$c$} (G);

\draw[-,near start, above] (D) to node {$a$} (F);
\draw[-,near start, left] (D) to node {$b$} (G);

\draw[-,above] (E) to node {$c$} (H);
\draw[-,right] (F) to node {$b$} (H);
\draw[-,above] (G) to node {$a$} (H);

\end{tikzpicture}

&

$\Longrightarrow$

&

\begin{tikzpicture}[-,shorten >=1pt,node distance =1cm, thick, baseline={([yshift=-1.8ex]current bounding box.center)}] 



\node[text width=1.4cm, text centered] (E) at (3,3) {$\mathscr{B}_{\{a,b\}}^{++}$};
\node[text width=1.4cm, text centered] (F) at (4.5,1.5) {$\mathscr{B}_{\{a,c\}}^{++}$};
\node[text width=1.4cm, text centered] (G) at (1.5,4.5) {$\mathscr{B}_{\{b,c\}}^{++}$};

\node[text width=1.4cm, text centered] (H) at (4.5,4.5) {$\mathscr{B}_{\{a,b,c\}}^{++}$};





\draw[-,above] (E) to node {$c$} (H);
\draw[-,right] (F) to node {$b$} (H);
\draw[-,above] (G) to node {$a$} (H);

\end{tikzpicture}
\end{tabular}
\caption{\label{B-eSMC4}The relevant $S^*_A|\phi$ on the left and $S^*_A|(\phi,\psi)$ on the right (Definition \ref{def:UpdateS}). The announcement of $\psi$ also functions as intended.}
\vspace{1mm}
\hrule
\end{figure}    
\end{center}

It should be reiterated here that this only proves that the conditions in Proposition \ref{prop:MuddyChildren2} guarantee that the muddy children puzzle goes through in our P-tS-based setting. It does not prove that these are in fact \emph{minimal} requirements in the strong sense. There can be other requirements, either weaker or simply different, that will also have the desired results. We do contend, however, that these conditions are at least the most natural restrictions to put on the bases and directly represent the situation as given in the description of the example.

We suggest that the need for imposing such conditions in the P-tS-based analysis of the muddy children puzzle is wholly consistent with an inferentialist position in which we expect to make assumptions explicit (cf. the position presented in \cite{Brandom1994}).

\subsection{Information and Intensionality}
\label{sec:Intensionality}

When trying to analyse information through examples like the card game and the muddy children, it is important to identify exactly what it is that allows the agents to draw their conclusions. This means that in the context of public announcements, we must identify exactly what about the announcement and the context in which it was given/received allows the agents to make their specific inferences. 

In Section \ref{sec:Examples}, we can see that the way in which worlds in the Kripke semantics modelling each of these examples have been chosen is purely extensional. For the card game, we simply picked worlds at which the card-draw lined up with the example. Similarly, in the muddy children example, we picked a world that corresponds to each configuration of the muddy children for that example. In fact, worlds are inherently extensional in this way because, ultimately, the truth of a formula is based on such abstract 
interpretations. 

Our inferentialist analysis aims to achieve a deeper understanding of the information involved in these examples, characterized by drawing on the intensional character of bases. As we have already mentioned at the beginning of this section, the two examples highlight two ways in which bases and base rules can be used. After going through the examples, we want to restate these two ways and connect them with an intensional approach to information. 

In the card game example, we specifically \emph{do not} just consider bases at which the rules of the game hold, but rather construct bases that have base rules 
which directly represent those rules of the game. This allows us to be able to say not just that these rules hold, as is the case in the Kripke model, but exactly \emph{why} they hold. 

In both examples, we give a minimal set-up that tells us that any extension of this set-up will allow the example to go through as expected. This is especially highlighted in the muddy children example, in which the example fails for a specific set of requirements. This allows us to identify exactly what bases and relations are required for the example to go through and so what information must be encoded in the model for the agent to make the expected inferences. Although maybe not as obviously intensional in character as implementing the rules directly, as is done in the card game, any minimal instantiation of this set-up by specific base rules will say a child is muddy (or not) is supported exactly because of those base rules. Furthermore, at any extension of such a minimal instantiation, that support persists because of the presence of those same base rules.\footnote{Of course, for any extension that is also an extension of another minimal set-up, there are now two reasons for the formula to be supported. In this sense, smaller bases are `more' intensional than larger ones. At the extreme, maximally-consistent bases, which can be considered to amount to Kripke worlds, contain all the ways that cause the formula to be supported and, therefore, are extensional in character.}

\section{Future Work} \label{sec:conclusion}

We view this work as a first step in an inferentialist analysis of the group of dynamic epistemic logics. Public announcement logic is the simplest of these logics. The natural next step is to look at action model logics (AML). In our setting, public announcement leave the bases unchanged and only change the modal relations. This means that a base before an action is the same as after an action. In AML, there can be multiple possible actions that the agents consider could have taken place. So, a base before an action would have to correspond to multiple bases after the action. We believe, our semantic can be generalized to further extend the language with propositional atoms that mark the action that has taken place. However we also suggest an approach that instead allows the agents to consider different modal relations, one for each possible action they consider. In this sense the agents uncertainty about what action actually took place is modelled by their uncertainty in what are the correct modal relations. 

There are also some interesting aspects of (static) epistemic logic and PAL that still need to be analysed from the inferentialist perspective. We have yet to discuss group notions of knowledge, such as distributed knowledge or common knowledge. Common knowledge is especially interesting in the context of PAL, as there is a close association between it and public announcements. Another notion relevant to PAL is unsuccessful update; that is,  announcements of formulae that become false when they are announced. All this analysis is important in identifying the informational content of announcements. See \cite{Ditmarsch2007} for an overview on these concepts in the context of dynamic epistemic logic in general.

We suggest that the future work discussed in this section is instrumental in providing an inferentialist analysis of information in certain circumstances. Specifically, we conjecture that the inferentialist view of meaning as being derived from use --- in particular, from inference in P-tS --- provides a fruitful perspective on the concept of information itself, not only from a foundational perspective, but also as it occurs in settings such as informatics and technology. While this work and the directions for future work discussed above are useful in their own terms, they also provide some first steps towards an inferentialist account of information.


\iffalse{

OLD


\begin{figure}[ht]
\begin{center}

\begin{tikzpicture}[-,shorten >=1pt,node distance =2cm, thick, baseline={([yshift=-1.8ex]current bounding box.center)}] 

\node[text width=1cm, text centered] (A) at (0,0) {{\small$\{\neg m_a,$\\ $\neg m_b$, \\ $\neg m_c \}$}};

\node[text width=1cm, text centered] (C) at (0,3) {{\small$\{\neg m_a,$ \\$m_b,$\\ $\neg m_c \}$}};
\node[text width=1cm, text centered] (D) at (1.5,1.5) {{\small$\{\neg m_a,$ \\$ \neg m_b,$ \\$m_c\}$}};
\node[text width=1cm, text centered] (B) at (3,0) {{\small$\{ m_a$ \\$ \neg m_b,$ \\$\neg m_c \}$}};

\node[text width=1cm, text centered] (E) at (3,3) {{\small$\{m_a,$ \\ $m_b$\\ $\neg m_c\}$}};
\node[text width=1cm, text centered] (F) at (4.5,1.5) {{\small$\{m_a,$ \\ $\neg m_b,$ \\$m_c\}$}};
\node[text width=1cm, text centered] (G) at (1.5,4.5) {{\small$\{\neg m_a$\\ $m_b,$\\ $m_c\}$}};

\node[text width=1.8cm, text centered] (H) at (4.5,4.5) {{\small$\{m_a,$\\ $m_b,$\\ $m_c\}$}};

\draw[-,above] (A) to node {$a$} (B);
\draw[-,left] (A) to node {$b$} (C);
\draw[-,above] (A) to node {$c$} (D);

\draw[-, near start,left] (B) to node {$b$} (E);
\draw[-,above] (B) to node {$c$} (F);

\draw[-,near end, above] (C) to node {$a$} (E);
\draw[-,above] (C) to node {$c$} (G);

\draw[-,near start, above] (D) to node {$a$} (F);
\draw[-,near start, left] (D) to node {$b$} (G);

\draw[-,above] (E) to node {$c$} (H);
\draw[-,left] (F) to node {$b$} (H);
\draw[-,above] (G) to node {$a$} (H);

\end{tikzpicture}

\caption{\label{B-eSMC4} The relevant fragment of the base-extension version of the muddy children puzzle ahead of the father's announcements with the correct bases.}
\end{center}
\end{figure}

\begin{figure}[ht]
\begin{center}

\begin{tikzpicture}[-,shorten >=1pt,node distance =2cm, thick, baseline={([yshift=-1.8ex]current bounding box.center)}] 


\node[text width=1cm, text centered] (C) at (0,3) {{\small$\{\neg m_a,$ \\$m_b,$\\ $\neg m_c \}$}};
\node[text width=1cm, text centered] (D) at (1.5,1.5) {{\small$\{\neg m_a,$ \\$ \neg m_b,$ \\$m_c\}$}};
\node[text width=1cm, text centered] (B) at (3,0) {{\small$\{ m_a$ \\$ \neg m_b,$ \\$\neg m_c \}$}};

\node[text width=1cm, text centered] (E) at (3,3) {{\small$\{m_a,$ \\ $m_b$\\ $\neg m_c\}$}};
\node[text width=1cm, text centered] (F) at (4.5,1.5) {{\small$\{m_a,$ \\ $\neg m_b,$ \\$m_c\}$}};
\node[text width=1cm, text centered] (G) at (1.5,4.5) {{\small$\{\neg m_a$\\ $m_b,$\\ $m_c\}$}};

\node[text width=1.8cm, text centered] (H) at (4.5,4.5) {{\small$\{m_a,$\\ $m_b,$\\ $m_c\}$}};


\draw[-, near start,left] (B) to node {$b$} (E);
\draw[-,above] (B) to node {$c$} (F);

\draw[-,near end, above] (C) to node {$a$} (E);
\draw[-,above] (C) to node {$c$} (G);

\draw[-,near start, above] (D) to node {$a$} (F);
\draw[-,near start, left] (D) to node {$b$} (G);

\draw[-,above] (E) to node {$c$} (H);
\draw[-,left] (F) to node {$b$} (H);
\draw[-,above] (G) to node {$a$} (H);

\end{tikzpicture}

\caption{\label{B-eSMC5} The relevant fragment of the base-extension version of the muddy children puzzle after of the father's announcements with the correct bases. Bases no longer considered are removed for readability.}
\end{center}
\end{figure}

\begin{figure}[ht]
\begin{center}

\begin{tikzpicture}[-,shorten >=1pt,node distance =2cm, thick, baseline={([yshift=-1.8ex]current bounding box.center)}] 



\node[text width=1cm, text centered] (E) at (3,3) {{\small$\{m_a,$ \\ $m_b$\\ $\neg m_c\}$}};
\node[text width=1cm, text centered] (F) at (4.5,1.5) {{\small$\{m_a,$ \\ $\neg m_b,$ \\$m_c\}$}};
\node[text width=1cm, text centered] (G) at (1.5,4.5) {{\small$\{\neg m_a$\\ $m_b,$\\ $m_c\}$}};

\node[text width=1.8cm, text centered] (H) at (4.5,4.5) {{\small$\{m_a,$\\ $m_b,$\\ $m_c\}$}};





\draw[-,above] (E) to node {$c$} (H);
\draw[-,left] (F) to node {$b$} (H);
\draw[-,above] (G) to node {$a$} (H);

\end{tikzpicture}

\caption{\label{B-eSMC6} The relevant fragment of the base-extension version of the muddy children puzzle when the children solve it after the father's second announcement. Bases no longer considered are removed for readability.}
\end{center}
\end{figure}

}\else

\subsection*{Acknowledgements} 

This work has been partially supported by UK EPSRC grants EP/S013008/1 and EP/R006865/1. We are grateful to Gabriele Brancati, Tim Button, Yll Buzoku, Alex Gheorghiu, Tao Gu, Timo Lang, Elaine Pimentel, and Will Stafford  for helpful discussions of this work. We thank an anonymous reviewer for pointing out some typographical errors in an earlier version of this manuscript.

\bibliographystyle{plain}
	\bibliography{DELpts}

\appendix

 \section{Proof Details for Section \ref{sec:B-eS-S5}}
    \label{App1}
    \medskip

We give here the lemmas elided in 
Section~\ref{sec:B-eS-S5}. 

\begin{lemma} \label{ModalMonotonicity}
If $\Gamma\Vdash_{\mathscr{B},\mathfrak{R}_A} \phi$ and $\mathscr{B}\subseteq\mathscr{C}$, then $\Gamma\Vdash_{\mathscr{C},\mathfrak{R}_A} \phi$.
    
\end{lemma}
\begin{proof}
    This proof goes by induction on the complexity of $\phi$. If $\Gamma$ is not empty, this follows immediately from the validity conditions. 
    
    For empty $\Gamma$, the propositional cases remain unchanged from the proof for classical propositional logic in \cite{Sandqvist2009}. The new case is the case in which $\phi=K_a\psi$. In order for $\Vdash_{\mathscr{B},\mathfrak{R}_A} K_a\psi$, at all $\mathscr{D}$ s.t. $\mathfrak{R}_a\mathscr{B}\mathscr{D}$, $\Vdash_{\mathscr{D},\mathfrak{R}_A} \psi$.
    By condition (d) of Definition \ref{modalRelation}, we know that for every $\mathscr{E}$ s.t. $\mathfrak{R}_a\mathscr{C}\mathscr{E}$, there is such a $\mathscr{D}$ with $\mathscr{D}\subseteq\mathscr{E}$. By induction hypothesis $\Vdash_{\mathscr{E},\mathfrak{R}_A} \psi$ and, since this holds for all $\mathscr{E}$ s.t. $\mathfrak{R}_a\mathscr{C}\mathscr{E}$, we conclude $\Vdash_{\mathscr{C},\mathfrak{R}_A} K_a\psi$.
\end{proof}

Additionally, the conditions of Definition \ref{modalRelation} guarantees {\it ex falso quod libet} in the sense that any formula $\phi$ follows from $\bot$.

\medskip 

\begin{lemma}
    \label{EFQ}
    For all inconsistent bases $\mathscr{B}$ and $\mathfrak{R}_A$, $\Vdash_{\mathscr{B},\mathfrak{R}_A}\phi$, for all $\phi$.
    
\end{lemma}
\begin{proof}
    This proof goes by induction on the complexity of the formula $\phi$. For propositional $\phi$ this follows immediately from the validity conditions in Definition \ref{UpdateValidity}. For the case in which $\phi=K_a\psi$, the condition (a) ensures that any inconsistent base can only have access to other inconsistent bases and, by induction hypothesis, $\psi$ holds at these bases.
\end{proof}

We can show that Kripke-like validity conditions hold for maximally-consistent bases as they do not have consider extensions.

\medskip 

\begin{lemma}
\label{ModalBehaviour}
For any update set $\mathfrak{R}_A$ and maximally-consistent base $\mathscr{B}$, the following hold:

\begin{itemize}
    \item[—] $\nVdash_{\mathscr{B},\mathfrak{R}_A} \bot$, and
    \item[—] $\Vdash_{\mathscr{B},\mathfrak{R}_A} \phi \to \psi$ iff $\nVdash_{\mathscr{B},\mathfrak{R}_A} \phi$ or $\Vdash_{\mathscr{B},\mathfrak{R}_A} \psi$.
\end{itemize}

\end{lemma}

\begin{proof}
    As the validity conditions for $\to$ and $\bot$ do not depend on $\mathfrak{R}_A$  the proof remains unchanged from \cite{Makinson2014}.
\end{proof}

Additionally, for any formula $\phi$, at a maximally-consistent base either it or its negation will hold.

\medskip 

\begin{lemma}
    \label{lem:excludedmiddle}
    For every formula $\phi$, maximally-consistent base $\mathscr{B}$, and update set $\mathfrak{R}_A$, either $\Vdash_{\mathscr{B},\mathfrak{R}_A} \phi$ or $\Vdash_{\mathscr{B},\mathfrak{R}_A} \phi\to\bot$.
\end{lemma}

\begin{proof}
    Consider the case in which $\nVdash_{\mathscr{B},\mathfrak{R}_A} \phi$. Any base $\mathscr{C}$ s.t. $\mathscr{C}\supset\mathscr{B}$ is inconsistent and so $\Vdash_{\mathscr{C},\mathfrak{R}_A} \bot$ and so $\Vdash_{\mathscr{B},\mathfrak{R}_A} \phi\to\bot$. 
\end{proof}

\begin{lemma}
\label{lem: ModalMaxCon} 
For every formula $\phi$, if there is a base $\mathscr{B}$ s.t. $\nVdash_{\mathscr{B},\mathfrak{R}_A} \phi$, then there is a maximally-consistent base $\mathscr{C}\supseteq \mathscr{B}$ with $\nVdash_{\mathscr{C},\mathfrak{R}_A} \phi$.
\end{lemma}

\begin{proof}
    The proof goes by induction on the complexity of the formula $\phi$. The proof for the classical cases follows the proof of Lemma 3.3 and 3.4 in \cite{Makinson2014}.

    We consider the case in which $\phi=K_a\psi$. Since $\nVdash_{\mathscr{B},\mathfrak{R}_A} K_a\psi$, there is a $\mathscr{C}$ s.t. $\mathfrak{R}_a\mathscr{B}\mathscr{C}$ and $\nVdash_{\mathscr{C},\mathfrak{R}_A} \psi$. By our induction hypothesis, we know that there is a maximally-consistent $\mathscr{C}^*\supseteq\mathscr{C}$. By symmetry, $\mathfrak{R}_a\mathscr{C}\mathscr{B}$ and so, by condition (c) of Defintition \ref{modalRelation}, there is a $\mathscr{B}'\supseteq\mathscr{B}$ s.t. $\mathfrak{R}_a\mathscr{C}^*\mathscr{B}'$. Again by symmetry, we have $\mathfrak{R}_a\mathscr{B}'\mathscr{C}^*$. If $\mathscr{B}'$ is maximally-consistent, let $\mathscr{B}'$ be our $\mathscr{B}^*$. Otherwise, note that, for any maximally-consistent $\mathscr{B}^*\supseteq\mathscr{B}'$, we have $\mathfrak{R}_a\mathscr{B}^*\mathscr{C}^*$ because of condition (c) again and because $\mathscr{C}^*$ is maximally-consistent. Finally, since $\nVdash_{\mathscr{C}^*,\mathfrak{R}_A} \psi$, we have $\nVdash_{\mathscr{B}^*,\mathfrak{R}_A} K_a\psi$.\footnote{It might be of interest to note here that this proof relies on $\mathfrak{R}_a$ being symmetric. This suggests that it cannot be used to show that this property holds for arbitrary modal relations. More conditions might be required for this lemma to hold for non-symmetric modal relations.}
\end{proof}

\iffalse{
\begin{cor}
    Any formula $\phi$ that is valid at all maximally-consistent bases for all $\mathfrak{R}_A$, is valid.
\end{cor} \medskip 
}\else

    We now move to the proof of Theorem \ref{theo:UpdateCompleteness}. We use a construction that is an adaptation of the proof of Soundness for $S5$ base-extension semantics see \cite{EckhardtPymS5}. The fundamental strategy, to give a construction that takes a model with a world at which a formula is not true and gives us bases and a set of relations such that the formula does not hold at one of these bases, remains unchanged. We have to add some more detail to ensure the resulting set of relations is an update set as defined in Definition \ref{def:UpdateRelations} and, of course, show that it, in fact, results in an update set.


     We assume an $S5$ model $M = \langle F,V\rangle$ with $F = \langle W,R_A\rangle$, $R_A$ is a set of relation $R_a$, and a world $w$ s.t. $M,w\nvDash \phi$. We now implement the steps of the proof sketch.
    
    In a model, we can have worlds with the same valuation but not maximally-consistent bases that agree on all atomic formulae. So we need to adapt our $M$ so that every world disagrees on some atomic formulae that are not relevant to $\phi$. We do so by assigning a different atomic formula $p_i$ that does not appear in $\phi$ to each world. This is possible as we have infinitely many atomic formulae. For every $v\in W$, let $p_v$ be some atomic formula s.t. it does not appear in $\phi$ and, for all worlds $u$, if $u\neq v$, then $p_v \neq p_u$. Additionally, let $p_\emptyset$ be some atomic formula s.t. it does not appear in $\phi$ and for all $v\in W$, $p_\emptyset\neq p_v$. The idea is that we want $p_v$ to hold at all worlds except $v$ and $p_\emptyset$ to hold at no worlds.

    Let $M' = \langle F, V'\rangle$ be s.t. for all $v\in W$, $V'(p_v) = W\setminus\{v\}$, $V'(p_\emptyset)=\emptyset$, and $V'(p) = V(p)$ for all other atomic formulae $p$. Note that $M', w \nvDash \phi$. This concludes our first step.

    This model is used to construct bases that correspond to its worlds. For every world $w\in W$, we define a base $\mathscr{A}_w$ in the following way:

    \[
    \begin{split}
    \mathscr{A}_w := & \{\Rightarrow p : M',w\vDash p\} \cup \\
                 & \{\Rightarrow p_v : v \in W \text{ and } v\neq w\} \cup \\
                 & \{p\Rightarrow p_w, p_w\Rightarrow p : M',w\nvDash p\} \cup\\
                 & \{p\Rightarrow p_\emptyset, p_\emptyset\Rightarrow p : M',w\nvDash p\} \\
    \end{split}
    \]

    By Lemma \ref{lem: ModalMaxCon}, we know there is a maximally-consistent $\mathscr{C}\supseteq\mathscr{A}_w$ s.t. $\Vdash_{\mathscr{C}} q_w$. We call this $\mathscr{C}$ $\mathscr{B}_w$.


    We define a function $\delta_W$ that tells us for every base which subset bases it shares with the bases that correspond to our worlds (i.e., if a base $\mathscr{B}$ has a subset that is also a subset of $\mathscr{B}_w$ and $\mathscr{B}_v$ but no other $\mathscr{B}_u$, then $\{w,v\}\in \delta_W(\mathscr{B)}$).

    \medskip 

    \begin{definition}
        \label{Def:delta}

        For a base $\mathscr{B}$ and a set of worlds $W$, 
        
        \begin{center}
        $\delta_W(\mathscr{B}) =\left\{ \{w_1,\dots,w_j\}\middle\vert \begin{array}{l}
        \{w_1,\dots,w_j\} \subseteq W \text{ and there is } \mathscr{B}'\subseteq\mathscr{B} \text{ s.t.}\\
        \text{for all }w_i \in \{w_1,\dots,w_j\}, \mathscr{B}'\subseteq\mathscr{B}_{w_i}
              
        \end{array}\right\}$
        \end{center}\vspace{-7mm}
        \fillBox 
        \end{definition} \medskip


    We can show that for any possible outcome of $\delta_W$ there exist a base $\mathscr{B}$ given our construction. Note, that any set that does not include $W$ itself is not possible as every base is a superset of the $\emptyset$-base.

\medskip 

    \begin{lemma}
        \label{lem:ExistsDelta}

        For any set $\Delta \subseteq P(W)$ with $W\in \Delta$ there is a base $\mathscr{B}$ s.t. $\delta_W(\mathscr{B}) = \Delta$.
        
    \end{lemma}
    \begin{proof}
        We proof this giving construction rules for these bases. Let $\mathscr{B}$ be the set with for all $\{w_1,\dots,w_n\} \in \Delta$ with $n>1$, $(p_{w_1},\dots, p_{w_n} \Rightarrow p_\emptyset) \in \mathscr{B}$ and no other rules.

        Note that the base containing the single rule $(p_{w_1},\dots, p_{w_n} \Rightarrow p_\emptyset)$ is a subset of $\mathscr{B}_{w_1}, \dots, \mathscr{B}_{w_n}$ but no other $\mathscr{B}_{v}$.
\end{proof}
    
    $\mathfrak{R}_A$ is the set of modal relations such that, for every agent $a\in A$, $\mathfrak{R}_a \mathscr{B}\mathscr{C}$ only by the following:

\begin{enumerate}[label=(\arabic{enumi})]
    \item[—]if $R_a wv$, then $\mathfrak{R}_a\mathscr{B}_w\mathscr{B}_v$
    \item[—]if $\mathscr{B}$ and $\mathscr{C}$ are inconsistent, then $\mathfrak{R}_a\mathscr{B}\mathscr{C}$
    
    
    \item[—]If $\mathscr{B}$ and $\mathscr{C}$ \begin{itemize}
        \item[—]are consistent
        \item[—]there is no $w\in W$ s.t. $\mathscr{B}=\mathscr{B}_w$ or $\mathscr{C}=\mathscr{B}_w$
        \item[—]there is a $w \in W$ s.t. $\mathscr{B}\subseteq\mathscr{B}_w$ iff there is a $v\in W$ s.t. $\mathscr{C}\subseteq\mathscr{B}_v$

        \item[—]there is a rule $r\in\mathscr{B}$ s.t. $r\notin\mathscr{B}_w$ for all $w\in W$ iff there is a rule $r'\in\mathscr{C}$ s.t. $r'\notin\mathscr{B}_w$ for all $w\in W$
        
        \item[—]there are bijective functions $f: \delta_W(\mathscr{B}) \to \delta_W(\mathscr{C})$ and $g: W \rightarrow W$ s.t. for every $\delta \in \delta_W(\mathscr{B)}$ and $w\in \delta$, $g(w)\in f(\delta)$ and $R_a w g(w)$ 


    \end{itemize}
         then $\mathfrak{R}_a\mathscr{B}\mathscr{C}$. 

\end{enumerate}

\vspace{.5cm}

    We start by showing that the resulting relations are equivalence relations on an arbitrary $\mathscr{R}_a$. Note that this trivially holds for all relations added by (2). 
    
    For reflexivity, it suffices to point out that $R_a$ was reflexive and so for any $\mathscr{B}_w$, we have $\mathfrak{R}_a \mathscr{B}_w\mathscr{B}_w$. For all other bases note that we can simply take functions $f(\delta)=\delta$ and $g(w)=w$ and so $\mathfrak{R}_a\mathscr{B}\mathscr{B}$. 
    
    For transitivity, from $\mathfrak{R_A}\mathscr{B}_w\mathscr{B}_v$ and $\mathfrak{R_A}\mathscr{B}_v\mathscr{B}_u$, it follow that $\mathfrak{R_A}\mathscr{B}_w\mathscr{B}_u$, simply by the transitivity of $R_a$. For the other cases note that there are functions $f$ from $\delta_W(\mathscr{B}$ to $\delta_W(\mathscr{C}$ and $f'$ $\delta_W(\mathscr{C}$ to $\delta_W(\mathscr{D}$, so we can take the new function $f'':\delta_w(\mathscr{B)\rightarrow\delta_w(\mathscr{D)}}$ with $f''(\delta)=f'(f(\delta))$ and function $g$ and $g'$,  let $g''(w)=g'(g(w))$. So, by (3) $\mathfrak{R}_a\mathscr{B}\mathscr{D}$.
    
    The Euclidean case follows analogously.

    To show this is an update relation we go through the steps (1)-(3) and show that none of these steps add relations that violate the condition of update sets. As step (2) is the only one concerning inconsistent bases and remains unchanged from our proof in \cite{EckhardtPymS5}, conditions (a) and (b) hold trivially.
    

    We proceed to show (c) and (d).

    For condition (c), take some $\mathscr{R}_a \mathscr{B}\mathscr{C}$. There are three cases to consider. 

    If $\mathscr{B}=\mathscr{B}_w$ and $\mathscr{C}=\mathscr{B}_v$ for some $w$ and $v$, then any consistent $\mathscr{B}'=\mathscr{B}_w$ and this holds trivially.

    If $\mathscr{B}'=\mathscr{B}_w$ for some $w$, then there is a $\mathscr{B}_v\supseteq\mathscr{C}$ for some $v$ s.t. $R_a wv$ and so $\mathfrak{R}_a\mathscr{B}_w\mathscr{B}_v$.

    For all other $\mathscr{B}$ and $\mathscr{B}'$, $\mathfrak{R}_a\mathscr{B}\mathscr{C}$ has to hold by step (3) and so there are $f$ and $g$ we can adapt. Additionally, $\delta_W(\mathscr{B})\subseteq \delta_W(\mathscr{B}')$. So, we can define an $f'$ and $g'$ based on them. For all $\mathscr{D}$, let $g'(\mathscr{D})=g(\mathscr{D})$ and $f'(\delta)=\{g'(\mathscr{D})|\mathscr{D}\in\delta\}$. So, for any $\mathscr{C}'\supseteq\mathscr{C}$ s.t. $\delta_W(\mathscr{C}') = \{f'(\delta)|\delta\in\delta_W(\mathscr{B}')\}$, we have $\mathfrak{R}_a \mathscr{B}'\mathscr{C}'$. We obtain such a $\mathscr{C}'$ by simply adding rules to $\mathscr{C}$ that correspond to those $\delta$ s.t. $\delta\in\delta_W(\mathscr{C}')$ but $\delta\notin\delta_W(\mathscr{C})$.

    Condition (d) follows if we straightforwarldy modify the approach for all other $\mathscr{B}$ and $\mathscr{B}'$ to deal with subsets rather than supersets. We simply remove rules rather than adding them.

    Now we need to show that $\mathfrak{R}_A$ is an update set.
    
    To show that for every base $\mathscr{B}$ there is a $\pi(\mathscr{B})$, let $\Delta$ be the set of all $\delta_W$ s.t. there are bijective functions $f: \delta_W(\mathscr{B}) \to \delta_W$ and $g: W \rightarrow W$ s.t. for every $\delta \in \delta_W(\mathscr{B)}$ and $w\in \delta$, $g(w)\in f(\delta)$ and $R_a w g(w)$. Let $\pi= \{\pi(\delta_W)| \delta_W\in \Delta\}$ and $\pi(\delta_W)$ the set of all $\mathscr{C}$ reachable from $\mathscr{B}$ s.t. $\delta_W(\mathscr{C}) = \delta_W$. $pi_\mathscr{B}$ is clearly finite as there are only a finite amount of worlds in $W$, the $\pi(\delta_W)$ only contain members of $\Omega_P$ and for all $\mathscr{C}$, $\mathscr{C}$ is reachable by $\mathscr{B}$ iff there is a $\pi(\delta_W)$ s.t. $\mathscr{C} \in \pi(\delta_W)$. 
    
    To show that any $\pi(\delta_W)$ only contains bases that are reachable from each other via the sub/superset take any $\mathscr{C}$ and $\mathscr{C}'$ in $\pi(\delta_W)$ and let $\mathscr{C}^+ = \mathscr{C}\cup\mathscr{C}'$. If $\mathscr{C}= \mathscr{C}'$ this is trivial. There are now two cases to consider: If it is not the case that $\mathscr{C}^+ = \mathscr{B}_w$ for some $w$, then simply $\mathscr{C}^+ \in \pi(\delta_W)$ and as it is the superset of both $\mathscr{C}$ and $\mathscr{C}'$ we are done. Otherwise, take some rule $r$ that is in $\mathscr{C}$ but not in $\mathscr{C}'$. The base $(\mathscr{C}^+ - r)$ is then a superset of $\mathscr{C}'$. Now we remove a rule $r'$ that is in $\mathscr{C}'$ but not in $\mathscr{C}$ and get $((\mathscr{C}^+-r)-r')$ adding $r$ back in results in $(\mathscr{C}^+-r')$ which is a superset of $\mathscr{C}$. Obviously, $((\mathscr{C}^+-r)-r')$ is a subset of both $(\mathscr{C}^+-r)$ and $(\mathscr{C}^+ -r')$. As we have only manipulated rules of $\mathscr{C}$ and $\mathscr{C}'$, $(\mathscr{C}^+ -r)$, $((\mathscr{C}^+ - r) - r')$, and $(\mathscr{C}^+ -r')$ are members of $\pi(\delta_W)$ and so $\mathscr{C}'$ can be reached from $\mathscr{C}$ through them.
    
    
    So, finally, we show that for $\mathscr{B}'\subseteq\mathscr{B}$ s.t. $\mathscr{B}'$ is reachable from $\mathscr{B}$, $\mathscr{B}$ and $\mathscr{B}'$ share access to the same bases, that is, we have an update set.

    Step (1) only add relations between maximally-consistent sets. Two maximally-consistent sets can only be subsets of each other if they are the same set and so also share the same bases they have access to.

   Step (2) only adds relations between inconsistent sets and our update condition only restrict consistent sets.


    Finally, we look at the (3) steps. 
    Note that, for any sub- or superset of a base it either has the same $\delta_W$ or one with a higher or lower length. If the length is different, the super-/subset cannot be reachable from the inital base as $(3)$ requires the length to be equal for two bases to be connected, as it requires a bijective function between them, and if it's the same, then the bases get connected to the same bases by $(3)$.
    

    This proves that the set of relations resulting from this construction is an update set and so the base-extension semantics for multi-agent $S5$-modal logic based on update sets is sound.

\section{Proof Details for Section \ref{sec:B-eS-PAL}}
    \label{App2}
    \medskip 

We give here the lemmas elided in 
Section~\ref{sec:B-eS-PAL}. 

\subsection*{Existence of effective updates}
\medskip 

We start by showing that the relations resulting from Definition~\ref{def:EffectiveUpdateRelation} are an update set.

\begin{lemma}
    \label{PALrelationIsModal}

    For any update set $\mathfrak{R}_A$, $a\in A$, base $\mathscr{B}$ and formula $\phi$ s.t. $\Vdash_{\mathscr{B},\mathfrak{R}_A}\phi$, $\mathfrak{U}_a^\mathscr{B}|\phi$ is a $S5$-modal relation and $\mathfrak{U}_A^\mathscr{B}|\phi$ is an update set.   
\end{lemma}
\begin{proof}
This proof follows the same strategy as the proof of soundness for $S5$ as the construction is largely the same.

 We start by showing that the resulting relations are equivalence relations for an arbitrary $\mathfrak{U}^\mathscr{B}_a|\phi \in \mathfrak{U}^\mathscr{B}_A|\phi$. Note that they trivially hold for all relations added by (2). 
    
    For reflexivity, it suffices to point out that $S_a$, as defined in Definition \ref{def:UpdateS}, is reflexive and so for any $\mathscr{C}^+$, we have $\mathfrak{U}^\mathscr{B}_a|\phi \mathscr{C}^+\mathscr{C}^+$. For all other bases note that we can simply take functions $f(\delta)=\delta$ and $g(\mathscr{X})=\mathscr{X}$ and so $\mathfrak{U}^\mathscr{B}_a|\phi\mathscr{C}\mathscr{C}$. 
    
    For transitivity, from $\mathfrak{U}^\mathscr{B}_a|\phi\mathscr{C}^+\mathscr{D}^+$ and $\mathfrak{U}^\mathscr{B}_a|\phi\mathscr{D}^+\mathscr{E}^+$, it follow that $\mathfrak{U}^\mathscr{B}_a|\phi\mathscr{C}^+\mathscr{E}^+$, simply by the transitivity of $S_a$. For the other cases note that there are functions $f$ from $\delta_{\pi^+}(\mathscr{C})$ to $\delta_{\pi^+}(\mathscr{D})$ and $f'$ $\delta_{\pi^+}(\mathscr{D})$ to $\delta_{\pi^+}(\mathscr{E})$, so we can take the new function $f'':\delta_{\pi^+}(\mathscr{C)\rightarrow\delta_{\pi^+}(\mathscr{E)}}$ with $f''(\delta)=f'(f(\delta))$ and function $g$ and $g'$,  let $g''(\mathscr{X})=g'(g(\mathscr{X}))$. So, by (3) $\mathfrak{U}^\mathscr{B}_a|\phi\mathscr{C}\mathscr{E}$.
    The Euclidean case follows analogously.

    To show this is an update relation we go through the conditions for modal relations in Definition \ref{modalRelation} and update sets in Definition \ref{def:UpdateRelations} and show that these hold for the relations defined above.
    
    As step (2) is the only one concerning inconsistent bases and remains unchanged from our proof in \cite{EckhardtPymS5}, conditions (a) and (b) hold trivially.
    

     We proceed to show (c) and (d). To show (c), take some $\mathfrak{U}^\mathscr{B}_a|\phi\mathscr{C}\mathscr{D}$. There are three cases to consider: 

    First, let $\mathscr{C}= \mathscr{C}^+$ and $\mathscr{D}=\mathscr{D}^+$ for some $\mathscr{C}^+$ and $\mathscr{D}^+$. We know that $\delta_{\pi^+}(\mathscr{C}^+) = \{\delta| \delta\subset \overline{S_A^*|\phi} \text{ and } \mathscr{C}^+\in \delta\}$ and $\delta_{\pi^+}(\mathscr{D}^+) = \{\delta| \delta\subset \overline{S_A^*|\phi} \text{ and } \mathscr{D}^+\in \delta\}$ and let $g(\mathscr{\mathscr{C}^+)}=\mathscr{D}^+$, $g(\mathscr{\mathscr{D}^+)}=\mathscr{C}^+$, and $g(\mathscr{E)=E}$ for all other $\mathscr{E}$. Accordingly $f(\delta)= \{g(\mathscr{E})|\mathscr{E}\in \delta\}$. Let $g'(\mathscr{E})=g(\mathscr{E})$ and let $f'(\delta) = \{g'(\mathscr{E})|\mathscr{E}\in\delta\}$. So, for any $\mathscr{D}'\supseteq\mathscr{D}$ s.t. $\delta_{\pi^+}(\mathscr{D}') = \{f'(\delta)|\delta\in\delta_{\pi^+}(\mathscr{C}')\}$, we have $\mathfrak{U}^\mathscr{B}_a|\phi\mathscr{C}'\mathscr{D}'$. We obtain such a $\mathscr{D}'$  by simply adding rules to $\mathscr{D}$ that correspond to those $\delta$ s.t. $\delta\in \delta_w(\mathscr{D}')$ but $\delta\notin \delta_w(\mathscr{D})$.
    
    For the second case, if $\mathscr{C}'= \mathscr{C}^+$, then there is a $\mathscr{D}^+\supseteq\mathscr{D}$ with $\mathfrak{U}^\mathscr{B}_a|\phi \mathscr{C}^+\mathscr{D}^+$, as $\mathfrak{U}^\mathscr{B}_a|\phi\mathscr{C}\mathscr{D}$ has to hold by (3).

    For all other $\mathscr{C}$, $\mathscr{D}$, and $\mathscr{C}'$, we know that $\mathfrak{U}^\mathscr{B}_a|\phi\mathscr{C}\mathscr{D}$ has to hold by (3) and so there are $f$ and $g$. Additionally, $\delta_{\pi^+}(\mathscr{C})\subseteq\delta_{\pi^+}(\mathscr{C}')$. So, this case follows the same as the case in which $\mathscr{C}=\mathscr{C}^+$:  For all $\mathscr{E}$, let $g'(\mathscr{E})=g(\mathscr{E})$ and let $f'(\delta) = \{g'(\mathscr{E})|\mathscr{E}\in\delta\}$. So, for any $\mathscr{D}'\supseteq\mathscr{D}$ s.t. $\delta_{\pi^+}(\mathscr{D}') = \{f'(\delta)|\delta\in\delta_{\pi^+}(\mathscr{C}')\}$, we have $\mathfrak{U}^\mathscr{B}_a|\phi\mathscr{C}'\mathscr{D}'$. We obtain such a $\mathscr{D}'$  by simply adding rules to $\mathscr{D}$ that correspond to those $\delta$ s.t. $\delta\in \delta_{\pi^+}(\mathscr{D}')$ but $\delta\notin \delta_{\pi^+}(\mathscr{D})$.

    Condition (d) follows if we straightforwardly modify this to deal with subsets rather than supersets. Note that where we had $\delta_{\pi^+}(\mathscr{C})\subseteq\delta_{\pi^+}(\mathscr{C}')$, we now have $\delta_{\pi^+}(\mathscr{C}')\subseteq\delta_{\pi^+}(\mathscr{C})$ instead and so we remove rules corresponding to the $\delta \in \delta_{\pi^+}(\mathscr{C})$ that are not in $\delta_{\pi^+}(\mathscr{C}')$ rather than adding new rules.

    



    


    To show that this is a update set, we first show that any $\pi^+_i$ only contains bases that are reachable from each other via the sub/superset. Take any $\mathscr{D}$ and $\mathscr{D}'$ in $\pi^+_i$ and let $\mathscr{D}'' = \mathscr{D}\cup\mathscr{D}'$. If $\mathscr{D}= \mathscr{D}'$ this is trivial. There are now two cases to consider: If it is not the case that $\mathscr{D}'' = \mathscr{C}^+$ for some $\mathscr{C}^+$, then simply $\mathscr{D}'' \in \pi^+_i$ and as it is the superset of both $\mathscr{D}$ and $\mathscr{D}'$ we are done. Otherwise, take some rule $r$ that is in $\mathscr{D}$ but not in $\mathscr{D}'$. The base $(\mathscr{D}'' - r)$ is then a superset of $\mathscr{D}'$. Now we remove a rule $r'$ that is in $\mathscr{D}'$ but not in $\mathscr{D}$ and get $((\mathscr{D}''-r)-r')$ adding $r$ back in results in $(\mathscr{D}''-r')$ which is a superset of $\mathscr{D}$. Obviously, $((\mathscr{D}''-r)-r')$ is a subset of both $(\mathscr{D}''-r)$ and $(\mathscr{D}'' -r')$. As we have only manipulated rules of $\mathscr{D}$ and $\mathscr{D}'$, $(\mathscr{D}'' -r)$, $((\mathscr{D}'' - r) - r')$, and $(\mathscr{D}'' -r')$ are members of $\pi^+_i$ and so $\mathscr{D}'$ can be reached from $\mathscr{D}$ through them.
    
    
    So, finally, we show that for $\mathscr{C}'\subseteq\mathscr{C}$ s.t. $\mathscr{C}'$ is reachable from $\mathscr{C}$, $\mathscr{C}$ and $\mathscr{C}'$ share access to the same bases.

    Step (1) only add relations between maximally-consistent sets. Two maximally-consistent sets can only be subsets of each other if they are the same set and so also share the same bases they have access to.

   Step (2) only adds relations between inconsistent sets and our update condition only restrict consistent sets.


    The only thing left is to look at the (3) steps. 
    Note that, for any sub- or superset of a base it either has the same $\delta_{\pi^+}$ or one with a higher or lower length. If the length is different, the super-/subset cannot be reachable from the inital base as $(3)$ requires the length to be equal for two bases to be connected, as it requires a bijective function between them, and if it's the same, then the bases get connected to the same bases by $(3)$. 
\end{proof}

As we have shown, the resulting relations are an update set, however, we also need to show that they constitute an effective update. 

\medskip 

 \begin{lemma} \label{lem:effective}
    For any update set $\mathfrak{R}_A$ and formulae $\phi$ such that $\Vdash_{\mathscr{B},\mathfrak{R}_A} \phi$, $\mathfrak{U}_A^\mathscr{B}|\phi$ is an effective update of $\mathfrak{R}_A$ at $\mathscr{B}$ by $\phi$, that is, $\mathfrak{U}_A^\mathscr{B}|\phi \in \mathfrak{R}_A|\phi^\mathscr{B}$. 
    \end{lemma}
    \begin{proof}
    As shown in Lemma \ref{PALrelationIsModal}, $\mathfrak{U}_A^\mathscr{B}|\phi$ is an update set.
    To show that the update was effective we have to to show that the two conditions in Definition \ref{def:effective} are fulfilled. 
    Obviously, $P^+ \supseteq P$.
    
    The second condition is that for every $\mathscr{C}$ reachable from $\mathscr{B}$ there is a $\mathscr{C}'\supseteq\mathscr{C}$ s.t. $\mathscr{C}'\in\Omega_{P'}$ and for all classical formulae $\phi$ containing only atomic formulae in $P$ $\Vdash_{C',\mathfrak{R}'_A} \phi$ iff $\Vdash_{C,\mathfrak{R}_A} \phi$. It is easy to see that for any $\mathscr{C}$ there is such a $\mathscr{C}^+$. For formulae involving $\bot$, note that since $(q\Rightarrow p_i)$ for all the $p_i$ added, in any circumstance in which $\bot$ held at $\mathscr{C}$, it will still hold at $\mathscr{C}^+$. 
        
    The third condition is that for all agents $a\in A$ and bases $\mathscr{C}$ and $\mathscr{D}$ reachable by $\mathfrak{R}_A$ from $\mathscr{B}$, $\mathfrak{R}'_a\mathscr{C}\mathscr{D}$ iff $\mathfrak{R}_a\mathscr{C}\mathscr{D}$, $\Vdash_{\mathscr{C},\mathfrak{R}_A} \phi$, and $\Vdash_{\mathscr{D},\mathfrak{R}_A} \phi$.
    It suffices to point out that $S^*_A|\phi$ fulfills this criteria and the construction $\mathfrak{U}_A^\mathscr{B}|\phi$ only alter proper sub- and supersets of the bases in $|S^*_A|\phi^\mathscr{B}|$, but not those bases themselves. 
    \end{proof}



\subsection*{Equivalence of Effective Updates}
\medskip 


        

    \EquivalenceEffUp*

    \begin{proof}

    We show this by induction on $\chi$. Note that, for atomic $\chi$ and $\chi =\bot$ this follows directly from the definition of effective updates in Definition \ref{def:effective}.

    For $\chi=(\chi_1\to\chi_2)$, it suffices to point out that, by induction hypothesis, $\Vdash_{\mathscr{B}^1,\mathfrak{R}^1}\chi_1$ iff $\Vdash_{\mathscr{B}^2,\mathfrak{R}^2}\chi_1$ and the same for $\chi_2$.

    For $\chi= K_a\chi_1$, note that for all $\mathscr{C}^1$ s.t. $\mathfrak{R}^1_A \mathscr{B}^1\mathscr{C}^1$, by Definition \ref{def:effective}, there is a $\mathscr{C}$ s.t. $\mathfrak{R}_A \mathscr{B}\mathscr{C}$ and so there is also going to be a $\mathscr{C}^2$ s.t. $\mathfrak{R}^2_A \mathscr{B}^2 \mathscr{C}^2$ and vice versa. By induction hypothesis, $\Vdash_{\mathscr{C}^1,\mathfrak{R}^1}\chi_1$ iff $\Vdash_{\mathscr{C}^2,\mathfrak{R}^2}\chi_1$.

    For $\chi=[\chi_1]\chi_2$, $\nVdash_{\mathscr{B}^1,\mathfrak{R}^1}\chi_1$ iff $\nVdash_{\mathscr{B}^2,\mathfrak{R}^2}\chi_1$. However, if $\Vdash_{\mathscr{B}^1,\mathfrak{R}^1}\chi_1$ and $\Vdash_{\mathscr{B}^2,\mathfrak{R}^2}\chi_1$, then note that an effective update of $\mathfrak{R}_A^1$ at $\mathscr{B}^1$ by $\chi_1$ is an effective update of $\mathscr{R}_A$ at $\mathscr{B}$ by $\Delta;\chi_1$ and the same goes for effective updates of $\mathfrak{R}_A^2$ at $\mathscr{B}^2$ by $\chi_1$. So, by induction hypothesis, $\Vdash^{\chi_1}_{\mathscr{B}^1,\mathfrak{R}^1}\chi_2$ iff $\Vdash^{\chi_1}_{\mathscr{B}^2,\mathfrak{R}^2}\chi_2$.
    \end{proof}

\subsection*{Soundness}
\medskip 

\begin{lemma}
    \label{PALMonotonicity}
     If $\Gamma\Vdash_{\mathscr{B},\mathfrak{R}_A}^{\Delta} \phi$ and $\mathscr{B}\subseteq\mathscr{C}$, then $\Gamma\Vdash_{\mathscr{C},\mathfrak{R}_A}^{\Delta} \phi$.
        
\end{lemma}

\begin{proof}
    For empty $\Delta$, it suffices to point out that the only new case compared to Lemma \ref{ModalMonotonicity} is the case that $\phi=[\psi]\chi$ which follows immediately from Definition \ref{PALValidity}.

    The cases for non-empty $\Delta$ follow in exactly the same way for all cases except $\phi=K_a\psi$. Note that this case follows immediately from Definition \ref{PALValidity} as well. 
    \end{proof}
     




We also show that every formula holds at inconsistent bases.

\medskip 

\begin{lemma}
    \label{lem:PALEFQ}
    For all update sets $\mathfrak{R}_A$, $\Delta$, and bases $\mathscr{B}$, if $\mathscr{B}$ is inconsistent, then $\Vdash_{\mathscr{B},\mathfrak{R}_A}^{\Delta}\phi$ for all $\phi$ s.t. $\phi$ only contains atoms of $P$.

    
\end{lemma}

\begin{proof}
    We prove this by induction on the complexity of $\phi$. For $\phi = p$ and $\phi = \bot$ this follows immediately from the definition of inconsistent bases and the validity conditions in Definition \ref{PALValidity}, since $p\in P$. For $\phi= \psi\to\chi$, we have $\Vdash_{\mathscr{B},\mathfrak{R}_A}^{\Delta} \psi\to\chi$ iff $\phi\Vdash_{\mathscr{B},\mathfrak{R}_A}^{\Delta} \psi$. Since all $\mathscr{C}\supseteq\mathscr{B}$ s.t. $\Vdash_{\mathscr{C},\mathfrak{R}_A}^{\Delta} \psi$ will also be inconsistent by Lemma \ref{PALMonotonicity}, we know that $\Vdash_{\mathscr{B},\mathfrak{R}_A}^{\Delta} \chi$ by the induction hypothesis. 
    
    
    For the next case let $\phi = K_a \psi$. By Lemma \ref{PALMonotonicity}, all $\mathscr{C}\supseteq\mathscr{B}$ are also inconsistent and, for the case in which $\Delta$ is non-empty, so are all $\mathscr{D}$ for which there is a $\mathfrak{R}' \in \mathfrak{R}_A|\Delta^\mathscr{C}$ with $\mathfrak{R}_a\mathscr{C}\mathscr{D}$ and we have $\Vdash_{\mathscr{D},\mathfrak{R}_A}^{\Delta} \psi$ by the induction hypothesis. For the case in which $\Delta$ is empty, it suffices to point out that, due to condition $(a)$, all bases $\mathscr{X}$ s.t. $\mathfrak{R}_a \mathscr{B}\mathscr{X}$ are also inconsistent and $\Vdash_{\mathscr{D},\mathfrak{R}_A}^{\Delta} \psi$ by the induction hypothesis.

    Finally, let $\phi = [\psi]\chi$. For any $\mathscr{C}\supseteq\mathscr{B}$, if we have $\Vdash_{\mathscr{C},\mathfrak{R}_A}^{\Delta} \bot$, then we also have $\Vdash_{\mathscr{C},\mathfrak{R}_A}^{\Delta,\phi} \bot$ by the validity conditions of $\bot$. So, by the induction hypothesis, $\Vdash_{\mathscr{C},\mathfrak{R}_A}^{\Delta,\phi} \chi$ and $\Vdash_{\mathscr{C},\mathfrak{R}_A}^{\Delta} [\psi]\chi$. 
    %
\end{proof}

Given this, we can now look at maximally-consistent bases which will turn out to be helpful for proving the completeness of our semantics.

We show that whether a formula holds at a maximally-consistent base can be established without referencing any of its supersets. 

\medskip 

\begin{lemma}
\label{PALBehaviour}
For any update set $\mathfrak{R}_A$, $\Delta$, and maximally-consistent base $\mathscr{B}$, the following hold for formulae limited to atoms in $P$:

\[
\begin{array}{l@{\quad}c@{\quad}l}

    \nVdash_{\mathscr{B},\mathfrak{R}_A}^{\Delta} \bot & &\\
    
    \Vdash_{\mathscr{B},\mathfrak{R}_A}^{\Delta} \phi \to \psi & \mbox{iff} & \nVdash_{\mathscr{B},\mathfrak{R}_A}^{\Delta} \phi \mbox{ or } \Vdash_{\mathscr{B},\mathfrak{R}_A}^{\Delta} \psi\\
    
    \Vdash_{\mathscr{B},\mathfrak{R}_A}^{\Delta} K_a \phi & \mbox{iff} & \mbox{for all $ \mathscr{C}$ s.t. $\mathfrak{R}_a \mathscr{B}\mathscr{C}, \Vdash_{\mathscr{C},\mathfrak{R}_A}^{\Delta} \phi$}\\
    
    \Vdash_{\mathscr{B},\mathfrak{R}_A}^{\Delta} [\phi]\psi & \mbox{iff} & \mbox{  $\Vdash_{\mathscr{B},\mathfrak{R}_A}^{\Delta} \phi$ implies $\Vdash_{\mathscr{B},\mathfrak{R}_A}^{\Delta,\phi} \psi$}\\



    
    
    
    
\end{array}\]

\end{lemma}

\begin{proof}


    The $\bot$-cases follow from directly from consistency and our definition of effective updates.

    For the other cases it suffices to point out that any $\mathscr{B}'\supset\mathscr{B}$ is inconsistent and so by Lemma \ref{lem:PALEFQ} any formula limited to $P$ is supported at them given any $\mathfrak{R}_A$ or $\mathfrak{R}_A|\Delta^{\mathscr{B}'}$.
\end{proof} 

From this, it follows that, for formulae limited to $P$, the {\it law of excluded middle} holds on maximally-consistent bases: for every formula $\phi$, either $\phi$ or $\phi\to\bot$.

\medskip 

\begin{lemma}
    \label{lem:MaxConOr}
For every formula $\phi$ limited to $P$, maximally-consistent base $\mathscr{B}$, $S5$-modal relations $\mathfrak{R}_A$, and $\Delta$, either $\Vdash_{\mathscr{B},\mathfrak{R}_A}^{\Delta} \phi$ or $\Vdash_{\mathscr{B},\mathfrak{R}_A}^{\Delta} \phi\to\bot$. 

\end{lemma}

\begin{proof}

    We assume $\nVdash_{\mathscr{B},\mathfrak{R}_A}^{\Delta} \phi$ and show $\Vdash_{\mathscr{B},\mathfrak{R}_A}^{\Delta} \phi\to\bot$. 
    
    By Lemma \ref{PALBehaviour}, $\Vdash_{\mathscr{B},\mathfrak{R}_A}^{\Delta} \phi\to\bot$ iff if $\Vdash_{\mathscr{B},\mathfrak{R}_A}^{\Delta} \phi$, then $\Vdash_{\mathscr{B},\mathfrak{R}_A}^{\Delta} \bot$. So, since we assumed $\nVdash_{\mathscr{B},\mathfrak{R}_A}^{\Delta} \phi$, the antecedent is false and so $\Vdash_{\mathscr{B},\mathfrak{R}_A}^{\Delta} \phi\to\bot$.
    %
\end{proof}

There is one more important property of maximally-consistent bases that we require for the soundness proof. If a formula does not hold at a base, then that base can be extended to a maximally-consistent base on which the formula does not hold.


To show this we need two more results about the compositionality of announcements.

\medskip 

\begin{lemma}
\label{CompR}
   For any formulae $\phi$ and $\psi$, base $\mathscr{B}$, and update set $\mathfrak{R}_A$, 
   \begin{enumerate}
       \item[—]for any $\mathfrak{R}^{\phi\wedge[\phi]\psi}_A \in \mathfrak{R}_A|(\phi\wedge[\phi]\psi)^\mathscr{B}$ there is a $\mathfrak{R}^{\phi;\psi}_A \in \mathfrak{R}_A|(\phi;\psi)^\mathscr{B}$ s.t. for any formula $\chi$ limited to atoms in $P$, for any update $\mathscr{B}'$ of $\mathscr{B}$ for $\mathfrak{R}^{\phi\wedge[\phi]\psi}_A$ $\Vdash_{\mathscr{\mathscr{B}', \mathfrak{R}^{\phi\wedge[\phi]\psi}}_A}^{\phi\wedge[\phi]\psi} \chi$ iff for any update $\mathscr{B}^*$ of $\mathscr{B}$ for $\mathfrak{R}^{\phi;\psi}_A$ $\Vdash_{\mathscr{\mathscr{B}^*, \mathfrak{R}^{\phi;\psi}}_A}^{\phi;\psi} \chi$     
       
       \item[—]for any $\mathfrak{R}^{\phi;\psi}_A \in \mathfrak{R}_A|(\phi;\psi)^\mathscr{B}$ there is a $\mathfrak{R}^{\phi\wedge[\phi]\psi}_A \in \mathfrak{R}_A|(\phi\wedge[\phi]\psi)^\mathscr{B}$ s.t. for any formula $\chi$ limited to atoms in $P$, for any update $\mathscr{B}^*$ of $\mathscr{B}$ for $\mathfrak{R}^{\phi;\psi}_A$ $\Vdash_{\mathscr{\mathscr{B}^*, \mathfrak{R}^{\phi;\psi}}_A}^{\phi;\psi} \chi$ iff for any update $\mathscr{B}'$ of $\mathscr{B}$ for $\mathfrak{R}^{\phi\wedge[\phi]\psi}_A$ $\Vdash_{\mathscr{\mathscr{B}', \mathfrak{R}^{\phi\wedge[\phi]\psi}}_A}^{\phi\wedge[\phi]\psi} \chi$.
   \end{enumerate}

  
\end{lemma}

\begin{proof}
 
To prove the second part of this lemma we simply show that any $\mathfrak{R}^{\phi;\psi}_A$ is going to be an effective update by $\phi\wedge[\phi]\psi$. For any $\mathscr{C}^*$ to be an update of $\mathscr{C}$ for $\mathfrak{R}^{\phi;\psi}_A$, note that $\Vdash_{\mathscr{C},\mathfrak{R}_A} \phi$ and for the $\mathscr{C}^1$ and $\mathfrak{R}_A^\phi$ s.t. $\mathscr{C}^1$ is an update of $\mathscr{C}$ for $\mathfrak{R}_A^{\phi}$ $\Vdash^\phi_{\mathscr{C}^1, \mathfrak{R}_A^\phi} \psi$. By Lemma \ref{lem:EquivalenceOfEffectiveUpdates}, it follows that $\Vdash_{\mathscr{C}, \mathfrak{R}_A} [\phi]\psi$ and so $\mathfrak{R}^{\phi;\psi}_A \in \mathfrak{R}_A|(\phi\wedge[\phi]\psi)^\mathscr{B}$.

The first part follows directly from this and Lemma~\ref{lem:EquivalenceOfEffectiveUpdates}.
\end{proof}

\medskip 

\begin{cor}
    \label{DELTAformula}

    For every string of announcements $\Delta = \phi_1; \dots ; \phi_n$, there is a single formula $\phi(\Delta)$ s.t., for every $\mathscr{B}, \mathscr{C}$ and $\mathscr{D}$ and $\mathfrak{R}_A$ s.t. $\mathfrak{R}_A \mathscr{C}\mathscr{D}$, there is a $\mathfrak{R}^1_A$ that is an effective update of $\mathfrak{R}_A$ by $\phi(\Delta)$ at $\mathscr{B}$ with updates $\mathscr{C^1}$ and $\mathscr{D}^1$ s.t. $\mathfrak{R}^1_a \mathscr{C}^1\mathscr{D}^1$ iff there is a $\mathfrak{R}^2_A$ that is an effective update of $\mathfrak{R}_A$ by $\Delta$ at $\mathscr{B}$ with updates $\mathscr{C}^2$ and $\mathscr{D}^2$ s.t. $\mathfrak{R}^2_a \mathscr{C}^2\mathscr{D}^2$. 
    Specifically, for non-empty $\Delta$, $\phi(\Delta) = \phi_1 \wedge [\phi_1](\phi_2 \wedge [\phi_2](\phi_3\wedge \dots \wedge [\phi_{n-1}] \phi_n))$. For empty $\Delta$, $\phi(\Delta)$ can be any tautology. 
\end{cor} \medskip 

Now we can show the lemma we initially wanted.

\medskip 

\begin{lemma}
\label{lem: PALMaxCon}

For any string of announcements $\Delta$, formula $\phi$, base $\mathscr{B}$ and update set $\mathfrak{R}_A$, if $\Vdash_{\mathscr{B},\mathfrak{R}_A} \phi(\Delta)$ and $\nVdash_{\mathscr{B},\mathfrak{R}_A}^{\Delta} \phi$, then there exists a maximally-consistent base $\mathscr{B}^*\supseteq\mathscr{B}$ with $\nVdash_{\mathscr{B}^*,\mathfrak{R}_A}^{\Delta} \phi$.




\end{lemma}

\begin{proof}



    We proof this by induction on the complexity of $\phi$. The cases for $\phi= p$ and $\phi=\bot$ are straightforward, as the choice of $\mathfrak{R}$ is irrelevant to them.

    Note that, for empty $\Delta$ the cases for $\phi = K_a\psi$ and $\phi= \psi\to\chi$ follow immediately from Lemma \ref{lem: ModalMaxCon}. We proceed to show these cases for non-empty $\Delta$.

    For $\phi= \psi\to\chi$, we start with $\nVdash_{\mathscr{B},\mathfrak{R}_A}^{\Delta} \psi\to\chi$. So, there is a $\mathscr{C}\supseteq \mathscr{B}$ s.t. $\Vdash_{\mathscr{C},\mathfrak{R}_A}^{\Delta} \psi$ and $\nVdash_{\mathscr{C},\mathfrak{R}_A}^{\Delta} \chi$. By induction hypothesis, there is a maximally-consistent $\mathscr{C}^*\supseteq\mathscr{C}$ with $\nVdash_{\mathscr{C}^*,\mathfrak{R}_A}^{\Delta} \chi$ and, since $\mathscr{C}^*\supseteq\mathscr{C}$, $\Vdash_{\mathscr{C}^*,\mathfrak{R}_A}^{\Delta} \psi$. So, $\nVdash_{\mathscr{C}^*,\mathfrak{R}_A}^{\Delta} \psi\to \chi$.

    The interesting case is the one in which $\phi= K_a \psi$. Given $\nVdash_{\mathscr{B},\mathfrak{R}_A}^{\Delta} K_a\psi$, there are $\mathscr{X}\supseteq\mathscr{B}$, $\mathscr{C}$ and $\mathfrak{R}'_A \in \mathfrak{R}_A|\Delta^\mathscr{X}$ s.t. $\nVdash_{\mathscr{C},\mathfrak{R}_A}^{\Delta} \psi$, by induction hypothesis, there is a maximally-consistent $\mathscr{C}^*\supseteq\mathscr{C}$ with $\nVdash^\Delta_{\mathscr{C}^*,\mathfrak{R}_A} \psi$. Let $\mathscr{X}_1$ and $\mathscr{C}_1$ be the corresponding updates of $\mathscr{X}$ and $\mathscr{C}$. From $\mathfrak{R}'_a \mathscr{X}_1\mathscr{C}_1$, we know that $\mathfrak{R}_a \mathscr{X}\mathscr{C}$ and, by symmetry, $\mathfrak{R}_a\mathscr{C}\mathscr{X}$. Since $\mathfrak{R}_a$ is an update relation, we know that, by condition (c) of Definition \ref{modalRelation}, there is a $\mathscr{X}'\supseteq\mathscr{X}$ s.t. $\mathfrak{R}_a \mathscr{C}^* \mathscr{X}'$ and, again by symmetry, that $\mathfrak{R}_a \mathscr{X}' \mathscr{C}^*$. Again by (c), for every maximally-consistent $\mathscr{X}^* \supseteq\mathscr{X}'$, $\mathfrak{R}_a\mathscr{X}^*\mathscr{C}^*$, because $\mathscr{C}^*$ does not have any consistent superset bases other than itself. Now we only need to show that the relation still holds between the updates of them for some relation $\mathfrak{R}''_A \in \mathfrak{R}_A|\Delta^{\mathscr{X}^*}$.
    
    Since, by assumption, $\Vdash_{\mathscr{X},\mathfrak{R}_A} \phi(\Delta)$ and $\mathscr{X}^*\supseteq\mathscr{X}$, we also have $\Vdash_{\mathscr{X}^*,\mathfrak{R}_A} \phi(\Delta)$. Similarly, because of $\mathfrak{R}'_a \mathscr{X}^+\mathscr{C}^+$, we know that $\Vdash_{\mathscr{C},\mathfrak{R}_A} \phi(\Delta)$ and so $\Vdash_{\mathscr{C}^*,\mathfrak{R}_A} \phi(\Delta)$. So, we have $\mathfrak{R}''_a\mathscr{X}_2^*\mathscr{C}_2^*$, where $\mathscr{X}^*_2$ and $\mathscr{C}_2^*$ are the corresponding updates of $\mathscr{X}$ and $\mathscr{C}$, and, so, $\nVdash_{\mathscr{X}^*,\mathfrak{R}_A}^{\Delta} K_a \psi$.

    Finally, consider the case that $\phi=[\psi]\chi$. Take $\nVdash_{\mathscr{B},\mathfrak{R}_A}^{\Delta} [\psi]\chi$, so there is a $\mathscr{C}\supseteq\mathscr{B}$ s.t. $\Vdash_{\mathscr{C},\mathfrak{R}_A}^{\Delta} \psi$ but $\nVdash_{\mathscr{C},\mathfrak{R}_A}^{\Delta,\phi} \chi$. So, by induction hypothesis, there is a maximally-consistent such $\mathscr{C}^*\supseteq\mathscr{C}$ and we are done. 
\end{proof}

 We can now finally show that the rules and axioms of modal logic hold for our base-extension semantics of PAL. We break this up to make it easier on the reader and start with the classical axioms and modus ponens.


\medskip 

    \begin{lemma}
    \label{lem:ClassicalAxioms}
        The rules and axioms of $S5$ modal logic are valid on the base-extension semantics for PAL.

        \begin{itemize}
        \item[—](1) $\phi\Vdash \psi\to\phi$ 
        \item[—](2) $\phi\to(\psi\to\chi)\Vdash (\phi\to\psi)\to(\phi\to\chi)$
        \item[—](3) $(\phi\to\bot)\to(\psi\to\bot)\Vdash \psi\to\phi$
        \item[—](MP) If $\Vdash_{\mathscr{B},\mathfrak{R}_A} \phi\to\psi$ and $\Vdash_{\mathscr{B},\mathfrak{R}_A} \phi$, then $\Vdash_{\mathscr{B},\mathfrak{R}_A} \psi$ 

        \end{itemize}
    \end{lemma}

    \begin{proof}
            We start by showing \emph{(MP)} as it is helpful in establishing the other results. Simply note that because $\Vdash_{\mathscr{B},\mathfrak{R}_A} \phi\to\psi$, we have for all $\mathscr{C}\supseteq\mathscr{B}$, if $\Vdash_{\mathscr{C},\mathfrak{R}_A} \phi$, then $\Vdash_{\mathscr{C},\mathfrak{R}_A} \psi$. By assumption, $\Vdash_{\mathscr{B},\mathfrak{R}_A} \phi$ and so $\Vdash_{\mathscr{B},\mathfrak{R}_A} \psi$.

    By Definition \ref{PALValidity}, for \emph{(1)} we need to show that for any $\mathscr{B}$, if $\Vdash_{\mathscr{B},\mathfrak{R}_A} \phi$, then $\Vdash_{\mathscr{B},\mathfrak{R}_A} \psi\to\phi$. By Lemma \ref{PALMonotonicity}, for any $\mathscr{C}\supseteq\mathscr{B}$ s.t. $\Vdash_{\mathscr{C},\mathfrak{R}_A} \psi$ we also have $\Vdash_{\mathscr{C},\mathfrak{R}_A} \phi$ and so $\Vdash_{\mathscr{B},\mathfrak{R}_A} \psi\to\phi$.

    For \emph{(2)} we follow the same strategy of showing that if the formula on the left of the turnstyle holds at a base, so does the right. To show that $\Vdash_{\mathscr{B},\mathfrak{R}_A} (\phi\to\psi)\to(\phi\to\chi)$, we need to show that for all $\mathscr{C}\supseteq\mathscr{B}$, if $\Vdash_{\mathscr{C},\mathfrak{R}_A} \phi\to\psi$, then $\Vdash_{\mathscr{C},\mathfrak{R}_A} \phi\to\chi$. Take some $\mathscr{D}\supseteq\mathscr{C}$ s.t. $\Vdash_{\mathscr{D},\mathfrak{R}_A} \phi\to(\psi\to\chi)$, $\Vdash_{\mathscr{D},\mathfrak{R}_A} \phi\to\psi$ and $\Vdash_{\mathscr{D},\mathfrak{R}_A} \phi$. By multiple application of \emph{(MP)}, we get $\Vdash_{\mathscr{D},\mathfrak{R}_A} \chi$ and so, $\Vdash_{\mathscr{C},\mathfrak{R}_A} \phi\to\chi$ and, finally, $\Vdash_{\mathscr{B},\mathfrak{R}_A} (\phi\to\psi)\to(\phi\to\chi)$.

    For \emph{(3)} we take some $\mathscr{B}$ with $\Vdash_{\mathscr{B},\mathfrak{R}_A} (\phi\to\bot)\to(\psi\to\bot)$. To show that $\Vdash_{\mathscr{B},\mathfrak{R}_A} \psi\to\phi$, we show that for all $\mathscr{C}\supseteq\mathscr{B}$ s.t. $\Vdash_{\mathscr{C},\mathfrak{R}_A} \psi$ we also have $\Vdash_{\mathscr{C},\mathfrak{R}_A} \phi$. Assume $\nVdash_{\mathscr{C},\mathfrak{R}_A} \phi$ for contradiction. By Lemma \ref{lem:MaxConOr}, all maximally-consistent $\mathscr{C}^*\supseteq\mathscr{C}$ either $\Vdash_{\mathscr{C}^*,\mathfrak{R}_A} \phi$ or $\Vdash_{\mathscr{C}^*,\mathfrak{R}_A} \phi\to\bot$. By Lemma \ref{PALMonotonicity}, we know that $\Vdash_{\mathscr{C}^*,\mathfrak{R}_A} (\phi\to\bot)\to(\psi\to\bot)$ and $\Vdash_{\mathscr{C}^*,\mathfrak{R}_A} \psi$. So, if $\Vdash_{\mathscr{C}^*,\mathfrak{R}_A} \phi\to\bot$, then, by \emph{(MP)}, $\Vdash_{\mathscr{C}^*,\mathfrak{R}_A} \psi\to\bot$ and $\Vdash_{\mathscr{C}^*,\mathfrak{R}_A} \bot$. This is a contradiction as $\mathscr{C}^*$ is, by assumption, consistent. So, $\phi$ holds at all maximally-consistent superset bases of $\mathscr{C}$ and so, by the contrapositive of Lemma \ref{lem: PALMaxCon}, $\Vdash_{\mathscr{C},\mathfrak{R}_A} \phi$ and so $\Vdash_{\mathscr{B},\mathfrak{R}_A} \psi \to \phi$.
    \end{proof}

    This concludes the classical portion. We continue with the modal axioms and (modal) \emph{(NEC)}.

\medskip 

 \begin{lemma}
 \label{lem:S5Axioms}
     The rules and axioms of $S5$ modal logic are valid on the base-extension semantics for PAL.

    \begin{itemize}
        \item[—](K) $K_a(\phi\to\psi)\Vdash K_a\phi\to K_a\psi$
        \item[—](T)  $K_a\phi\Vdash \phi$
        \item[—](4) $K_a \phi\Vdash K_a K_a\phi$
        \item[—](5) $\neg K_a \phi\Vdash K_a\neg K_a\phi$
        \item[—](NEC) If $\Vdash \phi$, then $\Vdash K_a\phi$
        
    \end{itemize}
 \end{lemma}

\begin{proof}


    For \emph{(K)}, we show the contrapositive and assume $\nVdash_{\mathscr{B},\mathfrak{R}_A} K_a\phi\to K_a\psi$ for an arbitrary $\mathscr{B}$ and $\mathfrak{R}_A$. It follows that there is a $\mathscr{C}\supseteq\mathscr{B}$ s.t. $\Vdash_{\mathscr{C},\mathfrak{R}_A} K_a\phi$ and $\nVdash_{\mathscr{C},\mathfrak{R}_A} K_a\psi$. So, there are $\mathscr{D}\supseteq\mathscr{C}$ and $\mathscr{E}$ s.t. $\mathfrak{R}_a\mathscr{D}\mathscr{E}$ with $\Vdash_{\mathscr{E},\mathfrak{R}_A} \phi$ but $\nVdash_{\mathscr{E},\mathfrak{R}_A} \psi$ and so $\nVdash_{\mathscr{E},\mathfrak{R}_A} \phi\to\psi$ and $\nVdash_{\mathscr{D},\mathfrak{R}_A} K_a(\phi\to\psi)$. Since $\mathscr{D}\supseteq\mathscr{B}$, $\nVdash_{\mathscr{B},\mathfrak{R}_A} K_a(\phi\to\psi)$.  

    For \emph{(T)}, take some $\Vdash_{\mathscr{B},\mathfrak{R}_A} K_a\phi$. By the reflexivity of $\mathfrak{R}_a$, $\mathfrak{R}_A\mathscr{B}\mathscr{B}$ and so $\Vdash_{\mathscr{B},\mathfrak{R}_A} \phi$.

For \emph{(4)}, we, without loss of generality, assume some $\mathscr{B}$ and $\mathfrak{R}_A$ s.t. $\Vdash_{\mathscr{B}, \mathfrak{R}_A} K_a \phi$. For contradiction, we also assume $\nVdash_{\mathscr{B}, \mathfrak{R}_A} K_a K_a\phi$. So, there are $\mathscr{C}\supseteq\mathscr{B}$ and $\mathscr{D}$ s.t. $\mathfrak{R}_a\mathscr{C}\mathscr{D}$ and $\nVdash_{\mathscr{D}, \mathfrak{R}_A} K_a \phi$. This means there have to be $\mathscr{E}\supseteq\mathscr{D}$ and $\mathscr{F}$ s.t. $\mathfrak{R}_a\mathscr{E}\mathscr{F}$ and $\nVdash_{\mathscr{F}, \mathfrak{R}_A} \phi$. By Definition \ref{modalRelation} $(d)$, we know there is a $\mathscr{G}\subseteq\mathscr{F}$ and $\mathfrak{R}_a \mathscr{D}\mathscr{G}$ and, by transitivity, $\mathfrak{R}_a\mathscr{C}\mathscr{G}$. Obviously, $\nVdash_{\mathscr{G}, \mathfrak{R}_A} \phi$ and so $\nVdash_{\mathscr{B}, \mathfrak{R}_A} K_a \phi$, which contradicts our assumption.
    
    For \emph{(5)}, we assume, without loss of generality, some $\mathscr{B}$ and $\mathfrak{R}_A$ s.t. $\Vdash_{\mathscr{B}, \mathfrak{R}_A} \neg K_a \phi$ and, for contradiction, $\nVdash_{\mathscr{B}, \mathfrak{R}_A} K_a\neg K_a \phi$. So, there is $\mathscr{C}$ s.t. $\mathfrak{R}_a\mathscr{B}\mathscr{C}$ and $\nVdash_{\mathscr{C}, \mathfrak{R}_A} \neg K_a \phi$. So, there is a $\mathscr{D}\supseteq\mathscr{C}$ s.t $\mathscr{D}$ is consistent and $\Vdash_{\mathscr{D}, \mathfrak{R}_A} K_a \phi$.

    From $\Vdash_{\mathscr{B},\mathfrak{R}_A} \neg K_a \phi$ it follows that for all consistent $\mathscr{E}\supseteq\mathscr{B}$, we have $\nVdash_{\mathscr{B}, \mathfrak{R}_A} K_a \phi$. So there is a $\mathscr{F}$ s.t. $\mathfrak{R}_a\mathscr{E}\mathscr{F}$ and $\nVdash_{\mathscr{F}, \mathfrak{R}_A} \phi$.

    By symmetry, we have $\mathfrak{R}_a \mathscr{C}\mathscr{B}$ and so, by condition (c) of Definition \ref{modalRelation}, for one of these $\mathscr{E}\supseteq\mathscr{B}$ we have $\mathfrak{R}_a\mathscr{D}\mathscr{E}$ and, again by symmetry, $\mathfrak{R}_a \mathscr{E}\mathscr{D}$. By Euclidean and $\mathfrak{R}_a\mathscr{E}\mathscr{F}$, we have $\mathfrak{R}_a\mathscr{D}\mathscr{F}$ which gives us our contradiction: We have $\Vdash_{\mathscr{D}, \mathfrak{R}_A} K_a \phi$ and $\nVdash_{\mathscr{F}, \mathfrak{R}_A} \phi$.

    \iffalse{
    
    For (5), we assume, without loss of generality, some $\mathscr{B}$ and $\mathfrak{R}$ s.t. $\Vdash^\gamma_{\mathscr{B}, \mathfrak{R}} \lozenge \phi$ and, for contradiction, $\nVdash^\gamma_{\mathscr{B}, \mathfrak{R}} \square\lozenge \phi$. So, there are $\mathscr{C}\supseteq\mathscr{B}$ and $\mathscr{D}$ s.t. $\mathfrak{R}\mathscr{C}\mathscr{D}$ and $\nVdash^\gamma_{\mathscr{D}, \mathfrak{R}} \lozenge \phi$. So, there are $\mathscr{E}\supseteq\mathscr{D}$ and for all $\mathscr{F}$ s.t. $\mathfrak{R}\mathscr{E}\mathscr{F}$ $\nVdash^\gamma_{\mathscr{F}, \mathfrak{R}} \phi$.

    \noindent By $\Vdash^\gamma_{\mathscr{B}, \mathfrak{R}} \lozenge \phi$, we know there is a $\mathscr{G}$ s.t. $\mathfrak{R}\mathscr{C}\mathscr{G}$ and $\Vdash^\gamma_{\mathscr{G}, \mathfrak{R}} \phi$ and, by Euclidean, $\mathfrak{R}\mathscr{D}\mathscr{G}$. By (c), we know that there is an $\mathscr{H}\supseteq\mathscr{G}$ s.t. $\mathfrak{R}\mathscr{E}\mathscr{H}$ and, by Lemma \ref{ModalMonotonicity}, $\Vdash^\gamma_{\mathscr{H}, \mathfrak{R}} \phi$. So, we have a contradiction.

DIAMOND!
}\else

    We show \emph{(NEC)} via the contrapositive. If $\nVdash K_a\phi$, then $\nVdash \phi$. Suppose $\nVdash K_a\phi$. So, there is a $\mathscr{B}$ and a $\mathfrak{R}_A$ s.t. $\nVdash_{\mathscr{B},\mathfrak{R}_A} K_a\phi$. Therefore, there are $\mathscr{C}\supseteq\mathscr{B}$ and $\mathscr{D}$ s.t. $\mathfrak{R}_a\mathscr{C}\mathscr{D}$ and $\nVdash_{\mathscr{D},\mathfrak{R}_A} \phi$ and, so, $\nVdash\phi$.
\end{proof}

Before showing the axioms of PAL, we show Lemma \ref{lem:annK}. It is used in the proof of \emph{announcement and knowledge}.

\medskip 

\begin{lemma}
    \label{lem:annK}

For all $\mathscr{B}$, $\mathfrak{R}_A$, and $\phi$ the following are equivalent:

\begin{enumerate}
    \item[—]For all $\mathscr{C}\supseteq \mathscr{B}$, $\mathfrak{R}'_A \in \mathfrak{R}_A|\phi^\mathscr{C}$ and $\mathscr{D}$ s.t. for all corresponding updates $\mathscr{C}^+$ and $\mathscr{D}^+$ of $\mathscr{C}$ and $\mathscr{D}$, $\mathfrak{R}'_a \mathscr{C}^+\mathscr{D}^+$ , $\Vdash_{\mathscr{D},\mathfrak{R}_A}^{\phi}\psi$. 
    
    \item[—]For all $\mathscr{X}$ s.t. $\mathfrak{R}_a \mathscr{B}\mathscr{X}$ and $\mathscr{Y}\supseteq\mathscr{X}$, if $\Vdash_{\mathscr{Y},\mathfrak{R}_A} \phi$, then $\Vdash_{\mathscr{Y},\mathfrak{R}_A}^{\phi}\psi$.
\end{enumerate}
    
\end{lemma}

\begin{proof}
    We show that any $\mathscr{D}$ in the former corresponds to some $\mathscr{Y}$ in the latter and vice versa. 

    For the direction from 1 to 2, since $\mathfrak{R}_a'\mathscr{C}^+\mathscr{D}^+$, we know that for the corresponding $\mathscr{C}$ and $\mathscr{D}$ before the update, $\mathfrak{R}_a\mathscr{C}\mathscr{D}$ and, by condition (d) of modal relations (Definition \ref{modalRelation}), there is an $\mathscr{X}\subseteq\mathscr{D}$ s.t. $\mathfrak{R}_a\mathscr{B}\mathscr{X}$. Additionally, we know that in order for $\mathscr{D}^+$ to remain after the update $\Vdash_{\mathscr{D},\mathfrak{R}_A}\phi$. Since this holds for all $\mathscr{D}^+$, we can replace it with $\mathscr{Y}^+$ and conclude that for all $\mathscr{X}$ s.t. $\mathfrak{R}_a\mathscr{B\mathscr{X}}$ and $\mathscr{Y}\supseteq\mathscr{X}$, if $\Vdash_\mathscr{Y},\mathfrak{R}_A \phi$, then $\Vdash_{\mathscr{Y}^+,\mathfrak{R}'_A}\psi$ and so, by Lemma \ref{lem:EquivalenceOfEffectiveUpdates}, $\Vdash_{\mathscr{Y},\mathfrak{R}_A}^\phi \psi$.

    For the direction from 2 to 1, we know, by symmetry, that $\mathfrak{R}_A\mathscr{X}\mathscr{B}$ and so, by condition (c) of modal relations (Definition \ref{modalRelation}), for every $\mathscr{Y}\supseteq\mathscr{X}$ there has to be a $\mathscr{C}\supseteq\mathscr{B}$ s.t. $\mathfrak{R}_a \mathscr{Y}\mathscr{C}$ and, again by symmetry, $\mathfrak{R}_A\mathscr{C}\mathscr{Y}$. Let $\mathscr{D}$ be any $\mathscr{Y}$ for which $\Vdash_{\mathscr{Y},\mathfrak{R}_A}\phi$. So for any effective update $\mathfrak{R}'_A$ there are $\mathscr{C}^+$ and $\mathscr{D}^+$ that are corresponding updates of $\mathscr{C}$ and $\mathscr{D}$ and $\mathfrak{R}'_A \mathscr{C}^+\mathscr{D}^+$ and $\Vdash^\phi_{\mathscr{D}, \mathfrak{R}_A}\psi$.
\end{proof}

\begin{figure}
\begin{center}
\begin{tabular}{c p{.7cm} c}
     
\begin{tikzpicture}[-,shorten >=2pt,node distance =2cm, thick, baseline={([yshift=-1.8ex]current bounding box.center)}]

\node (B) [above of=A] {$\mathscr{C}$};
\node(D) [above of=B] {$\mathscr{D}$};
\node[label=right:{$\phi$}] (E) [right of=D] {$\mathscr{E}$};

\draw[dotted] (B) to node {} (D);
\draw[-,above] (D) to node {$\mathfrak{R}'$} (E);

\end{tikzpicture}

&

$\Leftrightarrow$

&

\begin{tikzpicture}[-,shorten >=2pt,node distance =2cm, thick, baseline={([yshift=-1.8ex]current bounding box.center)}]

\node (B) [above of=A] {$\mathscr{C}$};
\node (C) [right of=B] {$\mathscr{X}$};
\node[label=right:{If $\phi$, then $\psi$}] (E) [right of=D] {$\mathscr{Y}$};

\draw[dotted] (C) to node {} (E);
\draw[-,above] (B) to node {$\mathfrak{R}$} (C);

\end{tikzpicture}

\end{tabular}
\caption{\label{Fig:AnnK} Illustration of steps two (left) and three (right) of the proof of Announcement and Knowledge. Note that $\mathfrak{R}'\mathscr{D}\mathscr{E}$ only if $\Vdash_{\mathscr{E}, \mathfrak{R}'}\phi$.}
\end{center}
\end{figure}

Given that we can now show that the new axioms and rules specific to PAL also hold for our system.

\medskip 

\begin{lemma}
\label{lem:PALAxioms}
    The axioms and rules of PAL are valid on the base-extension semantics for PAL.

 \[
    \begin{array}{l@{\quad}l}
    [\phi]p \leftrightarrow \phi\to p & \mbox{atomic permanence} \\\relax
    [\phi]\bot\leftrightarrow \phi\to\bot & \mbox{announcement and bot} \\\relax
    [\phi](\psi\to\chi)\leftrightarrow [\phi]\psi \to [\phi]\chi  & \mbox{announcement and implication} \\\relax
    [\phi]K_a \psi\leftrightarrow \phi\to K_a[\phi]\psi  & \mbox{announcement and knowledge} \\\relax   
    [\phi][\psi]\chi\leftrightarrow [\phi\wedge[\phi]\psi]\chi  & \mbox{announcement composition} \\

    \mbox{If $\psi$, then $[\phi]\psi$} & \mbox{necessitation of announcements}
    \end{array}\]

\end{lemma}

\begin{proof}
\;
\begin{center}   
$[\phi]p \leftrightarrow \phi\to p$ 
\end{center}  
 By definition of validity, $\Vdash_{\mathscr{B},\mathfrak{R}_A}[\phi]p$ iff for all $\mathscr{C}\supseteq\mathscr{B}$, if $\Vdash_{\mathscr{C},\mathfrak{R}_A} \phi$, then $\Vdash_{\mathscr{B},\mathfrak{R}_A}^{\phi} p$. Similarly, $\Vdash_{\mathscr{B},\mathfrak{R}_A}\phi \to p$ iff for all $\mathscr{C}\supseteq\mathscr{B}$, if $\Vdash_{\mathscr{C},\mathfrak{R}_A} \phi$, then $\Vdash_{\mathscr{C},\mathfrak{R}_A} p$. It suffices to point out that $\Delta$ has no bearing on the support of propositional $p$.
\begin{center}
    $[\phi]\bot \leftrightarrow \phi\to\bot$
\end{center}

The case for $\bot$ works analogously to the case for $p$. Simply note that $\bot$ holds iff all $p$ hold. So, this follows immediately from the result for atomic permanence.


\begin{center}
    $[\phi](\psi\to\chi) \leftrightarrow [\phi]\psi \to [\phi]\chi$
\end{center}

This case is done by simple unpacking of the definitions and moving the precondition $\phi$:

\[\begin{array}{l@{\quad}c@{\quad}l}
\Vdash_{\mathscr{B},\mathfrak{R}_A}[\phi](\psi\to\chi) & \mbox{iff} & \mbox{for all $\mathscr{C}\supseteq\mathscr{B}$, if $\Vdash_{\mathscr{C}, \mathfrak{R}_A} \phi$,}\\
& & \mbox{then $\Vdash_{\mathscr{C},\mathfrak{R}_A}^{\phi} \psi\to\chi$}\\

& \mbox{iff} & \mbox{for all $\mathscr{C}\supseteq\mathscr{B}$, if $\Vdash_{\mathscr{C}, \mathfrak{R}_A} \phi$,}\\
& & \mbox{then for all $\mathscr{D}\supseteq\mathscr{C}$, if $\Vdash_{\mathscr{D},\mathfrak{R}_A}^{\phi} \psi$,}\\

& & \mbox{then $\Vdash_{\mathscr{D},\mathfrak{R}_A}^{\phi} \chi$}\\

& \mbox{iff} & \mbox{for all $\mathscr{C}\supseteq\mathscr{B}$, if $\Vdash_{\mathscr{C},\mathfrak{R}_A}\phi$ implies that for all $\mathscr{D}\supseteq\mathscr{C}$,}\\

& & \mbox{$\Vdash_{\mathscr{D},\mathfrak{R}_A}^{\phi}\psi$, then $\Vdash_{\mathscr{C},\mathfrak{R}_A}\phi$ implies that for all $\mathscr{E}\supseteq\mathscr{C}$,}\\

& & \mbox{$\Vdash_{\mathscr{E},\mathfrak{R}_A}^{\phi}\chi$}\\

\Vdash_{\mathscr{B},\mathfrak{R}_A}[\phi]\psi\to[\phi]\chi & \mbox{iff} & \mbox{for all $\mathscr{C}\supseteq\mathscr{B}$, if $\Vdash_{\mathscr{C},\mathfrak{R}_A}[\phi]\psi$,}\\
& & \mbox{then $\Vdash_{\mathscr{C},\mathfrak{R}_A}[\phi]\chi$}

\end{array}
\]

For the steps between the second and third line, we simply moved $\Vdash_{\mathscr{C},\mathfrak{R}_A^\mathscr{C}}\phi$ and the universal quantification over the supersets of $\mathscr{C}$ between inside or outside the implication.



\begin{center}
    $[\phi]K_a\psi \leftrightarrow \phi\to K_a [\phi]\psi$
\end{center}

This case is also proven straightforwardly by unpacking and moving around the precondition $\phi$. The only exception is the step between the second and third line which holds by Lemma \ref{lem:annK}.

\[\begin{array}{l@{\quad}c@{\quad}l}
\Vdash_{\mathscr{B},\mathfrak{R}_A}[\phi]K_a\psi & \mbox{iff} & \mbox{for all $\mathscr{C}\supseteq\mathscr{B}$, if $\Vdash_{\mathscr{C}, \mathfrak{R}_A} \phi$,}\\
& & \mbox{then $\Vdash_{\mathscr{C},\mathfrak{R}_A}^{\phi} K_a\psi$}\\

& \mbox{iff} & \mbox{for all $\mathscr{C}\supseteq\mathscr{B}$, if $\Vdash_{\mathscr{C}, \mathfrak{R}_A} \phi$,}\\
& & \mbox{then for all $\mathscr{D\supseteq\mathscr{C}}$, $\mathscr{E}$, $\mathfrak{R}_A'\in \mathfrak{R}_A|\phi^\mathscr{D}$,}\\

& & \mbox{and updates $\mathscr{D}^+, \mathscr{E}^+$ of $\mathscr{D},\mathscr{E}$ for $\mathfrak{R}_A'$ s.t. $\mathfrak{R}'_a\mathscr{D}^+\mathscr{E}^+$,}\\

& & \mbox{$\Vdash_{\mathscr{E},\mathfrak{R}_A}^{\phi}\psi$}\\


& \mbox{iff} & \mbox{for all $\mathscr{C}\supseteq\mathscr{B}$, if $\Vdash_{\mathscr{C}, \mathfrak{R}_A} \phi$,}\\
& & \mbox{then for all $\mathscr{X}$ s.t. $\mathfrak{R}_a\mathscr{C}\mathscr{X}$ and $\mathscr{Y}\supseteq\mathscr{X}$, if $\Vdash_{\mathscr{Y},\mathfrak{R}_A}\phi$,}\\
& & \mbox{then $\Vdash_{\mathscr{Y},\mathfrak{R}_A}^{\phi}\psi$}\\

& \mbox{iff} & \mbox{for all $\mathscr{C}\supseteq\mathscr{B}$, if $\Vdash_{\mathscr{C}, \mathfrak{R}_A} \phi$,}\\
& & \mbox{then for all $\mathscr{X}$ s.t. $\mathfrak{R}_a\mathscr{C}\mathscr{X}$, $\Vdash_{\mathscr{X},\mathfrak{R}_A}[\phi]\psi$}\\

\Vdash_{\mathscr{B},\mathfrak{R}_A}\phi \to K_a [\phi]\psi & \mbox{iff} & \mbox{for all $\mathscr{C}\supseteq\mathscr{B}$, if $\Vdash_{\mathscr{C}, \mathfrak{R}_A} \phi$,}\\

& & \mbox{then $\Vdash_{\mathscr{C},\mathfrak{R}_A}K_a [\phi]\psi$}\\

\end{array}\]

\begin{center}
    $[\phi][\psi]\chi \leftrightarrow [\phi\wedge [\phi]\psi]\chi$
\end{center}

 This case, again, is proven by unpacking and because $\Vdash_{\mathscr{D},\mathfrak{R}_A}^{\phi\wedge[\phi]\psi}\chi$ iff $\Vdash_{\mathscr{D},\mathfrak{R}_A}^{\phi;\psi}\chi$ by Lemmas \ref{CompR} and \ref{lem:EquivalenceOfEffectiveUpdates}:

\[\begin{array}{l@{\quad}c@{\quad}l}
\Vdash_{\mathscr{B},\mathfrak{R}_A}[\phi][\psi]\chi & \mbox{iff} & \mbox{for all $\mathscr{C}\supseteq\mathscr{B}$, if $\Vdash_{\mathscr{C}, \mathfrak{R}_A} \phi$,}\\
& & \mbox{then $\Vdash_{\mathscr{C},\mathfrak{R}_A}^{\phi} [\psi]\chi$}\\

& \mbox{iff} & \mbox{for all $\mathscr{C}\supseteq\mathscr{B}$, if $\Vdash_{\mathscr{C},\mathfrak{R}_A}^{\phi} \phi$, then for all $\mathscr{D}\supseteq\mathscr{C}$,}\\
& & \mbox{ if $\Vdash_{\mathscr{D},\mathfrak{R}_A}^{\phi} \psi$, then $\Vdash_{\mathscr{D},\mathfrak{R}_A}^{\phi,\psi} \chi$}\\

& \mbox{iff} & \mbox{for all $\mathscr{C}\supseteq\mathscr{B}$, if $\Vdash_{\mathscr{C}, \mathfrak{R}_A} \phi$, then for all $\mathscr{D}\supseteq\mathscr{C}$,}\\
& & \mbox{if $\Vdash_{\mathscr{D},\mathfrak{R}_A}^{\phi} \psi$, then $\Vdash_{\mathscr{D},\mathfrak{R}_A}^{\phi\wedge [\phi]\psi} \chi$}\\

& \mbox{iff} & \mbox{for all $\mathscr{C}\supseteq\mathscr{B}$, if $\Vdash_{\mathscr{C}, \mathfrak{R}_A} \phi$ and for all $\mathscr{D}\supseteq\mathscr{C}$,}\\
& & \mbox{$\Vdash_{\mathscr{D}, \mathfrak{R}_A} \phi$ implies $\Vdash_{\mathscr{D},\mathfrak{R}_A}^{\phi} \psi$, then $\Vdash_{\mathscr{C},\mathfrak{R}_A}^{\phi\wedge [\phi]\psi} \chi$}\\

\Vdash_{\mathscr{B},\mathfrak{R}_A}[\phi\wedge[\phi]\psi]\chi & \mbox{iff} & \mbox{for all $\mathscr{C}\supseteq\mathscr{B}$, if $\Vdash_{\mathscr{C}, \mathfrak{R}_A} \phi\wedge [\phi]\psi$,}\\
& & \mbox{then $\Vdash_{\mathscr{C},\mathfrak{R}_A}^{\phi\wedge [\phi]\psi} \chi$}\\

\end{array}\]


\begin{center}
    If $\psi$, then $[\phi]\psi$
\end{center}

We show this by contrapositive: In order for $\nVdash_{\mathscr{B}, \mathfrak{R}_A} [\phi]\psi$, there has to be a $\mathscr{C}\supseteq\mathscr{B}$ with $\Vdash_{\mathscr{C}, \mathfrak{R}_A} \phi$ and $\nVdash_{\mathscr{C},\mathfrak{R}_A}^{\phi} \psi$. By Lemma \ref{lem: PALMaxCon}, there is a maximally-consistent $\mathscr{C}^*\supseteq\mathscr{C}$ s.t. $\nVdash_{\mathscr{C}^*, \mathfrak{R}_A|\phi^{\mathscr{C}^*}} \psi$. Take some relation $\mathfrak{R}'_A \in \mathfrak{R}_A^{\mathscr{C}^*}|\phi$. By Lemma \ref{PALrelationIsModal}, we know that $\mathfrak{R}'_a$ is a $S5$-modal relation. From Lemma \ref{PALBehaviour}, it follows straightforwardly that for all formulae $\chi$, $\Vdash_{\mathscr{C}^*,\mathfrak{R}_A}^{\phi} \chi$ iff $\Vdash_{\mathscr{C}^{*+}, \mathfrak{R}'_A} \chi$ for all updates $\mathscr{C}^{*+}$ of $\mathscr{C}^*$ for $\mathfrak{R}_a'$. So we have $\nVdash_{\mathscr{C}^*, \mathfrak{R}_A'} \psi$, which contradicts our assumption that $\psi$ is valid.
%
%
\end{proof}

We can now proof soundness of our system. 

\medskip 

\PALsound*

\begin{proof}
    This follows directly from Lemmas \ref{lem:ClassicalAxioms}, \ref{lem:S5Axioms} and \ref{lem:PALAxioms}.
\end{proof}

\subsection*{Completeness}
\medskip    

\begin{definition}
    \label{PALtranslation}
    The translation $t: PAL \to S5$ is defined as follows: 

\begin{center}$
\begin{array}{l@{\quad}c@{\quad}l}
t(p) & = & p\\

t(\bot) & = & \bot \\


t(\phi\to\psi) & = & t(\phi) \to t(\psi)\\

t(K_a \phi) & = & K_a t(\phi)\\

t([\phi] p) & = & t(\phi\to p)\\

t([\phi] \bot) & = & t(\phi\to \bot)\\

t([\phi] (\psi\to\chi)) & = & t([\phi]\psi \to [\phi]\chi)\\

t([\phi] K_a\psi) & = & t(\phi\to K_a[\phi]\psi)\\

t([\phi][\psi]\chi ) & = & t([\phi\wedge [\phi]\psi] \chi)\\

\end{array}
$\end{center}\vspace{-7mm}
\fillBox     
\end{definition} \medskip

This translation is used for an inductive proof to show that for every formula in PAL there is an equivalent formula in $S5$. To do this we need a complexity measure to show that such an induction concludes.

\medskip 

\begin{definition}
    \label{TranslationComplexity}
The complexity $c : PAL \to \mathbb{N}$ is defined as follows: 
    \begin{center}$
    \begin{array}{l@{\quad}c@{\quad}l}
    
    c(p) & = & 1\\
    c(\bot) & = & 1\\
    c(\phi\to\psi) & = & 1+ max(c(\phi),c(\psi))\\
    c(K_a\phi) & = & 1+ c(\phi)\\
    c([\phi]\psi) & = & (2+ c(\phi)) \cdot c(\psi)\\

    \end{array}
    $\end{center}\vspace{-7mm}
    \fillBox 
\end{definition} \medskip

Before we go to the necessary complexity results, here are some helpful preliminary results based on the definitions of $\neg$ and $\wedge$ from our primary connectives.

\medskip 

\begin{lemma}
    \label{prelimComplexivity}
The following equations hold: \smallskip

\begin{enumerate}
    \item[—]$c(\neg \phi) = 1+ c(\phi)$
    \item[—]$c(\phi\wedge\psi) = 3+ max(1+c(\phi); c(\psi)$
\end{enumerate}

\end{lemma}

\begin{proof}
    For the first equation,$\neg\phi$ is shorthand for $\phi\to\bot$. We have $c(\phi\to\bot) = 1+ max(c(\phi);c(\bot))$ and $c(\bot) = 1$. Depending on $\phi$, $c(\phi)$ is at least $1$ and so $max$ never has to pick $c(\bot)$ and so $c(\neg\phi) = 1 + c(\phi)$.

    For the second equation. $\phi\wedge\psi$ is $\neg(\neg\neg\phi\to\neg\psi)$ (for simplicity we use $\neg$ instead of $\to\bot$ here). By the above, $c(\neg(\neg\neg\phi\to\neg\psi)= 1 + c(\neg\neg\phi\to\neg\psi) = 1+1+ max(c(\neg\neg\phi);c(\neg\psi)) = 2 + max(2+c(\phi);1+c(\psi)) = 3+ max(1+c(\phi);c(\psi))$. 
\end{proof}

Given this measure the following hold:

\medskip 

\begin{lemma}
    \label{LowerComplexity}

The following hold: \smallskip

    \begin{enumerate}
        \item[—]$c(\psi)\geq c(\phi)$ if $\phi\in Sub(\psi)$
        \item[—]$c([\phi] p) > c(\phi\to p)$
        \item[—]$c([\phi]\bot) > c(\phi\to\bot)$
        \item[—]$c([\phi](\psi\to\chi)) > c([\phi]\psi\to[\phi]\chi)$
        \item[—]$c([\phi]K_a\psi) > c(\phi\to K_a [\phi]\psi)$
        \item[—]$c([\phi][\psi]\chi) > c([\phi\wedge[\phi]\psi]\chi)$.
    \end{enumerate}
    
\end{lemma}

\begin{proof}

The first result follows immediately from the definition of our complexity measure in Definition \ref{TranslationComplexity}.

The second and third result follow in the same way as $c(p)=c(\bot)=1$. So, we only show it for $\bot$. $c(\phi\to \bot) = 1+ c(\phi)$ as in the proof of Lemma \ref{prelimComplexivity}. $c([\phi]\bot) = (2+c(\phi))\times 1$. Obviously, $2+c(\phi) > 1+c(\phi)$.

For the fourth result, $c([\phi]\psi\to\chi) = (2+c(\phi)) \times (1+max(c(\psi);c(\chi)))$ and $c([\phi]\psi \to [\phi]\chi) = 1+ max(c([\phi]\psi;[\phi]\chi)) = 1+ max(((2+c(\phi))\times c(\psi)); ((2+c(\phi))\times c(\chi))) = 1+ ((2+c(\phi)) \times max(c(\psi);c(\chi)))$. The result $(2+c(\phi)) \times (1+max(c(\psi);c(\chi)))\geq 1+ ((2+c(\phi)) \times max(c(\psi);c(\chi)))$ follows because $(2+c(\phi)$ is greater than $1$.

For the fifth result, $c([\phi]K_a\psi) = (2+c(\phi))\times (1+c(\psi))$ and $c(\phi\to K_a[\phi]\psi) = 1+ max(c(\phi); 1+((2+c(\phi))\times c(\psi)))$.  Obviously $c(\phi)$ cannot be greater than $1+((2+c(\phi))\times c(\psi))$ and so, we simplify to $2+((2+c(\phi))\times c(\psi))$. To show the result it suffices to point out that $2$ is less than $2+c(\phi)$.

For the final result of this lemma, $c([\phi][\psi]\chi)= (2+c(\phi)) \times ((2+c(\psi)) \times c(\chi)$ and $c([\phi\wedge[\phi]\psi]\chi) = (2 + c(\phi\wedge[\phi]\psi)) \times c(\chi) = (2+3+max(1+c(\phi); c([\phi]\psi)))\times c(\chi)$. Note that $1+c(\phi)$ is never greater than $c([\phi]\psi)$ and so, we can simplify again to $(5+c([\phi]\psi))\times c(\chi) = (5+ ((2+c(\phi))\times c(\psi)) \times c(\chi)$. If we compare $(2+c(\phi)) \times ((2+c(\psi)) \times c(\chi)$ and $(5+ ((2+c(\phi))\times c(\psi)) \times c(\chi)$, the former is greater if and only if $(2\times (2+c(\phi) \times c(\chi))$ is greater than $5$. This obviously holds.
\end{proof}

This allows us to establish the following result: 

\medskip 

\begin{lemma}
    \label{translationEquivalence}
    For all formulae $\phi$, $\Vdash \phi \leftrightarrow t(\phi)$.
    
\end{lemma}

\begin{proof}
    As mentioned above, we establish this by induction on the complexity of $\phi$.

    For the base case note that $t(p)=p$. It is trivial to show that $\Vdash p\leftrightarrow p$.

    All other cases follow from the induction hypothesis and the corresponding case of Definition \ref {TranslationComplexity} and Lemma \ref{LowerComplexity}.
\end{proof}

From this and the completeness of $S5$ the following theorem follows immediately.

\medskip 

\PALcomp*

\begin{proof}
    Assume $\Vdash \phi$. By Theorem \ref{theorem:Soundness} and Lemma \ref{translationEquivalence}, we get $\vdash \phi \leftrightarrow t(\phi)$. Since $t(\phi)$ does not contain any announcement formulae and $S5$ is complete (see Theorem \ref{def:S5completeness}), we have $\vdash^{S5} t(\phi)$. Since PAL is an extension of S5, we also have $\vdash t(\phi)$ and so $\vdash \phi$.
\end{proof}
   
\end{document}